\newif\ifsmallpaper
\newif\ifprelimversion

 \smallpaperfalse

 \prelimversiontrue

\ifsmallpaper
\documentclass[10pt, a5paper, reqno, twoside]{scrbook} 
\else
\documentclass[11pt, a4paper, reqno, twoside]{scrartcl} 
\fi

\usepackage{varioref}

\usepackage{graphicx, amsmath, amssymb,amsthm}

\usepackage{wrapfig} 

\usepackage{enumitem}
\setenumerate{noitemsep, leftmargin=12mm, topsep=2mm, 
  label=\textnormal{(\roman*)}}

\usepackage{ifthen}
\usepackage[latin1,utf8]{inputenc}
\usepackage[ngerman,USenglish]{babel}

\def\hyph{-\penalty0\hskip0pt\relax} 

\usepackage{fancyhdr}

\ifsmallpaper

\usepackage[a5paper, hmargin={17mm, 17mm}, vmargin={20mm, 20mm}, headsep=5mm, headheight=5mm, footskip=30pt, bindingoffset=5mm]{geometry}
\else

\usepackage[a4paper, hmargin={25mm, 25mm}, vmargin={25mm,
  25mm},headsep=10mm,headheight=5mm,footskip=30pt]{geometry}
 
\fi

\usepackage{algorithm}
\usepackage{algorithmic}
\usepackage{mathdots}
\usepackage{hyperref}
\usepackage{color}

\definecolor{darkred}{rgb}{.5,0,0}
\definecolor{darkgreen}{rgb}{0,.4,.2}
\definecolor{darkblue}{rgb}{.1,.2,.6}

\hypersetup{
  pdftitle={Convex Optimization without Projection Steps},
  pdfauthor={Martin Jaggi},
  pdfsubject={Chapter out of PhD Thesis, 2011, ETH Z\374rich},
  pdfkeywords={},
  colorlinks=true,
  linkcolor=darkblue,
  citecolor=darkgreen,
  filecolor=cyan,
  urlcolor=cyan
}

\addtokomafont{caption}{\small}
\addtokomafont{captionlabel}{\sffamily\bfseries}

\newtheoremstyle{style}
 {\topsep}              
 {\topsep}              
 {\itshape}              
 {}                         
 {\sffamily\bfseries}  
 {.}	                     
 { }                         
 {}                          
\theoremstyle{style}

\floatname{algorithm}{\sffamily\bfseries Algorithm} 

\newtheorem{definition}{Definition}
\newtheorem{theorem}[definition]{Theorem}

\newtheorem{lemma}[definition]{Lemma}
\newtheorem{corollary}[definition]{Corollary}
\newtheorem{observation}[definition]{Observation}

\DeclareMathOperator{\st}{\textrm{s.t.}}
\DeclareMathOperator*{\conv}{conv}
\DeclareMathOperator{\diam}{diam}

\DeclareMathOperator*{\mini}{\rm minimize}
\DeclareMathOperator*{\maxi}{\rm maximize}
\DeclareMathOperator*{\argmin}{\arg\min}
\DeclareMathOperator*{\argmax}{\arg\max}
\DeclareMathOperator{\lmax}{\lambda_{\max}}
\DeclareMathOperator{\lmin}{\lambda_{\min}}

\DeclareMathOperator*{\diag}{diag}
\DeclareMathOperator*{\sign}{sign}
\DeclareMathOperator*{\tr}{Tr}
\DeclareMathOperator*{\rk}{Rk}
\DeclareMathOperator*{\card}{card}

\DeclareMathOperator{\approxLin}{\textsc{ApproxLinear}}
\DeclareMathOperator{\exactLin}{\textsc{ExactLinear}}
\DeclareMathOperator{\randomLin}{\textsc{RandomLinear}}
\providecommand{\abs}[1]{\left\lvert#1\right\rvert}
\providecommand{\norm}[1]{\left\lVert#1\right\rVert}
\providecommand{\frobnorm}[1]{\norm{#1}_{Fro}}
\providecommand{\nucnorm}[1]{\norm{#1}_{*}}
\providecommand{\spectnorm}[1]{\norm{#1}_{spec}}
\providecommand{\maxnorm}[1]{\norm{#1}_{\max}}

\newcommand{\R}{\mathbb{R}}
\newcommand{\X}{\mathcal{X}}
\newcommand{\Sym}{\mathbb{S}}
\newcommand{\Spectahedron}{\mathcal{S}}
\newcommand{\MaxCutPolytope}{\boxplus}
\newcommand{\FourPointPSD}{\Spectahedron^4_{\text{sparse}}}
\newcommand{\FourPointPSDplus}{\Spectahedron^{4+}_{\text{sparse}}}
\newcommand{\FourPointPSDminus}{\Spectahedron^{4-}_{\text{sparse}}}
\newcommand{\LOneBall}{\diamondsuit}
\newcommand{\signVec}{\mathbf{s}}
\newcommand{\N}{\mathbb{N}}
\newcommand{\id}{\mathbf{I}} 
\newcommand{\ind}{\mathbf{1}} 
\newcommand{\0}{\mathbf{0}} 
\newcommand{\unit}{\mathbf{e}} 
\newcommand{\zero}{\mathbf{0}}
\newcommand\SetOf[2]{\left\{#1\,\vphantom{#2}\right.\left|\vphantom{#1}\,#2\right\}}

\begin{document}
\pagestyle{fancy}

\ifprelimversion

\title{Convex Optimization without Projection Steps\\
{\large with Applications to Sparse and Low Rank Approximation}}
\author{Martin Jaggi\\
{\small ETH Z\"{u}rich, Switzerland}\\
{\small \href{mailto:jaggi@inf.ethz.ch}{jaggi@inf.ethz.ch}}}
\date{}

\else

\fi

\maketitle

\paragraph{Abstract.}
We study the general problem of minimizing a convex function over a compact convex domain. We will investigate a simple iterative approximation algorithm based on the method by Frank \& Wolfe~\cite{Frank:1956vp}, that does not need projection steps in order to stay inside the optimization domain.
Instead of a projection step, the linearized problem defined by a current subgradient is solved, which gives a step direction that will naturally stay in the domain.
Our framework generalizes the sparse greedy algorithm of~\cite{Frank:1956vp} and its primal-dual analysis by~\cite{Clarkson:2010hv} (and the low-rank SDP approach by~\cite{Hazan:2008kz}) to arbitrary convex domains. Analogously, we give a convergence proof guaranteeing $\varepsilon$-small duality gap after $O(\frac1\varepsilon)$ iterations.

The method allows us to understand the sparsity of approximate solutions for any $\ell_1$-regularized convex optimization problem (and for optimization over the simplex), expressed as a function of the approximation quality. We obtain matching upper and lower bounds of $\Theta(\frac1\varepsilon)$ for the sparsity for $\ell_1$-problems. The same bounds apply to low-rank semidefinite optimization with bounded trace, showing that rank $O(\frac1\varepsilon)$ is best possible here as well.
As another application, we obtain sparse matrices of $O(\frac1\varepsilon)$ non-zero entries as $\varepsilon$-approximate solutions when optimizing any convex function over a class of diagonally dominant symmetric matrices.

We show that our proposed first-order method also applies to nuclear norm and max-norm matrix optimization problems. 
For nuclear norm regularized optimization, such as matrix completion and low-rank recovery, we demonstrate the practical efficiency and scalability of our algorithm for large matrix problems, as e.g. the Netflix dataset. 
For general convex optimization over bounded matrix max-norm, our algorithm is the first with a convergence guarantee, to the best of our knowledge.\\

{\small(This article consists of the first two chapters of the author's PhD thesis~\cite{Jaggi:2011ux}.)}

\newpage
\begin{small}
\tableofcontents
\end{small}
\newpage


\section{Introduction}

\paragraph{Motivation.}
For the performance of large scale approximation algorithms for convex optimization, the trade-off between the number of iterations on one hand, and the computational cost per iteration on the other hand, is of crucial importance. The lower complexity per iteration is among the main reasons why first-order methods (i.e., methods using only information from the first derivative of the objective function), as for example stochastic gradient descent, are currently used much more widely and successfully in many machine learning applications --- despite the fact that they often need a larger number of iterations than for example second-order methods.

Classical gradient descent optimization techniques usually require a projection step in each iteration, in order to get back to the feasible region. For a variety of applications, this is a non-trivial and costly step.
One prominent example is semidefinite optimization, where the projection of an arbitrary symmetric matrix back to the PSD matrices requires the computation of a complete eigenvalue-decomposition.%

Here we study a simple first-order approach that does not need any projection steps, and is applicable to any convex optimization problem over a compact convex domain. The algorithm is a generalization of an existing method originally proposed by Frank \& Wolfe~\cite{Frank:1956vp}, which was recently extended and analyzed in the seminal paper of Clarkson~\cite{Clarkson:2010hv} for optimization over the unit simplex.

Instead of a projection, the primitive operation of the optimizer here is to minimize a \emph{linear} approximation to the function over the same (compact) optimization domain. Any (approximate) minimizer of this simpler linearized problem is then chosen as the next step-direction. Because all such candidates are always feasible for the original problem, the algorithm will automatically stay in our convex feasible region. The analysis will show that the number of steps needed is roughly identical to classical gradient descent schemes, meaning that $O\left(\frac1\varepsilon\right)$ steps suffice in order to obtain an approximation quality of $\varepsilon>0$.

The main question about the efficiency per iteration of our algorithm, compared to a classical gradient descent step, can not be answered generally in favor of one or the other. Whether a projection or a linearized problem is computationally cheaper will crucially depend on the shape and the representation of the feasible region.
Interestingly, if we consider convex optimization over the Euclidean $\norm{.}_2$-ball, the two approaches fully coincide, i.e., we exactly recover classical gradient descent.
However there are several classes of optimization problems where the linearization approach we present here is definitely very attractive, and leads to faster and simpler algorithms. This includes for example $\ell_1$-regularized problems, which we discuss in Sections~\ref{sec:vecSimplex} and~\ref{sec:vecL1}, as well as semidefinite optimization under bounded trace, as studied by~\cite{Hazan:2008kz}, see Section~\ref{sec:algTrace}.

\paragraph{Sparsity and Low-Rank.}
For these mentioned specific classes of convex optimization problems, we will additionally demonstrate that our studied algorithm leads to (optimally) sparse or low-rank solutions. This property is a crucial side-effect that can usually not be achieved by classical optimization techniques, and corresponds to the \emph{coreset} concept known from computational geometry, see also~\cite{GartnerJaggi:2009}.
More precisely, we show matching upper and lower bounds of $\Theta\left(\frac1\varepsilon\right)$ for the sparsity of solutions to general $\ell_1$-regularized problems, and also for optimizing over the simplex, if the required approximation quality is $\varepsilon$.
For matrix optimization, an analogous statement will hold for the rank in case of nuclear norm regularized problems.

\paragraph{Applications.}
Applications of the first mentioned class of $\ell_1$-regularized problems do include many machine learning algorithms ranging from support vector machines (SVMs) to boosting and multiple kernel learning, as well as $\ell_2$-support vector regression (SVR), mean-variance analysis in portfolio selection~\cite{Markowitz:1952tg}, the smallest enclosing ball problem~\cite{Badoiu:2007bj}, $\ell_{1}$-regularized least squares (also known as basis pursuit de-noising in compressed sensing), the \emph{Lasso}~\cite{Tibshirani:1996wb}, and $\ell_{1}$-regularized logistic regression~\cite{Koh:2007wo} as well as walking of artificial dogs over rough terrain~\cite{Kalakrishnan:2010fr}.

The second mentioned class of matrix problems, that is, optimizing over semidefinite matrices with bounded trace, has applications in low-rank recovery~\cite{Fazel:2001vw,Candes:2009kj,Candes:2010jb}, dimensionality reduction, matrix factorization and completion problems, as well as general semidefinite programs (SDPs).

Further applications to nuclear norm and max-norm optimization, such as sparse/robust PCA will be discussed in Section~\ref{chap:optNucMax}.

\paragraph{History and Related Work.}
The class of first-order optimization methods in the spirit of Frank and Wolfe~\cite{Frank:1956vp} has a rich history in the literature. %
Although the focus of the original paper was on quadratic programming, its last section~\cite[Section 6]{Frank:1956vp} already introduces the general algorithm for minimizing convex functions using the above mentioned linearization idea, when the optimization domain is given by linear inequality constraints. In this case, each intermediate step consists of solving a linear program. The given convergence guarantee bounds the primal error, and assumes that all internal problems are solved exactly.

Later~\cite{Dunn:1978di} has generalized the same method to arbitrary convex domains, and improved the analysis to also work when the internal subproblems are only solved approximately, see also~\cite{Dunn:1980cw}. Patriksson in~\cite{Patriksson:1993bl,Patriksson:1998eh} then revisited the general optimization paradigm, investigated several interesting classes of convex domains, and coined the term ``cost approximation'' for this type of algorithms. More recently, \cite{TongZhang:2003da} considered optimization over convex hulls, and studies the crucial concept of sparsity of the resulting approximate solutions. However, this proposed algorithm does not use linear subproblems. 

The most recent work of Clarkson~\cite{Clarkson:2010hv} provides a good overview of the existing lines of research, and investigates the sparsity solutions when the optimization domain is the unit simplex, and establishing the connection to coreset methods from computational geometry. Furthermore, \cite{Clarkson:2010hv} was the first to introduce the stronger notion of convergence in primal-dual error for this class of problems, and relating this notion of duality gap to Wolfe duality.

\paragraph{Our Contributions.}
The character of this article mostly lies in reviewing, re-interpreting and generalizing the existing approach given by~\cite{Clarkson:2010hv},~\cite{Hazan:2008kz} and the earlier papers by~\cite{TongZhang:2003da,Dunn:1978di,Frank:1956vp}, who do deserve credit for the analysis techniques.
Our contribution here is to transfer these methods to the more general case of convex optimization over arbitrary bounded convex subsets of a vector space, while providing stronger primal-dual convergence guarantees.
To do so, we propose a very simple alternative concept of optimization duality, which will allow us to generalize the stronger primal-dual convergence analysis which~\cite{Clarkson:2010hv} has provided for the the simplex case, to optimization over arbitrary convex domains. So far, no such guarantees on the duality gap were known in the literature for the Frank-Wolfe-type algorithms~\cite{Frank:1956vp}, except when optimizing over the simplex. Furthermore, we generalize Clarkson's analysis~\cite{Clarkson:2010hv} to work when only approximate linear internal optimizers are used, and to arbitrary starting points. Also, we study the sparsity of solutions in more detail, obtaining upper and lower bounds for the sparsity of approximate solutions for a wider class of domains.

Our proposed notion of duality gives simple certificates for the current approximation quality, which can be used for any optimization algorithm for convex optimization problems over bounded domain, even in the case of non-differentiable objective functions.

We demonstrate the broad applicability of our general technique to several important classes of optimization problems, such as $\ell_1$- and $\ell_\infty$-regularized problems, as well as semidefinite optimization with uniformly bounded diagonal, and sparse semidefinite optimization.

Later in Section~\ref{chap:optNucMax} we will give a simple transformation in order to apply the first-order optimization techniques we review here also to nuclear norm and max-norm matrix optimization problems.

\paragraph{Acknowledgments.}
Credit for the important geometric interpretation of the duality gap over the spectahedron as the distance to the linearization goes to Soeren Laue.
Furthermore, the author would like to thank Marek Sulovsk{\'y}, Bernd G{\"a}rtner, Arkadi Nemirovski, Elad Hazan, Joachim Giesen, Sebastian Stich, Michel Baes, Michael B{\"u}rgisser and Christian Lorenz M{\"u}ller for helpful discussions and comments.

\section{The Poor Man's Approach to Convex Optimization and Duality}\label{sec:poorDual}

\paragraph{The Idea of a Duality given by Supporting Hyperplanes.}
Suppose we are given the task of minimizing a convex function $f$ over a bounded convex set $D \subset \R^n$, and let us assume for the moment that $f$ is continuously differentiable.

Then for any point $x\in D$, it seems natural to consider the tangential ``supporting'' hyperplane to the graph of the function $f$ at the point $(x,f(x))$. Since the function $f$ is convex, any such linear approximation must lie below the graph of the function.

Using this linear approximation for each point $x\in D$, we define a \emph{dual} function value $\omega(x)$ as the minimum of the linear approximation to $f$ at point $x$, where the minimum is taken over the domain $D$. We note that the point attaining this linear minimum also seems to be good direction of improvement for our original minimization problem given by $f$, as seen from the current point $x$. This idea will lead to the optimization algorithm that we will discuss below.

As the entire graph of $f$ lies above any such linear approximation, it is easy to see that $\omega(x) \le f(y)$ holds for each pair $x,y \in D$. This fact is called \emph{weak duality} in the optimization literature.

This rather simple definition already completes the duality concept that we will need in this paper. We will provide a slightly more formal and concise definition in the next subsection, which is useful also for the case of non-differentiable convex functions.
The reason we call this concept a \emph{poor man's} duality is that we think it is considerably more direct and intuitive for the setting here, when compared to classical Lagrange duality or Wolfe duality, see e.g.~\cite{Boyd:2004uz}.

\subsection{Subgradients of a Convex Function}
In the following, we will work over a general vector space $\X$ equipped with an inner product $\langle .,.\rangle$.
As the most prominent example in our investigations, the reader might always think of the case $\X=\R^n$ with $\langle x,y\rangle = x^Ty$ being the standard Euclidean scalar product.

We consider general convex optimization problems given by a convex function $f: \X \rightarrow \R$ over a compact\footnote{%
Here we call a set $D\subseteq X$ \emph{compact} if it is closed and bounded. See~\cite{Kurdila:2005ur} for more details.} %
convex domain $D \subseteq \X$, or formally
\begin{equation}\label{eq:optGenConvex}
   \mini_{x \in D} \, f(x)  \ .
\end{equation}

In order to develop both our algorithm and the notion of duality for such convex optimization problems in the following, we need to formally define the supporting hyperplanes at a given point $x\in D$. These planes coincide exactly with the well-studied concept of \emph{subgradients} of a convex function.

For each point $x\in D$, the \emph{subdifferential} at $x$ is defined as the set of normal vectors of the affine linear functions through $(x,f(x))$ that lie below the function $f$. Formally
\begin{equation}\label{eq:subDiff}
\partial f(x) := \SetOf{d_x\in\X}{f(y) \ge f(x)+ \langle y-x ,d_x\rangle ~~~\forall y \in D} \ .\!\!
\end{equation}

Any element $d_x\in\partial f(x)$ is called a \emph{subgradient} to $f$ at $x$. Note that for each $x$, $\partial f(x)$ is a closed convex set. Furthermore,
if $f$ is differentiable, then the subdifferential consists of exactly one element for each $x\in D$, namely $\partial f(x) = \{\nabla f(x)\}$, as explained e.g. in~\cite{Nemirovski:2005wu,Kurdila:2005ur}. %

If we assume that $f$ is convex and lower semicontinuous\footnote{%
The assumption that our objective function $f$ is lower semicontinuous on $D$, is equivalent to the fact that its epigraph --- i.e. the set $\SetOf{(x,t) \in D\times \R}{t\ge f(x)}$ of all points lying on or above the graph of the function --- is closed, see also~\cite[Theorem 7.1.2]{Kurdila:2005ur}.
} %
on $D$, then it is known that $\partial f(x)$ is non-empty, meaning that there exists at least one subgradient $d_x$ for every point $x\in D$. For a more detailed investigation of subgradients, we refer the reader to one of the works of e.g.~\cite{Rockafellar:1997ww,Boyd:2004uz,Nemirovski:2005wu,Kurdila:2005ur,Borwein:2006ts}.

\subsection{A Duality for Convex Optimization over Compact Domain}
For a given point $x\in D$, and any choice of a subgradient $d_x\in \partial f(x)$, we define a \emph{dual} function value 
\begin{equation}\label{eq:dual}
\omega(x,d_x) := \min_{y\in D} \  f(x)+ \langle y-x ,d_x\rangle \ .
\end{equation}
In other words $\omega(x,d_x) \in \R$ is the minimum of the linear approximation to $f$ defined by the subgradient $d_x$ at the supporting point $x$, where the minimum is taken over the domain $D$. This minimum is always attained, since $D$ is compact, and the linear function is continuous in $y$.

By the definition of the subgradient --- as lying below the graph of the function $f$ --- we readily attain the property of weak-duality, which is at the core of the optimization approach we will study below.

\begin{lemma}[Weak duality]\label{lem:weakDual}
For all pairs $x,y\in D$, it holds that
\[
\omega(x,d_x) \le f(y)
\]
\end{lemma}
\begin{proof}
Immediately from the definition of the dual $\omega(.,.)$: \vspace{-4pt}
\[
\begin{array}{rl}
\omega(x,d_x) =& \min_{z\in D} \  f(x)+ \langle z-x ,d_x\rangle \\
\le& f(x) + \langle y-x ,d_x\rangle \\
\le& f(y) \ .
\end{array}\vspace{-5pt}
\]
Here the last inequality is by the definition~(\ref{eq:subDiff}) of a subgradient.
\end{proof}
Geometrically, this fact can be understood as that any function value $f(y)$, which is ``part of'' the graph of $f$, always lies higher than the minimum over any linear approximation (given by $d_x$) to $f$.

In the case that $f$ is differentiable, there is only one possible choice for a subgradient, namely $d_x = \nabla f(x)$, and so we will then denote the (unique) dual value for each $x$ by
\begin{equation}\label{eq:dualDiff}
\omega(x) := \omega(x,\nabla f(x)) = \min_{y\in D} \  f(x)+ \langle y-x ,\nabla f(x)\rangle \ .
\end{equation}

\paragraph{The Duality Gap as a Measure of Approximation Quality.}
The above duality concept allows us to compute a very simple measure of approximation quality, for any candidate solution $x\in D$ to problem~(\ref{eq:optGenConvex}). This measure will be easy to compute even if the true optimum value $f(x^*)$ is unknown, which will very often be the case in practical applications.
The duality gap $g(x,d_x)$ at any point $x\in D$ and any subgradient subgradient $d_x\in \partial f(x)$ is defined as
\begin{equation}\label{eq:defGap}
g(x,d_x) := f(x) - \omega(x,d_x) =\  \max_{y\in D} \, \langle x-y, d_x \rangle \ ,
\end{equation}
or in other words as the difference of the current function value $f(x)$ to the minimum value of the corresponding linearization at $x$, taken over the domain $D$. The quantity $g(x,d_x) = f(x) - \omega(x,d_x)$ will be called the \emph{duality gap} at $x$, for the chosen $d_x$.

By the weak duality Lemma~\ref{lem:weakDual}, we obtain that for any optimal solution $x^*$ to problem~(\ref{eq:optGenConvex}), it holds that
\begin{equation}\label{eq:duGapGePrGap}
g(x,d_x) \ge f(x) - f(x^*) \ge 0 ~~~~~\forall x\in D, \ \forall d_x\in\partial f(x) \ .\!\!\!
\end{equation}
Here the quantity $f(x) - f(x^*)$ is what we call the \emph{primal error} at point $x$, which is usually impossible to compute due to $x^*$ being unknown. The above inequality~(\ref{eq:duGapGePrGap}) now gives us that the duality gap --- which is easy to compute, given $d_x$ --- is always an upper bound on the primal error. 
This property makes the duality gap an extremely useful measure for example as a stopping criterion in practical optimizers or heuristics.

We call a point $x\in\X$ an \emph{$\varepsilon$-approximation} if $g(x,d_x) \le \varepsilon$ for some choice of subgradient $d_x\in\partial f(x)$.

For the special case that $f$ is differentiable, we will again use the simpler notation $g(x)$ for the (unique) duality gap for each $x$, i.e.
\[
g(x) := g(x,\nabla f(x)) = \max_{y\in D} \, \langle x-y, \nabla f(x)\rangle \ .
\]

\paragraph{Relation to Duality of Norms.}
In the special case when the optimization domain $D$ is given by the unit ball of some norm on the space $\X$, we observe the following:
\begin{observation}\label{obs:gapDualNorm}
For optimization over any domain $D=\SetOf{x\in\X}{\norm{x} \le 1}$ being  the unit ball of some norm $\norm{.}$, the duality gap for the optimization problem $\displaystyle\min_{x\in D} f(x)$ is given by
\[
g(x,d_x) = \norm{d_x}_* + \langle x, d_x \rangle \ ,
\]
where $\norm{.}_*$ is the dual norm of $\norm{.}$.
\end{observation}
\begin{proof}
Directly by the definitions of the dual norm $\norm{x}_* = \sup_{\norm{y}\le 1} \langle y, x \rangle$, and the duality gap $g(x,d_x) = \max_{y\in D} \, \langle y, -d_x\rangle + \langle x, d_x\rangle$ as in~(\ref{eq:defGap}).
\end{proof}

\section{A Projection-Free First-Order Method for Convex Optimization}\label{sec:optGenConvex}

\subsection{The Algorithm}

In the following we will generalize the sparse greedy algorithm of~\cite{Frank:1956vp} and its analysis by~\cite{Clarkson:2010hv} to convex optimization over arbitrary compact convex sets $D \subseteq \X$ of a vector space. More formally, we assume that the space $\X$ is a Hilbert space, and consider problems of the form~(\ref{eq:optGenConvex}), i.e.,
\[
\mini_{x \in D} \, f(x) \ .
\]
Here we suppose that the objective function $f$ is differentiable over the domain~$D$, and that for any $x\in D$, we are given the gradient $\nabla f(x)$ via an oracle.

The existing algorithms so far did only apply to convex optimization over the simplex (or convex hulls in some cases)~\cite{Clarkson:2010hv}, or over the spectahedron of PSD matrices~\cite{Hazan:2008kz}, or then did not provide guarantees on the duality gap.
Inspired by the work of Hazan~\cite{Hazan:2008kz}, we can also relax the requirement of \emph{exactly} solving the linearized problem in each step, to just computing \emph{approximations} thereof, while keeping the same convergence results. This allows for more efficient steps in many applications.

Also, our algorithm variant works for arbitrary starting points, without needing to compute the best initial ``starting vertex'' of $D$ as in~\cite{Clarkson:2010hv}.

\paragraph{The Primal-Dual Idea.}
We are motivated by the geometric interpretation of the ``poor man's'' duality gap, as explained in the previous Section~\ref{sec:poorDual}. This duality gap is the maximum difference of the current value $f(x)$, to the linear approximation to $f$ at the currently fixed point $x$, where the linear maximum is taken over the domain $D$. This observation leads to the algorithmic idea of directly using the current linear approximation (over the convex domain $D$) to obtain a good direction for the next step, automatically staying in the feasible region.

The general optimization method with its two precision variants is given in Algorithm~\ref{alg:greedyConvex}. For the approximate variant, the constant $C_f>0$ is an upper bound on the curvature of the objective function $f$, which we will explain below in more details.

\begin{algorithm}[h!]
  \caption{Greedy on a Convex Set}
  \label{alg:greedyConvex}
\begin{algorithmic}
  \STATE {\bfseries Input:} Convex function $f$, convex set $D$, target accuracy $\varepsilon$
  \STATE {\bfseries Output:} $\varepsilon$-approximate solution for problem~(\ref{eq:optGenConvex})
  \STATE Pick an arbitrary starting point $x^{(0)} \in D$
  \FOR{$k=0\dots\infty$}
  \STATE Let $\alpha := \frac2{k+2}$
  \STATE \quad Compute $s := \exactLin\left(\nabla f(x^{(k)}),D\right)$
  \hspace{0.4cm}{\small\COMMENT{Solve the linearized primitive task exactly}}
  \STATE ---or---
  \STATE \quad Compute $s := \approxLin\left(\nabla f(x^{(k)}),D,\alpha C_f\right)$
  \hspace{0.22cm}{\small\COMMENT{Approximate the linearized problem}}
  \STATE Update $x^{(k+1)}:=x^{(k)}+\alpha(s - x^{(k)})$
  \ENDFOR
\end{algorithmic}
\end{algorithm}

\begin{figure}[h!]
\vspace{-0.1em}
\begin{center}
\includegraphics[width=0.5\columnwidth]{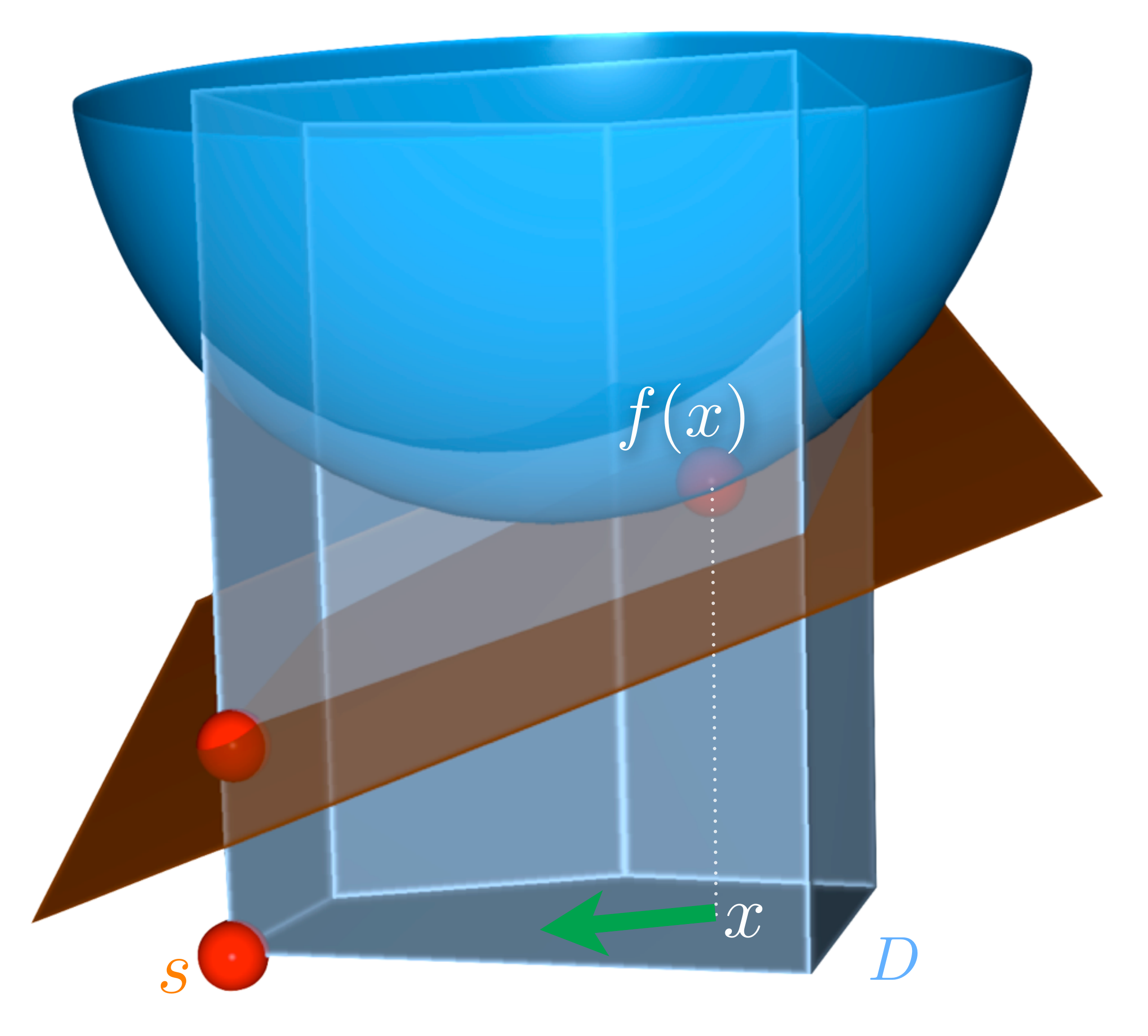}\vspace{-0.6em}
\caption{Visualization of a step of Algorithm \ref{alg:greedyConvex}, moving from the current point $x=x^{(k)}$ towards a linear minimizer $s\in D$. 
Here the two-dimensional domain $D$ is part of the ground plane, and we plot the function values vertically.
Visualization by Robert Carnecky.}
\label{fig:greedyConvex}
\end{center}
\vspace{-0.6em}
\end{figure} 

\paragraph{The Linearized Primitive.}
The internal ``step direction'' procedure $\exactLin(c,D)$ used in Algorithm~\ref{alg:greedyConvex} is a method that minimizes the linear function $\langle x,c\rangle$ over the compact convex domain $D$. Formally it must return a point $s \in D$ such that $\langle s,c\rangle = \displaystyle\min_{y\in D} \, \langle y,c\rangle$.  In terms of the smooth convex optimization literature, the vectors $y$ that have negative scalar product with the gradient, i.e. $\langle y,\nabla f(x)\rangle < 0$, are called \emph{descent directions}, see e.g.~\cite[Chapter 9]{Boyd:2004uz}. %
The main difference to classical convex optimization is that we always choose descent steps staying in the domain $D$, where traditional gradient descend techniques usually use arbitrary directions and need to project back onto $D$ after each step. We will comment more on this analogy in Section~\ref{subsec:optClassic}.

In the alternative interpretation of our duality concept from Section~\ref{sec:poorDual}, the linearized sub-task means that we search for a point $s$ that ``realizes'' the current duality gap $g(x)$, that is the distance to the linear approximation, as given in~(\ref{eq:defGap}).

In the special case that the set $D$ is given by an intersection of linear constraints, this sub-task is exactly equivalent to \emph{linear programming}, as already observed by~\cite[Section 6]{Frank:1956vp}. However, for many other representations of important specific feasible domains $D$, this internal primitive operation is significantly easier to solve, as we will see in the later sections.

The alternative approximate variant of Algorithm~\ref{alg:greedyConvex} uses the procedure $\approxLin\left(c,D,\varepsilon'\right)$ as the internal ``step-direction generator''. Analogously to the exact primitive, this procedure \emph{approximates} the minimum of the linear function $\langle x,c\rangle$ over the convex domain $D$, with additive error $\varepsilon' > 0$. Formally $\approxLin\left(c,D,\varepsilon'\right)$ must return a point $s \in D$ such that $\langle s,c\rangle \le \displaystyle\min_{y\in D} \, \langle y,c\rangle + \varepsilon'$. For several applications, this can be done significantly more efficiently than the exact variant, see e.g. the applications for semidefinite programming in Section~\ref{sec:algTrace}.

\paragraph{The Curvature.}
Everything we need for the analysis of Algorithm~\ref{alg:greedyConvex} is that the linear approximation to our (convex) function $f$ at any point $x$ does not deviate from $f$ by too much, when taken over the whole optimization domain $D$. 

The \emph{curvature constant} $C_{f}$ of a convex and differentiable function $f:\R^n\rightarrow\R$, with respect to the compact domain $D$ is defined as.
\begin{equation}\label{eq:Cf}
  C_f := \sup_{\substack{x,s\in D, \\
                      \alpha\in[0,1],\\
                      y = x+\alpha(s-x)}}
          \textstyle \frac{1}{\alpha^2}\left( f(y)-f(x)-\langle y-x, \nabla f(x)\rangle \right) \ .\vspace{-2pt}
\end{equation}
A motivation to consider this quantity follows if we imagine our optimization procedure at the current point $x=x^{(k)}$, and choosing the next iterate as $y=x^{(k+1)}:=x+\alpha(s-x)$. Bounded $C_f$ then means that the deviation of $f$ at $y$ from the ``best'' linear prediction given by $\nabla f(x)$ is bounded, where the acceptable deviation is weighted by the inverse of the squared step-size $\alpha$. For linear functions $f$ for example, it holds that $C_f=0$. 

The defining term $f(y)-f(x)-\langle y-x, d_x \rangle$ is also widely known as the \emph{Bregman divergence} defined by $f$.
The quantity $C_f$ turns out to be small for many relevant applications, some of which we will discuss later, or see also~\cite{Clarkson:2010hv}.

The assumption of bounded curvature $C_f$ is indeed very similar to a Lipschitz assumption on the gradient of $f$, see also the discussion in Sections~\ref{subsec:Curvature} and~\ref{subsec:optClassic}. In the optimization literature, this property is sometimes also called \emph{$C_f$-strong smoothness}.

It will not always be possible to compute the constant $C_f$ exactly. However, for all algorithms in the following, and also their analysis, it is sufficient to just use some upper bound on $C_f$. We will comment in some more details on the curvature measure in Section~\ref{subsec:Curvature}.

\paragraph{Convergence.}
The following theorem shows that after $O\big(\frac1\varepsilon\big)$ many iterations, Algorithm~\ref{alg:greedyConvex} obtains an $\varepsilon$-approximate solution. The analysis essentially follows the approach of~\cite{Clarkson:2010hv}, inspired by the earlier work of~\cite{Frank:1956vp,Dunn:1978di,Patriksson:1993bl} and~\cite{TongZhang:2003da}.
Later, in Section~\ref{subsec:vecLowerBound}, we will show that this convergence rate is indeed best possible for this type of algorithm, when considering optimization over the unit-simplex. More precisely, we will show that the dependence of the sparsity on the approximation quality, as given by the algorithm here, is best possible up to a constant factor. Analogously, for the case of semidefinite optimization with bounded trace, we will prove in Section~\ref{sec:matLowerBound} that the obtained (low) rank of the approximations given by this algorithm is optimal, for the given approximation quality.

\begin{theorem}[Primal Convergence]\label{thm:primalGreedy}
For each $k\ge1$, the iterate $x^{(k)}$ of the exact variant of Algorithm~\ref{alg:greedyConvex} satisfies
\[
f(x^{(k)}) - f(x^*) \le \frac{4C_f}{k+2} \ ,
\]
where $x^*\in D$ is an optimal solution to problem~(\ref{eq:optGenConvex}). For the approximate variant of Algorithm~\ref{alg:greedyConvex}, it holds that 
\[
f(x^{(k)}) - f(x^*) \le \frac{8C_f}{k+2} \ .
\]
(In other words both algorithm variants deliver a solution of primal error at most $\varepsilon$ after $O(\frac1\varepsilon)$ many iterations.)
\end{theorem}

The proof of the above theorem on the convergence-rate of the primal error crucially depends on the following Lemma~\ref{lem:step} on the improvement in each iteration. 
We recall from Section~\ref{sec:poorDual} that the duality gap for the general convex problem~(\ref{eq:optGenConvex}) over the domain $D$ is given by $g(x) = \displaystyle\max_{s\in D} \, \langle x-s, \nabla f(x)\rangle$. 

\begin{lemma}\label{lem:step}
For any step $x^{(k+1)}:=x^{(k)}+\alpha(s - x^{(k)})$ with arbitrary step-size $\alpha\in[0,1]$, it holds that
\[
  f(x^{(k+1)}) \le f(x^{(k)}) - \alpha g(x^{(k)}) + \alpha^2 C_f 
\]
if $s$ is given by $s := \exactLin\left(\nabla f(x),D\right)$.

If the approximate primitive $s := \approxLin\left(\nabla f(x),D,\alpha C_f\right)$ is used instead, then it holds that
\[
  f(x^{(k+1)}) \le f(x^{(k)}) - \alpha g(x^{(k)}) + 2\alpha^2 C_f  \ .
\]
\end{lemma}
\begin{proof}
  We write $x := x^{(k)}$, $y := x^{(k+1)} = x+\alpha(s-x)$, and $d_x := \nabla f(x)$ to simplify the notation, and first prove the second part of the lemma. We use the definition of the curvature constant $C_f$ of our convex function $f$, to obtain
  \[
  \begin{array}{rl}
    f(y) = & f(x+\alpha(s-x)) \\
    \le & f(x) + \alpha \langle s-x, d_x\rangle + \alpha^2C_f \ .
  \end{array}
  \]
  Now we use that the choice of $s := \approxLin\left(d_x,D,\varepsilon'\right)$ is a good ``descent direction'' on the linear approximation to $f$ at $x$. Formally, we are given a point $s$ that satisfies $\langle s, d_x\rangle \le \displaystyle\min_{y\in D} \langle y,d_x\rangle + \varepsilon'$,
  or in other words we have 
  \[
  \begin{array}{rl}
  \langle s-x, d_x\rangle 
  \le& \min_{y\in D} \langle y,d_x\rangle - \langle x, d_x\rangle + \varepsilon' \\
  =& -g(x,d_x) + \varepsilon' \ ,
  \end{array}
  \]
  from the definition~(\ref{eq:defGap}) of the duality gap $g(x) = g(x,d_x)$. Altogether, we obtain
  \[
  \begin{array}{rl}
  f(y) \le& f(x) + \alpha (-g(x) + \varepsilon') + \alpha^2C_f \\
         =& f(x) - \alpha g(x) + 2\alpha^2 C_f \ ,
  \end{array}
  \]
  the last equality following by our choice of $\varepsilon' = \alpha C_f$. This proves the lemma for the approximate case. The first claim for the exact linear primitive $\exactLin()$ follows by the same proof for $\varepsilon' = 0$.
\end{proof}

Having Lemma~\ref{lem:step} at hand, the proof of our above primal convergence Theorem~\ref{thm:primalGreedy} now follows along the same idea as in~\cite[Theorem 2.3]{Clarkson:2010hv} or~\cite[Theorem 1]{Hazan:2008kz}. Note that a weaker variant of Lemma~\ref{lem:step} was already proven by~\cite{Frank:1956vp}. 

\begin{proof}[Proof of Theorem~\ref{thm:primalGreedy}]
  From Lemma~\ref{lem:step} we know that for every step of the exact variant of Algorithm~\ref{alg:greedyConvex}, it holds that $f(x^{(k+1)}) \le f(x^{(k)}) - \alpha g(x^{(k)}) + \alpha^2 C_f$.
  
  Writing $h(x) := f(x) - f(x^*)$ for the (unknown) primal error at any point $x$, this implies that
  \begin{equation}\label{eq:hStep}
  \begin{array}{rl}
  h(x^{(k+1)}) \le& h(x^{(k)}) - \alpha g(x^{(k)}) + \alpha^2 C_f \\
                    \le& h(x^{(k)}) - \alpha h(x^{(k)}) + \alpha^2 C_f  \\
                    =& (1- \alpha) h(x^{(k)}) + \alpha^2 C_f 
  \ ,
  \end{array}
  \end{equation}
  where we have used weak duality $h(x) \le g(x)$ as given by in~(\ref{eq:duGapGePrGap}). %
  We will now use induction over $k$ in order to prove our claimed bound, i.e.,
  \[
  h(x^{(k+1)}) \le \frac{4C_f}{k+1+2} ~~~~~~ k=0\dots\infty \ .
  \]
  The base-case $k=0$ follows from~(\ref{eq:hStep}) applied for the first step of the algorithm, using $\alpha = \alpha^{(0)} = \frac{2}{0+2}=1$. 
  
  Now considering $k\ge1$, the bound~(\ref{eq:hStep}) gives us
  \[
  \begin{array}{rl}
  h(x^{(k+1)}) \le& (1- \alpha^{(k)}) h(x^{(k)}) + {\alpha^{(k)}}^2 C_f  \\[2pt]
                      =& (1-\frac2{k+2}) h(x^{(k)}) + (\frac2{k+2})^2 C_f  \\[2pt]
                      \le& (1-\frac2{k+2}) \frac{4C_f}{k+2} + (\frac2{k+2})^2 C_f \ ,
  \end{array}
  \]
  where in the last inequality we have used the induction hypothesis for $x^{(k)}$. Simply rearranging the terms gives
  \[
  \begin{array}{rl}
  h(x^{(k+1)}) \le& \frac{4C_f}{k+2} - \frac{8C_f}{(k+2)^2} + \frac{4C_f}{(k+2)^2} \\
                      =& 4C_f \left( \frac1{k+2} - \frac1{(k+2)^2} \right) \\[2pt]
                      =& \frac{4C_f}{k+2} \frac{k+2-1}{k+2} \\[3pt]
                     \le& \frac{4C_f}{k+2} \frac{k+2}{k+3} \\[3pt]
                      =& \frac{4C_f}{k+3} \ ,
  \end{array}
  \]
  which is our claimed bound for $k\ge 1$. 
  
  The analogous claim for Algorithm~\ref{alg:greedyConvex} using the approximate linear primitive $\approxLin()$ follows from the exactly same argument, by replacing every occurrence of $C_f$ in the proof here by $2C_f$ instead (compare to Lemma~\ref{lem:step} also).
\end{proof}

\subsection{Obtaining a Guaranteed Small Duality Gap}\label{subsec:smallGap}
From the above Theorem~\ref{thm:primalGreedy} on the convergence of Algorithm~\ref{alg:greedyConvex}, we have obtained small primal error. %
However, the optimum value $f(x^*)$ is unknown in most practical applications, where we would prefer to have an easily computable measure of the current approximation quality, for example as a stopping criterion of our optimizer in the case that $C_f$ is unknown. The duality gap $g(x)$ that we defined in Section~\ref{sec:poorDual} satisfies these requirements, and always upper bounds the primal error $f(x)-f(x^*)$.

By a nice argument of Clarkson~\cite[Theorem 2.3]{Clarkson:2010hv}, the convergence on the simplex optimization domain can be extended to obtain the stronger property of guaranteed small duality gap $g(x^{(k)}) \le \varepsilon$, after at most $O(\frac1\varepsilon)$ many iterations. 
This stronger convergence result was not yet known in earlier papers of~\cite{Frank:1956vp,Dunn:1978di,Jones:1992ek,Patriksson:1993bl} and~\cite{TongZhang:2003da}. Here we will generalize the primal-dual convergence to arbitrary compact convex domains.
The proof of our theorem below again relies on Lemma~\ref{lem:step}.

\begin{theorem}[Primal-Dual Convergence]\label{thm:primalDualGreedy}
Let $K := \left\lceil\frac{4C_f}{\varepsilon}\right\rceil$. We run the exact variant of Algorithm~\ref{alg:greedyConvex} for $K$ iterations (recall that the step-sizes are given by $\alpha^{(k)} := \frac2{k+2}$, $0\le k\le K$), and then continue for another $K+1$ iterations, now with the fixed step-size $\alpha^{(k)} := \frac2{K+2}$ for $K\le k\le 2K+1$.

Then the algorithm has an iterate $x^{(\hat k)}$, $K\le \hat k\le 2K+1$, with duality gap bounded by 
\[
g(x^{(\hat k)}) \le \varepsilon \ .
\]

The same statement holds for the approximate variant of Algorithm~\ref{alg:greedyConvex}, when setting $K := \left\lceil\frac{8C_f}{\varepsilon}\right\rceil$ instead.
\end{theorem}
\begin{proof}
The proof follows the idea of~\cite[Section 7]{Clarkson:2010hv}.

By our previous Theorem~\ref{thm:primalGreedy} we already know that the primal error satisfies $h(x^{(K)}) = f(x^{(K)})-f(x^*) \le \frac{4C_f}{K+2}$ after $K$ iterations.
In the subsequent $K+1$ iterations, we will now suppose that $g(x^{(k)})$ always stays larger than $\frac{4C_f}{K+2}$. We will try to derive a contradiction to this assumption.

Putting the assumption $g(x^{(k)})>\frac{4C_f}{K+2}$ into the step improvement bound given by Lemma~\ref{lem:step}, we get that
\[
\begin{array}{rl}
  f(x^{(k+1)}) - f(x^{(k)}) \le&  - \alpha^{(k)} g(x^{(k)}) + {\alpha^{(k)}}^2 C_f \\
                     <&  - \alpha^{(k)} \frac{4C_f}{K+2} + {\alpha^{(k)}}^2 C_f 
\end{array}
\]
holds for any step size $\alpha^{(k)} \in (0,1]$. Now using the fixed step-size $\alpha^{(k)} = \frac2{K+2}$ in the iterations $k\ge K$ of Algorithm~\ref{alg:greedyConvex}, this reads as

\[
\begin{array}{rl}
  f(x^{(k+1)}) - f(x^{(k)}) <& - \frac2{K+2} \frac{4C_f}{K+2} + \frac4{(K+2)^2} C_f \vspace{3pt}\\
           =& - \frac{4C_f}{(K+2)^2}
\end{array}
\]
Summing up over the additional steps, we obtain\vspace{-5pt}
\[
\begin{array}{rl}
  f(x^{(2K+2)}) - f(x^{(K)}) =& \displaystyle\sum_{k=K}^{2K+1} f(x^{(k+1)}) - f(x^{(k)})\\
  <& - (K+2)\frac{4C_f}{(K+2)^2}
                                   = -  \frac{4C_f}{K+2}  \ ,
\end{array}
\]
which together with our known primal approximation error $f(x^{(K)}) - f(x^*) \le \frac{4C_f}{K+2}$ would result in $ f(x^{(2K+2)}) - f(x^*) < 0$, a contradiction. Therefore there must exist $\hat k$, $K\le \hat k\le 2K+1$, with $g(x^{(\hat k)}) \le \frac{4C_f}{K+2} \le \varepsilon$.

The analysis for the approximate variant of Algorithm~\ref{alg:greedyConvex} follows using the analogous second bound from Lemma~\ref{lem:step}.
\end{proof}

Following~\cite[Theorem 2.3]{Clarkson:2010hv}, one can also prove a similar primal-dual convergence theorem for the line-search algorithm variant that uses the optimal step-size in each iteration, as we will discuss in the next Section~\ref{subsec:lineSearch}. This is somewhat expected as the line-search algorithm in each step is at least as good as the fixed step-size variant we consider here.

\subsection{Choosing the Optimal Step-Size by Line-Search}\label{subsec:lineSearch}
Alternatively, instead of the fixed step-size $\alpha = \frac2{k+2}$ in Algorithm~\ref{alg:greedyConvex}, one can also find the optimal $\alpha \in [0,1]$ by line-search. This will not improve the theoretical convergence guarantee, but might still be considered in practical applications if the best $\alpha$ is easy to compute. Experimentally, we observed that line-search can improve the numerical stability in particular if approximate step directions are used, which we will discuss e.g. for semidefinite matrix completion problems in Section~\ref{sec:applicationMatCompl}.

Formally, given the chosen direction $s$, we then search for the $\alpha$ of best improvement in the objective function $f$, that is
\begin{equation}\label{eq:genLineSearch}
\alpha := \argmin_{\alpha\in[0,1]} \,f\left(x^{(k)}+\alpha(s - x^{(k)})\right) \ .
\end{equation}
The resulting modified version of Algorithm~\ref{alg:greedyConvex} is depicted again in Algorithm~\ref{alg:greedyConvexLS}, and was precisely analyzed in~\cite{Clarkson:2010hv} for the case of optimizing over the simplex.

\begin{algorithm}[h!]
  \caption{Greedy on a Convex Set, using Line-Search}
  \label{alg:greedyConvexLS}
\begin{algorithmic}
  \STATE {\bfseries Input:} Convex function $f$, convex set $D$, target accuracy $\varepsilon$
  \STATE {\bfseries Output:} $\varepsilon$-approximate solution for problem~(\ref{eq:optVecSimplex})
  \STATE Pick an arbitrary starting point $x^{(0)} \in D$
  \FOR{$k=0\dots\infty$}
  \STATE \quad Compute $s := \exactLin\left(\nabla f(x^{(k)}),D\right)$
  \STATE ---or---
  \STATE \quad Compute $s := \approxLin\left(\nabla f(x^{(k)}),D,\frac{2C_f}{k+2}\right)$
  \STATE Find the optimal step-size $\alpha := \displaystyle\argmin_{\alpha\in[0,1]} \, f\left(x^{(k)}+\alpha(s - x^{(k)})\right)$
  \STATE Update $x^{(k+1)}:=x^{(k)}+\alpha(s - x^{(k)})$
  \ENDFOR
\end{algorithmic}
\end{algorithm}

In many cases, we can solve this line-search analytically in a straightforward manner, by differentiating the above expression with respect to $\alpha$: Consider $f_{\alpha} := f\left(x^{(k+1)}_{(\alpha)} \right) = f\left(x^{(k)} + \alpha \left(s - x^{(k)} \right)\right)$ and compute 
\begin{equation}\label{eq:bestOnLineSegment}
0 \stackrel{\mbox{\scriptsize !}}{=}
 \frac{\partial}{\partial\alpha} f_{\alpha} = \left\langle s - x^{(k)} , \nabla f\big(x^{(k+1)}_{(\alpha)} \big) \right\rangle \ .
\end{equation}
If this equation can be solved for $\alpha$, then the optimal such $\alpha$ can directly be used as the step-size in Algorithm~\ref{alg:greedyConvex}, and the convergence guarantee of Theorem~\ref{thm:primalGreedy} still holds. This is because the improvement in each step will be at least as large as if we were using the older (potentially sub-optimal) fixed choice of $\alpha = \frac2{k+2}$. Here we have assumed that $\alpha^{(k)}\in[0,1]$ always holds, which can be done when using some caution, see also~\cite{Clarkson:2010hv}.

Note that the line-search can also be used for the approximate variant of Algorithm~\ref{alg:greedyConvex}, where we keep the accuracy for the internal primitive method $\approxLin\left(\nabla f(x^{(k)}),D,\varepsilon'\right)$ fixed to $\varepsilon' = \alpha_{\mbox{\tiny fixed}} C_f = \frac{2C_f}{k+2}$. Theorem~\ref{thm:primalGreedy} then holds as as in the original case.

%
%

%
\subsection{The Curvature Measure of a Convex Function}\label{subsec:Curvature}

For the case of differentiable $f$ over the space $\X=\R^n$, we recall the definition of the curvature constant $C_f$ with respect to the domain $D \subset \R^n$, as stated in~(\ref{eq:Cf}),
\[
  C_f := \sup_{\substack{x,s\in D, \\
                      \alpha\in[0,1],\\
                      y = x+\alpha(s-x)}}
          \textstyle \frac{1}{\alpha^2}\left( f(y)-f(x)- (y-x)^T \nabla f(x) \right) \ .
\]

An overview of values of $C_f$ for several classes of functions $f$ over domains that are related to the unit simplex can be found in~\cite{Clarkson:2010hv}.

\paragraph{Asymptotic Curvature.}
As Algorithm~\ref{alg:greedyConvex} converges towards some optimal solution $x^*$, it also makes sense to consider the asymptotic curvature $C_f^*$, defined as
\begin{equation}\label{eq:CfAsympt}
  C_f^* := %
                \sup_{\substack{s\in D, \\
                          \alpha\in[0,1],\\
                          y = x^*+\alpha(s-x^*)}}
          \textstyle \frac{1}{\alpha^2}\left( f(y)-f(x^*)- (y-x^*)^T \nabla f(x^*) \right) \ .
\end{equation}
Clearly $C_f^* \le C_f$. As described in~\cite[Section 4.4]{Clarkson:2010hv}, we expect that as the algorithm converges towards $x^*$, also the improvement bound as given by Lemma~\ref{lem:step} should be determined by $C_f^*+o(1)$ instead of $C_f$, resulting in a better convergence speed constant than given Theorem~\ref{thm:primalGreedy}, at least for large $k$.
The class of \emph{strongly convex} functions is an example for which the convergence of the relevant constant towards $C_f^*$ is easy to see, since for these functions, convergence in the function value also imlies convergence in the domain, towards a unique point $x^*$, see e.g.~\cite[Section 9.1.2]{Boyd:2004uz}.

\paragraph{Relating the Curvature to the Hessian Matrix.}
Before we can compare the assumption of bounded curvature $C_f$ to a Lipschitz assumption on the gradient of $f$, we will need to relate $C_f$ to the Hessian matrix (matrix of all second derivatives) of $f$.

Here the idea described in~\cite[Section 4.1]{Clarkson:2010hv} is to make use of the degree-2 Taylor-expansion of our function $f$ at the fixed point $x$, as a function of $\alpha$, which is\vspace{-4pt}
\[
f(x+\alpha(s-x)) 
= f(x) + \alpha(s-x)^T \nabla f(x) + \frac{\alpha^2}{2} (s-x)^T \nabla^2 f(z) (s-x) \ ,
\]
where $z$ is a point on the line-segment $[x,y] \subseteq D \subset \R^d$ between the two points $x\in\R^n$ and $y=x+\alpha(s-x)\in\R^n$.
To upper bound the curvature measure, we can now directly plug in this expression for $f(y)$ into the above definition of $C_f$, obtaining
\begin{equation}\label{eq:CfHess}
C_f \le \sup_{\substack{x,y \in D, \\
                    z \in [x,y] \subseteq D}} 
            \frac12 (y-x)^T \nabla^2 f(z) (y-x) \ .
\end{equation}
From this bound, it follows that $C_f$ is upper bounded by the largest eigenvalue of the Hessian matrix of $f$, scaled with the domain's Euclidean diameter, or formally

\begin{lemma}\label{lem:CfHessBound} 
For any twice differentiable convex function $f$ over a compact convex domain $D$, it holds that
\[
C_f \,\le\, \frac12 \diam(D)^2 \cdot \displaystyle\sup_{z\in D} \, \lmax\left(\nabla^2 f(z)\right) \ .
\]
\end{lemma}
Note that as $f$ is convex, the Hessian matrix $\nabla^2 f(z)$ is positive semidefinite for all $z$, see e.g.~\cite[Theorem 7.3.6]{Kurdila:2005ur}.
\begin{proof}
Applying the Cauchy-Schwarz inequality to~(\ref{eq:CfHess}) for any $x,y\in D$ (as in the definition of $C_f$), we get
\[
\begin{array}{rl}
(y-x)^T \nabla^2 f(z) (y-x) \le& \norm{y-x}_2 \norm{\nabla^2 f(z) (y-x)}_2  \vspace{3pt}\\
\le& \norm{y-x}_2^2 \spectnorm{\nabla^2 f(z)} \vspace{3pt}\\
\le& \diam(D)^2 \cdot \displaystyle\sup_{z\in D} \lmax\left(\nabla^2 f(z)\right)
\ .
\end{array}
\]
The middle inequality follows from the variational characterization of the matrix spectral norm, i.e. $\spectnorm{A} := \sup_{x\ne 0}\frac{\norm{Ax}_2}{\norm{x}_2}$. Finally, in the last inequality we have used that by convexity of $f$, the Hessian matrix $\nabla^2 f(z)$ is PSD, so that its spectral norm is its largest eigenvalue.  %
\end{proof}

Note that in the case of $D$ being the unit simplex (see also the following Section~\ref{sec:vecSimplex}), we have that $\diam(\Delta_n) = \sqrt{2}$, meaning the scaling factor disappears, i.e. $C_f \le \displaystyle\sup_{z\in \Delta_n} \lmax\left(\nabla^2 f(z)\right)$.

\paragraph{Bounded Curvature vs. Lipschitz-Continuous Gradient.}
Our core assumption on the given optimization problems is that that the curvature $C_f$ of the function is bounded over the domain. Equivalently, this means that the function does not deviate from its linear approximation by too much. Here we will explain that this assumption is in fact very close to the natural assumption that the gradient $\nabla f$ is Lipschitz-continuous, which is often assumed in classical convex optimization literature, where it is sometimes called \emph{$C_f$-strong smoothness}, see e.g.~\cite{Nemirovski:2005wu,Kakade:2009wh} (or~\cite{dAspremont:2008gl} if the gradient information is only approximate). %

\begin{lemma}
Let $f$ be a convex and twice differentiable function, and assume that the gradient $\nabla f$ is Lipschitz-continuous over the domain $D$ with Lipschitz-constant $L>0$. Then
\[
C_f \le \frac12 \diam(D)^2 L \ .
\]
\end{lemma}
\begin{proof}
Having
$\norm{\nabla f(y)-\nabla f(x)}_2 \le L \norm{y-x}_2$ $\forall x,y \in D$ by the Cauchy-Schwarz inequality implies that $(y-x)^T(\nabla f(y)-\nabla f(x)) \le L \norm{y-x}_2^2$, so that%
\begin{equation}\label{eq:lipschitzCurvature}
f(y) \le f(x) + (y-x)^T \nabla f(x) + \frac L2 \norm{y-x}_2^2 \ .
\end{equation}

If $f$ is twice differentiable, it can directly be seen that the above condition implies that $L\cdot \id \succeq \nabla^2 f(z)$ holds for the Hessian of $f$, that is $\lmax\left(\nabla^2 f(z)\right) \le L$.  %

Together with our result from Lemma~\ref{lem:CfHessBound}%
, the claim follows.
\end{proof}
 
The above bound~(\ref{eq:lipschitzCurvature}) which is implied by Lipschitz-continuous gradient means that the function is not ``curved'' by more than $L$ in some sense, which is an interesting property. In fact this is exactly the opposite inequality compared to the property of \emph{strong convexity}, which is an assumption on the function $f$ that we do not impose here. Strong convexity on a compact domain means that the function is always curved \emph{at least} by some constant (as our $L$).
We just note that for strongly convex functions, ``accelerated'' algorithms with an even faster convergence of $\frac{1}{k^2}$ (meaning $O(\frac{1}{\sqrt{\varepsilon}})$ steps) do exist~\cite{Nesterov:2004ic,Nesterov:2007wm}. %

\subsection{Optimizing over Convex Hulls}\label{subsec:convHull}
In the case that the optimization domain $D$ is given as the convex hull of a (finite or infinite) subset $V \subset \X$, i.e.
\[
D = \conv(V) \ ,
\]
then it is particularly easy to solve the linear optimization subproblems as needed in our Algorithm~\ref{alg:greedyConvex}. Recall that $\conv(V)$ is defined as the set of all finite convex combinations $\sum_i\alpha_iv_i$ for a finite subset $\{v_1,\dots,v_k\} \subseteq V$, while $V$ can also be an infinite set.

\begin{lemma}[Linear Optimization over Convex Hulls]\label{lem:convHull}
Let $D = \conv(V)$ for any subset $V \subset \X$, and $D$ compact. Then any linear function $y \mapsto \langle y,c \rangle$ will attain its minimum and maximum over $D$ at some ``vertex'' $v\in V$.
\end{lemma}
\begin{proof}
W.l.g. we will only show the case for the maximum. Let $s \in D$ be a point attaining the linear optimum $\langle s,c \rangle = \max_{y\in D} \langle y,c \rangle$. 
Then by definition of $D$, we have that $s = \sum_{i=1}^{k} \alpha_i v_i$, meaning that $s$ is the weighted average of some finite set of ``vertices'' $v_1,\dots,v_k \in V$, with $\alpha_i\ge0, \sum_i \alpha_i =1$. By linearity of the inner product,\vspace{-5pt}
\[
\langle s,c \rangle 
= \left\langle \sum_{i=1}^{k} \alpha_i v_i , c \right\rangle 
= \sum_{i=1}^{k} \alpha_i \langle v_i,c \rangle \ ,
\]
and therefore we must have that $\langle v_i,c \rangle \ge \langle s,c \rangle$ for at least one of the indices $i$, meaning that $v_i \in V$ is also attaining the linear maximum.
\end{proof}

In the following we will discuss several application where this simple fact will be useful to solve the linearized subproblems $\exactLin()$ more efficiently, as the set $V$ is often much easier to describe than the full compact domain $D$.

The setting of convex optimization over a convex hull in a vector space was already studied %
by~\cite{TongZhang:2003da}. There, each iteration of the optimizer consists of greedily selecting the point (or ``vertex'') of $V$ which promises best improvement. \cite{TongZhang:2003da} then gave a similar primal convergence guarantee as in our Theorem~\ref{thm:primalGreedy} (but no primal-dual convergence result on general convex hulls was known so far, to the best of our knowledge). The above Lemma~\ref{lem:convHull} in a sense explains the relation to our linearized internal problem. The main difference is that the algorithm of~\cite{TongZhang:2003da} always evaluates the original non-linear function $f$ at all vertices $V$, while our slightly more general framework only relies on the linear subproblem, and allows for arbitrary means to approximately solve the subproblem.
\subsection{Randomized Variants, and Stochastic Optimization}\label{subsec:SGD}
For a variety of classes of our convex optimization problems,  randomization can help to solve the linearized subproblem more efficiently.
This idea is strongly related to online and stochastic optimization, see e.g.~\cite{Nesterov:2011wm}, and also the popular stochastic gradient descent (SGD) techniques~\cite{Bottou:2010uz}.

We can also apply such cheaper randomized steps in our described framework, instead of deterministically solving the internal linear problem in each iteration.
Assumed that the user of our method is able to decompose the linearized problem in some arbitrary way using randomization, and if the randomization is such that the linearized problem will be solved ``accurately enough'' with some probability $p>0$ in each iteration, then our convergence analysis still holds also in this probabilistic setting as follows:

Formally, we assume that we are given access to a randomized procedure $\randomLin\left(c,D,\varepsilon'\right)$, which returns a point $s \in D$ such that $\langle s,c\rangle \le \displaystyle\min_{y\in D} \, \langle y,c\rangle + \varepsilon'$ with probability at least $p>0$. In other words, with a positive probability, $\randomLin()$ will behave like $\approxLin()$.
In each iteration of the line-search variant of our algorithm (see Algorithm~\ref{alg:greedyConvexLS}), we will now use that randomized internal procedure instead. The expected improvement given by a step towards $s=\randomLin()$ is at least $p$ times the amount given in Lemma~\ref{lem:step}. (Here we have used that in the events of ``failure'' of $\randomLin()$, the objective function value will at least not become worse, due to the use of line-search).

In other words if we perform $\frac1p$ times more iterations than required for the deterministic Algorithm~\ref{alg:greedyConvex}, then we have that the convergence by Theorem~\ref{thm:primalGreedy} also holds for the randomized variant described here. 

\paragraph{Stochastic Gradient Descent (SGD).}
A classical example is when the linearized problem is given by simply finding the maximum over say $n$ coordinates, as we will e.g. see in the following Sections~\ref{sec:vecSimplex} and~\ref{sec:vecL1} for optimizing over the simplex, or over bounded $\ell_1$-norm.
In this case, by sampling uniformly at random, with probability $\frac1n$ we will pick the correct coordinate, for which the step improvement is as in the deterministic Algorithm~\ref{alg:greedyConvex}. Therefore 
we have obtained the same convergence guarantee as for the deterministic algorithm, but the necessary number of steps is multiplied by $n$.

For unconstrained convex optimization, the convergence of SGD and other related methods was analyzed e.g. in~\cite{Nesterov:2011wm} and also the earlier paper~\cite[Section 6]{Nesterov:2007ee}. Also here, a comparable slow-down was observed when using the cheaper randomized steps.
\subsection{Relation to Classical Convex Optimization}\label{subsec:optClassic}

\paragraph{Relation to Gradient Descent and Steepest Descent.}

The internal linear optimizer in our Algorithm~\ref{alg:greedyConvex} can also be interpreted in terms of descent directions. Recall that all vectors $y$ that have negative scalar product with the current gradient, i.e. $\langle y, \nabla f(x)\rangle < 0$, are called  \emph{descent directions}, see e.g.~\cite[Chapter 9.4]{Boyd:2004uz}. Also observe that $\langle y, \nabla f(x)\rangle$ is the directional derivative of $f$ in direction of $y$ if $y$ is of unit length. Our method therefore chooses the best descent direction over the entire domain $D$, where the quality is measured as the best possible absolute improvement as suggested by the linearization at the current point $x$. In any iteration, this will crucially depend on the \emph{global} shape of the domain $D$, even if the actual step-size $\alpha^{(k)}$ might be very small.

This crucially contrasts classical gradient descent techniques, which only use \emph{local} information to determine the step-directions, facing the risk of walking out of the domain $D$ and therefore requiring projection steps after each iteration.

\paragraph{Relation to Inaccurate and Missing Gradient Information.}
The ability of our Algorithm~\ref{alg:greedyConvex} to deal with only approximate internal linear optimizers as in $\approxLin()$ is also related to existing methods that assume that gradient information is only available with noise, or in a stochastic or sampling setting. 

For the case of optimizing smooth convex functions,~\cite{dAspremont:2008gl} has used a similar measure of error, namely that the linearization given by the ``noisy'' version $\tilde d_x$ of the gradient $\nabla f(x)$ does not differ by more than say $\varepsilon'$ when measured over the entire domain $D$, or formally 
\[
\abs{\langle y-z, \tilde d_x\rangle  - \langle y-z, \nabla f(x)\rangle} \le \varepsilon' \ , \ ~~\forall x,y,z\in D \ .
\]
This assumption is similar, but stronger than the additive approximation quality that we require in our above setting (we only need that the linearized \emph{optimal values} are within $\varepsilon'$). Also, the algorithm as well as the analysis in~\cite{dAspremont:2008gl} are more complicated than the method proposed here, due to the need of projections and proximal operators.

We have discussed the case where gradient information is available only in a stochastic oracle (e.g. such that the gradient is obtained in expectation) in the above Subsection~\ref{subsec:SGD}. For an overview of related randomized methods in unconstrained convex optimization, we refer the reader to the recent work by~\cite{Nesterov:2011wm}, which also applies when the gradient itself is not available and has to be estimated by oracle calls to the function alone.

If gradient information can be constructed in any way such that the linearized problem $\approxLin()$ can be solved to the desired additive error, then our above analysis of Algorithm~\ref{alg:greedyConvex} will still hold.

\paragraph{Relation to Mirror Descent, Proximal Methods and Conjugate Functions.}
Our proposed method is related to \emph{mirror descent} as well as \emph{proximal methods} in convex optimization, but our approach is usually simpler. 
The mirror descent technique originates from e.g.~\cite{Beck:2003dv,BenTal:2005wn}. For a brief overview of proximal methods with applications to some of the classes of sparse convex optimization problems as studied here, we refer to~\cite[Section 3]{Bach:2011uc}.

To investigate the connection, we write $f_{\mbox{\tiny lin}|x}(y) := f(x)+ \langle y-x , d_x \rangle$ for the linearization given by the (sub)gradient $d_x = \nabla f(x)$ at a fixed point $x\in D$.
A variant of mirror descent, see e.g.~\cite{BenTal:2005wn,Hazan:2011wf} is to find the next iterate $y$ as the point maximizing the Bregman divergence 
\begin{equation}\label{eq:bregman}
f(y)-f(x)- \langle y-x, d_x \rangle = f(y) - f_{\mbox{\tiny lin}|x}(y)
\end{equation}
relative to the currently fixed old iterate $x$.  This is the same task as maximizing the gap between the function $f(y)$ and its linear approximation at $x$, or equivalently we evaluate the \emph{conjugate function} $f^*(z) := \displaystyle\sup_{y\in D} \langle y,z\rangle - f(y)$ at $z=d_x$. The definition of the conjugate dual is also known as \emph{Fenchel duality}, see e.g.~\cite{Borwein:2006ts}. In~\cite{Nemirovski:2005wu}, the conjugate function is also called the Legendre transformation.

However in our approach, the inner task $\exactLin(d_x,D)$ as well as $\approxLin(d_x,D,\varepsilon')$ is a simpler linear problem. Namely, we directly minimize the linearization at the current point $x$, i.e. we maximize
\begin{equation}\label{eq:maxLinAtX}
-f(x)- \langle y-x, d_x\rangle = - f_{\mbox{\tiny lin}|x}(y)
\end{equation}
and then move towards an approximate maximizer $y$. In terms of Fenchel duality, this simpler linear problem is the evaluation of the conjugate dual of the \emph{characteristic function} of our domain $D$, i.e. 
\[
\ind_D^*(z) := \displaystyle\sup_{y\in \X} \langle y,z\rangle - \ind_D(y)\ ,
\]
where this function is evaluated at the current subgradient $z=d_x$. The \emph{characteristic function} $\ind_D:\X\rightarrow\overline\R$ of a set $D\subseteq\X$ is defined as $\ind_D(y)=0$ for $y\in D$ and $\ind_D(y)=\infty$ otherwise.
Compared to our algorithm, mirror descent schemes require a ``projection'' step in each iteration, sometimes also called the \emph{proximal} or \emph{mirror operator}. This refers to minimizing the linearization plus a strongly convex prox-function that punishes the distance to the starting point. If the squared Euclidean norm is used, the mirror operator corresponds to the standard projection back onto the domain $D$.
Our method uses no such prox-function, and neither is the zero-function a strongly convex one, as would be required for mirror descent to work. It is expected that the computational cost per iteration of our method will in most application cases be lower compared to mirror descent schemes. 

For convex optimization over the simplex, which we will study in more details in the following Section~\ref{sec:vecSimplex},~\cite{Beck:2003dv} have proposed a mirror descent algorithm, obtaining a convergence of $f(x^{(k)})-f(x^*) \le \sqrt{2\ln n} \frac{L}{\sqrt{k}}$. This however is worse than the convergence of our methods as given by Theorem~\ref{thm:primalGreedy}. Our convergence is independent of the dimension $n$, and goes with $\frac1k$ instead of $\frac1{\sqrt{k}}$.
Also the earlier paper by~\cite{BenTal:2001kd} only obtained a convergence of $O\big(\frac{1}{\sqrt{k}}\big)$ for the case of Lipschitz-continuous convex functions over the simplex.

The NERML optimizer by~\cite{BenTal:2005wn} is a variant of mirror descent that memorizes several past linearizations of the objective function, to improve the available knowledge about the current function landscape. It is an open research question if this idea could also help in our setting here, or for stochastic gradient descent schemes~\cite{Bottou:2010uz}.

\section{Sparse Approximation over the Simplex}\label{sec:vecSimplex}

As a first application of the above general scheme, we will now consider optimization problems defined over the unit simplex, or in other words the non-negative vectors of $\ell_1$-norm equal to one.
This will serve as a warm-up case before considering $\ell_1$-norm regularized problems in the next Section~\ref{sec:vecL1}.

Our approach here will allow us to understand the best achievable sparsity of approximate solutions, as a function of the approximation quality, as already shown by~\cite{Clarkson:2010hv}. 

In particular, we will show that our main Algorithm~\vref{alg:greedyConvex} and its analysis do lead to Clarkson's approach~\cite{Clarkson:2010hv} for optimizing over the simplex. In this case, it was already known that sparsity $O\big(\frac1\varepsilon\big)$ can always be achieved by applying Algorithm~\ref{alg:greedyConvex} to the simplex domain, see~\cite{Clarkson:2010hv}. We will also show that this is indeed optimal, by providing an asymptotically matching lower bound in Section~\ref{subsec:vecLowerBound}. Also, our analysis holds even if the linear subproblems are only solved approximately, and allows arbitrary starting points, in contrast to~\cite{Clarkson:2010hv}.

Having this efficient algorithm giving sparsity $O\big(\frac1\varepsilon\big)$ is in particularly attractive in view of the computational complexity of \emph{vector cardinality minimization}, which is known to be NP-hard, by a reduction to Exact-Cover, see~\cite{Natarajan:1995ww}. Vector cardinality minimization here refers to finding the sparsest vector that is an $\varepsilon$-approximation to some given convex minimization problem. Formally, finding the sparsest $x$ that satisfies $\norm{Ax-b}_2 \le \varepsilon$ for given $A,b$ and $\varepsilon$.

\paragraph{Set-Up.}
We suppose that a basis has been chosen in the space $\X$, so that we can assume $\X = \R^n$ with the standard inner product $\langle x,y\rangle = x^Ty$.
Here we consider one special class of the general optimization problems~(\ref{eq:optGenConvex}), namely we optimize over non-negative vectors that sum up to one, that is
\begin{equation}\label{eq:optVecSimplex}
\begin{array}{rl}
   \displaystyle\mini_{x \in \R^n} & f(x)  \\
    s.t. &  \norm{x}_1 = 1 \ , \\
         & x \ge 0 \ .
\end{array}
\end{equation}

In the following we write $\Delta_n := \SetOf{x \in \R^n}{x\ge 0,\ \norm{x}_1 = 1}$ for the unit simplex in $\R^n$.
As the objective function $f$ is now defined over $\R^n$, all subgradients or gradients of $f$ will also be represented by vectors in $\R^n$ in the following.

Note that the alternative case of optimizing under an inequality constraint $\norm{x}_1 \le 1$ instead of $\norm{x}_1 = 1$ can easily be reduced to the above form~(\ref{eq:optVecSimplex}) by introducing a new ``slack'' variable. Formally, one uses vectors $\hat x =(x_1,\dots,x_n,x_{n+1}) \in \R^{n+1}$ instead and optimizes the function $\hat f(\hat x) := f(x_1,\dots,x_n)$ over the simplex domain $\norm{\hat x}_1 = 1$, $\hat x \ge 0$.

\paragraph{Coresets.}
The coreset concept was originally introduced in computational geometry by~\cite{Badoiu:2002ug} and~\cite{Agarwal:2002fk}. For point-set problems, the coreset idea refers to identifying a very small subset (coreset) of the points, such that the solution just on the coreset is guaranteed to be a good approximation to original problem. Here for general convex optimization, the role of the coreset points is taken by the non-zero coordinates of our sought vector $x$ instead. The coreset size then corresponds to the sparsity of $x$.

Formally if there exists an $\varepsilon$-approximate solution $x \in D \subseteq \R^n$ to the convex optimization problem~(\ref{eq:optGenConvex}), using only $k$ many non-zero coordinates, then we say that the corresponding coordinates of $x$ form an \emph{$\varepsilon$-coreset} of size $k$ for problem~(\ref{eq:optGenConvex}).

In other words, the following upper and lower bounds of $O\big(\frac1\varepsilon\big)$ on the sparsity of approximations for problem~(\ref{eq:optVecSimplex}) are indeed matching upper and lower bounds on the coreset size for convex optimization over the simplex, analogous to what we have found in the geometric problem setting in~\cite{GartnerJaggi:2009}.

\subsection{Upper Bound: Sparse Greedy on the Simplex}\label{subsec:algSimplex}
Here we will show how the general algorithm and its analysis from the previous Section~\ref{sec:optGenConvex} do in particular lead to Clarkson's approach~\cite{Clarkson:2010hv} for minimizing any convex function over the unit simplex. The algorithm follows directly from Algorithm~\ref{alg:greedyConvex}, and will have a running time of $O\left(\frac1\varepsilon\right)$ many gradient evaluations. We will crucially make use of the fact that every linear function attains its minimum at a vertex of the simplex $\Delta_n$. Formally, for any vector $c \in\R^n$, it holds that $\displaystyle\min_{s\in\Delta_n} s^T c = \min_i c_i \,$. This property is easy to verify in the special case here, but is also a direct consequence of the small Lemma~\ref{lem:convHull} which we have proven for general convex hulls, if we accept that the unit simplex is the convex hull of the unit basis vectors.
We have obtained that the internal linearized primitive can be solved exactly by choosing
\[
\exactLin\left(c,\Delta_n \right) := \unit_i ~~\text{ with }i = \argmin_i \, c_i  \ .
\]

\begin{algorithm}[h!]
  \caption{Sparse Greedy on the Simplex}
  \label{alg:greedySimplex}
\begin{algorithmic}
  \STATE {\bfseries Input:} Convex function $f$, target accuracy $\varepsilon$
  \STATE {\bfseries Output:} $\varepsilon$-approximate solution for problem~(\ref{eq:optVecSimplex})
  \STATE Set $x^{(0)}:= \unit_1$
  \FOR{$k=0\dots\infty$}
  \STATE Compute $i := \argmin_i \left(\nabla f(x^{(k)})\right)_i$
  \STATE Let $\alpha := \frac2{k+2}$
  \STATE Update $x^{(k+1)}:=x^{(k)}+\alpha(\unit_i - x^{(k)})$
  \ENDFOR
\end{algorithmic}
\end{algorithm}

Observe that in each iteration, this algorithm only introduces at most one new non-zero coordinate, so that the sparsity of $x^{(k)}$ is always upper bounded by the number of steps $k$, plus one, given that we start at a vertex. 
Since Algorithm~\ref{alg:greedySimplex} only moves in coordinate directions, it can be seen as a variant of \emph{coordinate descent}.
The convergence result directly follows from the general analysis we gave in the previous Section~\ref{sec:optGenConvex}. 

\begin{theorem}[{\cite[Theorem 2.3]{Clarkson:2010hv}, Convergence of Sparse Greedy on the Simplex}]\label{thm:sparseSimplex}
For each $k\ge1$, the iterate $x^{(k)}$ of Algorithm~\ref{alg:greedySimplex} satisfies
\[
f(x^{(k)}) - f(x^*) \le \frac{4C_f}{k+2} \ .
\]
where $x^*\in \Delta_n$ is an optimal solution to problem~(\ref{eq:optVecSimplex}).

Furthermore, for any $\varepsilon>0$, after at most $2\left\lceil\frac{4C_f}{\varepsilon}\right\rceil +1 = O\left(\frac1\varepsilon\right)$ many steps\footnote{%
Note that in order for our Theorem~\ref{thm:primalDualGreedy} on the bounded duality gap to apply, the step-size in the second half of the iterations needs to be fixed to $\alpha^{(k)} := \frac2{K+2}$, see Section~\ref{subsec:smallGap}. This remark also applies to the later applications of our general Algorithm~\ref{alg:greedyConvex} that we discuss in the following. We already mentioned above that if line-search is used instead, then no such technicality is necessary, see also~\cite{Clarkson:2010hv}. 
}%
, it has an iterate $x^{(k)}$ of sparsity $O\left(\frac1\varepsilon\right)$, satisfying $g(x^{(k)}) \le\varepsilon$.
\end{theorem}
\begin{proof}
This is a corollary of Theorem~\ref{thm:primalGreedy} and Theorem~\ref{thm:primalDualGreedy}.
\end{proof}

\paragraph{Duality Gap.}
We recall from Section~\ref{sec:poorDual} that the duality gap~(\ref{eq:defGap}) at any point $x\in\Delta_n$ is easily computable from any subgradient, and in our case becomes
\begin{equation}\label{eq:gapSimplex}
\begin{split}
g(x,d_x) &= x^T d_x  - \min_i \, (d_x)_i \ ,\ \ \ \ \text{ and}\\
g(x) &= x^T \nabla f(x) - \min_i (\nabla f(x))_i \ .
\end{split}
\end{equation}
Here we have again used the observation that linear functions attain their minimum at a vertex of the domain, i.e, $\displaystyle\min_{s\in\Delta_n} s^T c = \min_i c_i $.

\paragraph{Applications.}
Many practically relevant optimization problems do fit into our setting~(\ref{eq:optVecSimplex}) here, allowing the application of Algorithm~\ref{alg:greedySimplex}. This includes linear classifiers such as support vector machines (SVMs), see also~\cite{GartnerJaggi:2009}, as well as kernel learning (finding the best convex combination among a set of base kernels)~\cite{Bach:2004td}.
Some other applications that directly fit into our framework are $\ell_2$-support vector regression (SVR), AdaBoost~\cite{TongZhang:2003da}, mean-variance analysis in portfolio selection~\cite{Markowitz:1952tg}, the smallest enclosing ball problem~\cite{Badoiu:2007bj}, and estimating mixtures of probability densities~\cite{Clarkson:2010hv}. For more applications we refer to~\cite{Clarkson:2010hv}.

\paragraph{Line-Search for the Best Step-Size.}
In most applications it will be a straight-forward task to find the optimal step-size $\alpha \in [0,1]$ in each step instead, as described in Section~\ref{subsec:lineSearch}.

For the special case of polytope distance and SVM problems, the resulting method then exactly corresponds to Gilbert's geometric algorithm~\cite{gilbert66}, as shown in~\cite{GartnerJaggi:2009}. Here the wording of ``line-search" makes geometric sense in that we need to find the point $s$ on a given line, such that $s$ is closest to the origin.

\paragraph{Away Steps.}
By performing more work with the currently non-zero coordinates, one can get the sparsity even smaller. More precisely the number of non-zeros can be improved close to $\frac{2C_f}{\varepsilon}$ instead of $2\left\lceil\frac{4C_f}{\varepsilon}\right\rceil$ as given by the above Theorem~\ref{thm:sparseSimplex}.
The idea of \emph{away-steps} introduced by~\cite{Todd:2007jl} is to keep the total number of non-zero coordinates (i.e. the coreset size) fixed over all iterations, by removing the smallest non-zero coordinate from $x$ after each adding step.
For more background we refer to~\cite[Algorithm 9.1]{Clarkson:2010hv}.
\subsection{$\Omega(\frac1\varepsilon)$ Lower Bound on the Sparsity}\label{subsec:vecLowerBound}

We will now show that sparsity $O\big(\frac1\varepsilon\big)$, as obtained by the greedy algorithm we analyzed in the previous section is indeed best possible, by providing a lower bound of $\Omega\big(\frac1\varepsilon\big)$. In the language of coresets, this means we will provide a matching lower bound on the size of coresets for convex optimization over the simplex. 
Together with the upper bound, this therefore completely characterizes the trade-off between sparsity and approximation quality for the family of optimization problems of the form~(\ref{eq:optVecSimplex}). The same matching sparsity upper and lower bounds will also hold for optimizing over the $\ell_1$-ball instead, see Section~\ref{sec:vecL1}. 

For the following lower bound construction we consider the differentiable function $f(x) := \norm{x}_2^2 = x^Tx$. This function has gradient $\nabla f(x) = 2x$. Its curvature constant is $C_f = 2$, which follows directly from the definition~(\ref{eq:Cf}), and the fact that here $f(y)-f(x)-(y-x)^T\nabla f(x) = y^Ty-x^Tx-(y-x)^T2x = \norm{x-y}_2^2$, so that
$C_f = \sup_{x,s\in \Delta_n}  \textstyle  \norm{x-s}_2^2 = \diam(\Delta_n)^2 = 2$. 

The following lemmata show that the sparse greedy algorithm of~\cite{Clarkson:2010hv} from Section~\ref{subsec:algSimplex} is indeed optimal for the approximation quality (primal as well as dual error respectively), giving best possible sparsity, up to a small multiplicative constant.

\begin{lemma}\label{lem:vecPrimalLB}
For $f(x) := \norm{x}_2^2$, and $1 \leq k \leq n$, it holds that
$$\min_{\substack{x \in \Delta_n \\
                            \card(x)\le k}}  f(x) = \frac{1}{k}.$$
\end{lemma}
\begin{proof}
We prove the inequality $\displaystyle\min_{x..} f(x) \ge \textstyle\frac1k$ by induction on $k$.

\fbox{Case $k=1$} For any unit length vector $x\in \Delta_n$ having just a single non-zero entry, $f(x) = \norm{x}_2 = \norm{x}_1 = 1$.

\fbox{Case $k>1$} For every $x\in \Delta_n$ of sparsity $\card(x) \le k$, we can pick a coordinate $i$ with $x_i \ne 0$, and write $x = (1-\alpha) v + \alpha \unit_i$ as the sum of two orthogonal vectors $v$ and a unit basis vector $\unit_i$, where $v \in \Delta_n$ of sparsity $\le k-1$, $v_i = 0$,
and $\alpha = x_i$. So for every $x\in \Delta_n$ of sparsity $\le k$, we therefore get that 
\[
\begin{array}{rcl}
f(x) = \norm{x}_2^2 &=& x^Tx\\
  &=& ((1-\alpha) v + \alpha \unit_i)^T ((1-\alpha) v + \alpha \unit_i)\\
  &=& (1-\alpha)^2 v^Tv + \alpha^2\\
  &\ge& (1-\alpha)^2 \frac1{k-1} + \alpha^2\\
  &\ge& \min_{0\leq \beta\leq 1} (1-\beta)^2 \frac{1}{k-1} + \beta^2 \\
  &=& \frac{1}{k}.
\end{array}
\]
In the first inequality we have applied the induction hypothesis for $v\in \Delta_n$ of sparsity $\le k-1$.

Equality: The function value $f(x) = \frac1k$ is obtained by setting $k$ of the coordinates of $x$ to $\frac1k$ each.
\end{proof}

In other words for any vector $x$ of sparsity $\card(x)=k$, the primal error $f(x) - f(x^*)$ is always lower bounded by $\frac1k - \frac1n$. For the duality gap $g(x)$, the lower bound is even slightly higher:

\begin{lemma}
For $f(x) := \norm{x}_2^2$, and any $k \in \N$, $k < n$, it holds that
$$g(x) \ge \frac2k \ \ \ \ \ \ \ \ \  \forall x \in \Delta_n \text{ s.t. } \card(x)\le k.$$
\end{lemma}
\begin{proof}
$g(x) = x^T \nabla f(x) - \min_i (\nabla f(x))_i = 2(x^Tx - \min_i x_i)$. We now use $\min_i x_i = 0$ because $\card(x) < n$, and that by Lemma~\ref{lem:vecPrimalLB} we have $x^Tx = f(x) \ge \frac1k$.
\end{proof}

\paragraph{\textbf{Note:}} We could also consider the function $f(x) := \gamma \norm{x}_2^2$ instead, for some $\gamma > 0$. This $f$ has \emph{curvature} constant $C_f = 2\gamma$, and for this scaling, our above lower bound on the duality gap will also scale linearly, giving~$\frac{C_f}{k}$.

\section{Sparse Approximation with Bounded $\ell_1$-Norm}\label{sec:vecL1}

In this second application case, will apply the general greedy approach from Section~\ref{sec:optGenConvex} in order to understand the best achievable sparsity for convex optimization under bounded $\ell_1$-norm, as a function of the approximation quality. Here the situation is indeed extremely similar to the above Section~\ref{sec:vecSimplex} of optimizing over the simplex, and the resulting algorithm will again have a running time of $O\left(\frac1\varepsilon\right)$ many gradient evaluations.

It is known that the vector $\norm{.}_1$-norm is the best convex approximation to the sparsity (cardinality) of a vector, that is $\card(.)$. More precisely, the function $\norm{.}_1$ is the convex envelope of the sparsity, meaning that it is the ``largest'' convex function that is upper bounded by the sparsity on the convex domain of vectors $\SetOf{x}{\norm{x}_\infty \le 1}$. This can be seen by observing that $\card(x) \ge \frac{\norm{x}_1}{\norm{x}_\infty}$, %
see e.g.~\cite{Recht:2010tf}.
We will discuss the analogous generalization to matrices in the second part of this article, see Section~\ref{chap:optNucMax}, namely using the matrix nuclear norm as the ``best'' convex approximation of the matrix rank.

\paragraph{Set-Up.}
Here we consider one special class of the general optimization problem~(\ref{eq:optGenConvex}), namely problems over vectors in $\R^n$ with bounded $\norm{.}_1$-norm, that is
\begin{equation}\label{eq:optVecL1}
\begin{array}{rl}
   \displaystyle\mini_{x \in \R^n} & f(x)  \\
    s.t. &  \norm{x}_1 \le 1 \ .
\end{array}
\end{equation}

We write $\LOneBall_n := \SetOf{x \in \R^n}{\norm{x}_1 \le 1}$ for the $\ell_1$-ball in $\R^n$. Note that one can simply rescale the function argument to allow for more general constraints $\norm{x}_1 \le t$ for $t>0$.
Again we have $\X = \R^n$ with the standard inner product $\langle x,y\rangle = x^Ty$, so that also the subgradients or gradients of $f$ are represented as vectors in $\R^n$.

\paragraph{The Linearized Problem.}
As already in the simplex case, the subproblem of optimizing a linear function over the $\ell_1$-ball is particularly easy to solve, allowing us to provide a fast implementation of the internal primitive procedure $\exactLin\left(c,\LOneBall_n \right)$. 

Namely, it is again easy to see that every linear function attains its minimum/maximum at a vertex of the ball $\LOneBall_n$, as we have already seen for general convex hulls in our earlier Lemma~\ref{lem:convHull}, and here $\LOneBall_n = \conv(\SetOf{\pm \unit_i}{i\in[n]})$. %
Here this crucial observation can also alternatively be interpreted as the known fact that the dual norm to the $\ell_1$-norm is in fact the $\ell_\infty$-norm, see also our earlier Observation~\ref{obs:gapDualNorm}.

\begin{observation}\label{obs:L1Vertex}
For any vector $c \in\R^n$, it holds that
\[
\unit_i \cdot \sign(c_i) \in \argmax_{y\in\LOneBall_n} \, y^Tc
\]
where $i \in [n]$ is an index of a maximal coordinate of $c$ measured in absolute value, or formally $i \in \argmax_j \,\abs{c_j}$.
\end{observation}

Using this observation for $c=-\nabla f(x)$ in our general Algorithm~\ref{alg:greedyConvex}, we therefore directly obtain the following simple method for $\ell_1$-regularized convex optimization, as depicted in the Algorithm~\ref{alg:greedyL1}.

\begin{algorithm}[h!]
  \caption{Sparse Greedy on the $\ell_1$-Ball}
  \label{alg:greedyL1}
\begin{algorithmic}
  \STATE {\bfseries Input:} Convex function $f$, target accuracy $\varepsilon$
  \STATE {\bfseries Output:} $\varepsilon$-approximate solution for problem~(\ref{eq:optVecL1})
  \STATE Set $x^{(0)}:= \0$
  \FOR{$k=0\dots\infty$}
  \STATE Compute $i := \argmax_i \abs{\left(\nabla f(x^{(k)})\right)_i}$,
  \STATE and let $s:= \unit_i \cdot \sign\left( \left(-\nabla f(x^{(k)})\right)_i \right)$
  \STATE Let $\alpha := \frac2{k+2}$
  \STATE Update $x^{(k+1)}:=x^{(k)}+\alpha(s - x^{(k)})$
  \ENDFOR
\end{algorithmic}
\end{algorithm}

Observe that in each iteration, this algorithm only introduces at most one new non-zero coordinate, so that the sparsity of $x^{(k)}$ is always upper bounded by the number of steps $k$. This means that the method is again of coordinate-descent-type, as in the simplex case of the previous Section~\ref{subsec:algSimplex}. Its convergence analysis again directly follows from the general analysis from Section~\ref{sec:optGenConvex}.

\begin{theorem}[Convergence of Sparse Greedy on the $\ell_1$-Ball]\label{thm:sparseL1}
For each $k\ge1$, the iterate $x^{(k)}$ of Algorithm~\ref{alg:greedyL1} satisfies
\[
f(x^{(k)}) - f(x^*) \le \frac{4C_f}{k+2} \ .
\]
where $x^*\in \LOneBall_n$ is an optimal solution to problem~(\ref{eq:optVecL1}).

Furthermore, for any $\varepsilon>0$, after at most $2\left\lceil\frac{4C_f}{\varepsilon}\right\rceil +1 = O\left(\frac1\varepsilon\right)$ many steps, it has an iterate $x^{(k)}$ of sparsity $O\left(\frac1\varepsilon\right)$, satisfying $g(x^{(k)}) \le\varepsilon$.
\end{theorem}
\begin{proof}
This is a corollary of Theorem~\ref{thm:primalGreedy} and Theorem~\ref{thm:primalDualGreedy}.
\end{proof}

\paragraph{The Duality Gap, and Duality of the Norms.}
We recall the definition of the duality gap~(\ref{eq:defGap}) given by the linearization at any point $x\in\LOneBall_n$, see Section~\ref{sec:poorDual}. Thanks to our Observation~\ref{obs:gapDualNorm}, the computation of the duality gap in the case of the $\ell_1$-ball here becomes extremely simple, and is given by the norm that is dual to the $\ell_1$-norm, namely the $\ell_\infty$-norm of the used subgradient, i.e.,
\[
\begin{split}
g(x,d_x) &= \norm{ d_x }_\infty  + x^T d_x ,\ \ \ \ \text{ and}\\
g(x) &= \norm{ \nabla f(x) }_\infty + x^T \nabla f(x) \ .
\end{split}
\]
Alternatively, the same expression can also be derived directly (without explicitly using duality of norms) by applying the Observation~\ref{obs:L1Vertex}.

\paragraph{A Lower Bound on the Sparsity.}
The lower bound of $\Omega\left(\frac1\varepsilon\right)$ on the sparsity as proved in Section~\ref{subsec:vecLowerBound} for the simplex case in fact directly translates to the $\ell_1$-ball as well. Instead of choosing the objective function $f$ as the distance to the origin (which is part of the $\ell_1$-ball), we consider the optimization problem $\displaystyle\min_{\norm{x}_1 \le 1} f(x) := \norm{x-r}_2^2$ with respect to the fixed point $r := (\frac2n,\dots,\frac2n) \in \R^n$. This problem is of the form~(\ref{eq:optVecL1}), and corresponds to optimizing the Euclidean distance to the point $r$ given by mirroring the origin at the positive facet of the $\ell_1$-ball. Here by the ``positive facet'', we mean the hyperplane defined by the intersection of the boundary of the $\ell_1$-ball with the positive orthant, which is exactly the unit simplex. Therefore, the proof for the simplex case from Section~\ref{subsec:vecLowerBound} holds analogously for our setting here.

We have thus obtained that sparsity $O\left(\frac1\varepsilon\right)$ as obtained by the greedy Algorithm~\ref{alg:greedyL1} is indeed best possible for $\ell_1$-regularized optimization problems of the form~(\ref{eq:optVecL1}).

\paragraph{Using Barycentric Coordinates Instead.} 
Clarkson~\cite[Theorem 4.2]{Clarkson:2010hv} already observed that  Algorithm~\ref{alg:greedySimplex} over the simplex $\Delta_n$ can be used to optimize a convex function $f(y)$ over arbitrary convex hulls, by just using barycentric coordinates $y = Ax, \ x\in \Delta_n$, for $A\in\R^{n \times m}$ being the matrix containing all $m$ vertices of the convex domain as its columns. Here however we saw that for the $\ell_1$-ball, the steps of the algorithm are even slightly simpler, as well as that the duality gap can be computed instantly from the $\ell_\infty$-norm of the gradient.

\paragraph{Applications.}
Our Algorithm~\ref{alg:greedyL1} applies to arbitrary convex vector optimization problems with an $\norm{.}_1$-norm regularization term, giving a guaranteed sparsity of $O\left(\frac1\varepsilon\right)$ for all these applications.

A classical example for problems of this class is given by the important \emph{$\norm{.}_1$-regularized least squares} regression approach, i.e.
\[
\min_{x \in \R^{n}} \norm{Ax-b}_{2}^{2} + \mu \norm{x}_{1} 
\] 
for a fixed matrix $A\in\R^{m\times n}$, a vector $b\in\R^m$ and a fixed regularization parameter $\mu>0$. The same problem is also known as \emph{basis pursuit de-noising} in the compressed sensing literature, which we will discuss more precisely in Section~\ref{sec:relComprSensing}.
The above formulation is in fact the Lagrangian formulation of the corresponding constrained problem for $\norm{x}_1\le t$ for some fixed parameter $t$ corresponding to $\mu$. This equivalent formulation is also known as the \emph{Lasso} problem~\cite{Tibshirani:1996wb} which is 
\[
\begin{array}{rl}
\displaystyle\min_{x \in \R^{n}} & \norm{Ax-b}_{2}^{2} \\
                       \st &\norm{x}_{1} \le t \ .
\end{array}
\] 
The above formulation is exactly a problem of our above form~(\ref{eq:optVecL1}), namely 
\[
\min_{\hat x \in \LOneBall_n} \norm{tA\hat x-b}_{2}^{2} \ ,
\] 
if we rescale the argument $x =: t\hat x$ so that $\norm{\hat x}_1\le1$.

Another important application for our result is \emph{logistic regression} with $\norm{.}_1$-norm regularization, see e.g.~\cite{Koh:2007wo}, which is also a convex optimization problem~\cite{Rennie:2005ww}. The reduction to an $\ell_1$-problem of our form~(\ref{eq:optVecL1}) works exactly the same way as described here.

\paragraph{Related Work.}
As we mentioned above, the optimization problem~(\ref{eq:optVecL1}) --- if $f$ is the squared error of a linear function --- is very well studied as the \emph{Lasso} approach, see e.g.~\cite{Tibshirani:1996wb} and the references therein.
For general objective functions $f$ of bounded curvature, the above interesting trade-off between sparsity and the approximation quality was already investigated by~\cite{ShalevShwartz:2010wq}, and also by our earlier paper~\cite{GartnerJaggi:2009} for the analogous case of optimizing over the simplex.
\cite[Theorem 2.4]{ShalevShwartz:2010wq} shows a sparse convergence analogous to our above Theorem~\ref{thm:sparseL1}, for the ``forward greedy selection'' algorithm on problem~(\ref{eq:optVecL1}), but only for the case that $f$ is differentiable.%

\subsection{Relation to Matching Pursuit and Basis Pursuit in Compressed Sensing}\label{sec:relComprSensing}
Both our sparse greedy Algorithm~\ref{alg:greedySimplex} for optimizing over the simplex and also Algorithm~\ref{alg:greedyL1} for general $\ell_1$-problems are 
very similar to the technique of \emph{matching pursuit}, which is one of the most popular techniques in sparse recovery in the vector case~\cite{Tropp:2004gc}.

Suppose we want to recover a sparse signal vector $x\in\R^n$ from a noisy measurement vector $Ax=y \in\R^m$. For a given dictionary matrix $A\in\R^{m\times n}$, matching pursuit iteratively chooses the dictionary element $A_i\in\R^m$ that has the highest inner product with the current residual, and therefore reduces the representation error $f(x) = \norm{Ax-y}_2^2$ by the largest amount. 
This choice of coordinate $i=\argmax_j A_j^T (Ax-y)$ exactly corresponds\footnote{%
The objective function $f(x) := \norm{Ax-y}_2^2$ can be written as $f(x) = (Ax-y)^T(Ax-y)=x^TA^TAx-2y^TAx-y^Ty$, so its gradient is  $\nabla f(x) = 2A^TAx-2A^Ty = 2A^T(Ax-y) \in \R^n$.
} %
to the choice of $i := \argmin_j \left(\nabla f(x^{(k)})\right)_j$ in Algorithm~\ref{alg:greedySimplex}.

Another variant of matching pursuit, called orthogonal matching pursuit (OMP)~\cite{Tropp:2004gc,Tropp:2007he}, includes an extra orthogonalization step, and is closer related to the coreset algorithms that optimize over the all existing set of non-zero coordinates before adding a new one, see e.g.~\cite[Algorithm 8.2]{Clarkson:2010hv}, or the analogous ``fully corrective'' variant of~\cite{ShalevShwartz:2010wq}. If $y=Ax$, with $x$ sparse and the columns of $A$ sufficiently incoherent, then OMP recovers the sparsest representation for the given~$y$~\cite{Tropp:2004gc}.

The paper \cite{Zhang:2011uy} recently proposed another algorithm that generalizes OMP, comes with a guarantee on correct sparse recovery, and also corresponds to ``completely optimize within each coreset''. The method uses the same choice of the new coordinate $i := \argmax_j \abs{\left(\nabla f(x^{(k)})\right)_j}$ as in our Algorithm~\ref{alg:greedyL1}.
However the analysis of~\cite{Zhang:2011uy} requires the not only bounded curvature as in our case, but also needs strong convexity of the objective function (which then also appears as a multiplicative factor in the number of iterations needed). Our Algorithm~\ref{alg:greedyL1} as well as the earlier method by~\cite{TongZhang:2003da} are simpler to implement, and have a lower complexity per iteration, as we do not need to optimize over several currently non-zero coordinates, but only change one coordinate by a fixed amount in each iteration.

Our Algorithm~\ref{alg:greedyL1} for general $\ell_1$-regularized problems also applies to solving the so called \emph{basis pursuit} problem~\cite{Chen:1998hm,Figueiredo:2007hz} and~\cite[Section 6.5.4]{Boyd:2004uz}, which is $\min_{x\in\R^n} \norm{x}_1$ s.t. $Ax=y$. %
Note that this is in fact just the constrained variant of the corresponding ``robust'' $\ell_{1}$-regularized least squares regression problem
\[
\min_{x\in\R^n} \norm{Ax-y}^2_2 + \mu \norm{x}_1 \ ,
\]
which is the equivalent trade-off variant of our problem of the form~(\ref{eq:optVecL1}). \cite{Figueiredo:2007hz} propose a traditional gradient descent technique for solving the above least squares problem, but do not give a convergence analysis. 

Solution path algorithms with approximation guarantees for related problems (obtaining solutions for all values of the tradeoff parameter $\mu$) have been studied in~\cite{Giesen:2010fx,Giesen:2012uj,Giesen:2012vh}, and the author's PhD thesis~\cite{Jaggi:2011ux}, building on the same duality gap concept we introduced in Section~\ref{sec:poorDual}.

\section{Optimization with Bounded $\ell_{\infty}$-Norm}\label{sec:vecLinf}

Applying our above general optimization framework for the special case of the domain being the $\norm{.}_\infty$-norm unit ball, we again obtain a very simple greedy algorithm. The running time will again correspond to $O\left(\frac1\varepsilon\right)$ many gradient evaluations. Formally, we consider problems of the form
\begin{equation}\label{eq:optVecLinfLe}
\begin{array}{rl}
   \displaystyle\mini_{x\in\R^n} & f(x)  \\
    s.t. &  \norm{x}_\infty \le 1 \ .
\end{array}
\end{equation}

We denote the feasible set, i.e. the $\norm{.}_\infty$-norm unit ball, by $\Box_n := \SetOf{x \in \R^n}{\norm{x}_\infty \le 1}$.
For this set, it will again be very simple to implement the internal primitive operation of optimizing a linear function over the same domain.
The following crucial observation allows us to implement $\exactLin\left(c,\Box_n \right)$ in a very simple way. This can also alternatively be interpreted as the known fact that the dual-norm to the $\ell_\infty$-norm is the $\ell_1$-norm, which also explains why the greedy algorithm we will obtain here is very similar to the $\ell_1$-version from the previous Section~\ref{sec:vecL1}.

\begin{observation}\label{obs:LinfPrimitive}
For any vector $c \in\R^n$, it holds that
\[
\signVec^c  \in  \argmax_{y\in\Box_n} \, y^Tc
\]
where $\signVec^c \in \R^n$ is the sign-vector of $c$, defined by the sign of each individual coordinate, i.e. $(\signVec^c)_i = \sign(c_i) \in \{-1,1\}$. 
\end{observation}

Using this observation for $c=-d_x$ in our general Algorithm~\ref{alg:greedyConvex}, we directly obtain the following simple method for optimization over a box-domain $\Box_n$, as depicted in Algorithm~\ref{alg:greedyBox}.

\begin{algorithm}[h!]
  \caption{Sparse Greedy on the Cube}
  \label{alg:greedyBox}
\begin{algorithmic}
  \STATE {\bfseries Input:} Convex function $f$, target accuracy $\varepsilon$
  \STATE {\bfseries Output:} $\varepsilon$-approximate solution for problem~(\ref{eq:optVecLinfLe})
  \STATE Set $x^{(0)}:= \zero$
  \FOR{$k=0\dots\infty$}
  \STATE Compute the sign-vector $\signVec$ of $\nabla f(x^{(k)})$, such that
  \STATE \quad\quad $\signVec_i = \sign\left(\left(-\nabla f(x^{(k)})\right)_i\right),$ $~~~i=1..n$
  \STATE Let $\alpha := \frac2{k+2}$
  \STATE Update $x^{(k+1)}:=x^{(k)}+\alpha(\signVec - x^{(k)})$
  \ENDFOR
\end{algorithmic}
\end{algorithm}

The convergence analysis again directly follows from the general analysis from Section~\ref{sec:optGenConvex}.
\begin{theorem}
For each $k\ge1$, the iterate $x^{(k)}$ of Algorithm~\ref{alg:greedyBox} satisfies
\[
f(x^{(k)}) - f(x^*) \le \frac{4C_f}{k+2} \ .
\]
where $x^*\in \Box_n$ is an optimal solution to problem~(\ref{eq:optVecLinfLe}).

Furthermore, for any $\varepsilon>0$, after at most $2\left\lceil\frac{4C_f}{\varepsilon}\right\rceil +1 = O\left(\frac1\varepsilon\right)$ many steps, it has an iterate $x^{(k)}$ with $g(x^{(k)}) \le\varepsilon$.
\end{theorem}
\begin{proof}
This is a corollary of Theorem~\ref{thm:primalGreedy} and Theorem~\ref{thm:primalDualGreedy}.
\end{proof}

\paragraph{The Duality Gap, and Duality of the Norms.}
We recall the definition of the duality gap~(\ref{eq:defGap}) given by the linearization at any point $x\in\Box_n$, see Section~\ref{sec:poorDual}. Thanks to our Observation~\ref{obs:gapDualNorm}, the computation of the duality gap in the case of the $\ell_\infty$-ball here becomes extremely simple, and is given by the norm that is dual to the $\ell_\infty$-norm, namely the $\ell_1$-norm of the used subgradient, i.e.,
\[
\begin{split}
g(x,d_x) &= \norm{ d_x }_1  + x^T d_x ,\ \ \ \ \text{ and}\\
g(x) &= \norm{ \nabla f(x) }_1 + x^T \nabla f(x) \ .
\end{split}
\]
Alternatively, the same expression can also be derived directly (without explicitly using duality of norms) by applying the Observation~\ref{obs:LinfPrimitive}.

\paragraph{Sparsity and Compact Representations.} The analogue of ``sparsity'' as in Sections~\ref{sec:vecSimplex} and~\ref{sec:vecL1} in the context of our Algorithm~\ref{alg:greedyBox} means that we can describe the obtained approximate solution $x$ as a convex combination of few (i.e. $O(\frac1\varepsilon)$ many) cube vertices. This does not imply that $x$ has few non-zero coordinates, but that we have a compact representation given by only $O(\frac1\varepsilon)$ many binary $n$-vectors indicating the corresponding cube vertices, of which $x$ is a convex combination.

\paragraph{Applications.} Any convex problem under coordinate-wise upper and lower constraints can be transformed to the form~(\ref{eq:optVecLinfLe}) by re-scaling the optimization argument.
A specific interesting application was given by~\cite{Mangasarian:2011ib}, who have demonstrated that integer linear programs can be relaxed to convex problems of the above form, such that the solutions coincide with high probability under some mild additional assumptions.

\paragraph{Using Barycentric Coordinates Instead.} 
Clarkson~\cite[Theorem 4.2]{Clarkson:2010hv} already observed that  Algorithm~\ref{alg:greedySimplex} over the simplex $\Delta_n$ can be used to optimize a convex function $f(y)$ over arbitrary convex hulls, by just using barycentric coordinates $y = Ax, \ x\in \Delta_n$, for $A\in\R^{n \times m}$ being the matrix containing all $m$ vertices of the convex domain as its columns. Here however we saw that for the unit box, the steps of the algorithm are much simpler, as well as that the duality gap can be computed instantly, without having to explicitly deal with the exponentially many vertices (here $m=2^n$) of the cube.

\section{Semidefinite Optimization with Bounded Trace}\label{sec:algTrace}

We will now apply the greedy approach from the previous Section~\ref{sec:optGenConvex} to semidefinite optimization problems, for the case of bounded trace.
The main paradigm in this section will be to understand the best achievable low-rank property of approximate solutions as a function of the approximation quality. 

In particular, we will show that our general Algorithm~\ref{alg:greedyConvex} and its analysis do lead to Hazan's method for convex semidefinite optimization with bounded trace, as given by~\cite{Hazan:2008kz}. Hazan's algorithm can also be used as a simple solver for general SDPs.
\cite{Hazan:2008kz} has already shown that guaranteed $\varepsilon$-approximations of rank $O\left(\frac{1}{\varepsilon}\right)$ can always be found. Here we will also show that this is indeed optimal, by providing an asymptotically matching lower bound in Section~\ref{sec:matLowerBound}.
Furthermore, we fix some problems in the original analysis of~\cite{Hazan:2008kz}, and require only a weaker approximation quality for the internal linearized primitive problems. We also propose two improvement variants for the method in Section~\ref{subsec:imprAlgos}.

Later in Section~\ref{chap:optNucMax}, we will discuss the application of these algorithms for nuclear norm and max-norm optimization problems, which have many important applications in practice, such as dimensionality reduction, low-rank recovery as well as matrix completion and factorizations.

We now consider convex optimization problems of the form~(\ref{eq:optGenConvex}) over the space $\X = \Sym^{n\times n}$ of symmetric matrices, equipped with the standard Frobenius inner product $\langle X,Y \rangle = X\bullet Y$. It is left to the choice of the reader to identify the symmetric matrices either with $\R^{n^2}$ and consider functions with $f(X) = f(X^T)$, or only ``using'' the variables in the upper right (or lower left) triangle, corresponding to $\R^{n(n+1)/2}$.
In any case, the subgradients or gradients of our objective function $f$ need to be available in the same representation (same choice of basis for the vector space $\X$).
 
Formally, we consider the following special case of the general optimization problems~(\ref{eq:optGenConvex}), i.e.,
\begin{equation}\label{eq:optTrEq1}
\begin{array}{rl}
   \displaystyle\mini_{X \in \Sym^{n\times n}} & f(X)  \\
    s.t. &  \tr(X) = 1 \ ,\\
         &  X \succeq 0
\end{array}
\end{equation}
We will write $\Spectahedron := \SetOf{X \in \Sym^{n \times n}}{X \succeq 0, ~ \tr(X) = 1}$ for the feasible set, that is the PSD matrices of unit trace. The set $\Spectahedron$ is sometimes called the \emph{Spectahedron}, and can be seen as a natural generalization of the unit simplex to symmetric matrices. By the Cholesky factorization, it can be seen that the Spectahedron is the convex hull of all rank-1 matrices of unit trace (i.e. the matrices of the form $vv^T$ for a unit vector $v\in \R^n$, $\norm{v}_2=1$). %

\subsection{Low-Rank Semidefinite Optimization with Bounded Trace: The $O(\frac1\varepsilon)$ Algorithm by Hazan}\label{subsec:algHazan}

Applying our general greedy Algorithm~\ref{alg:greedyConvex} that we studied in Section~\ref{sec:optGenConvex} to the above semidefinite optimization problem, we directly obtain the following Algorithm~\ref{alg:greedyHazan}, which is Hazan's method~\cite{Hazan:2008kz,Gartner:2011tl}.

Note that this is now a first application of Algorithm~\ref{alg:greedyConvex} where the internal linearized problem $\approxLin()$ is not trivial to solve, contrasting the applications for vector optimization problems we studied above.  The algorithm here obtains low-rank solutions (sum of rank-1 matrices) to any convex optimization problem of the form~(\ref{eq:optTrEq1}). More precisely, it guarantees $\varepsilon$-small duality gap after at most $O\left(\frac{1}{\varepsilon}\right)$ iterations, where each iteration only involves the calculation of a single approximate eigenvector of a matrix $M \in \Sym^{n\times n}$. We will see that in practice for example Lanczos' or the power method can be used as the internal optimizer $\approxLin()$.

\begin{algorithm}[h!]
  \caption{Hazan's Algorithm / Sparse Greedy for Bounded Trace}
  \label{alg:greedyHazan}
\begin{algorithmic}
  \STATE {\bfseries Input:} Convex function $f$ with curvature constant $C_{f}$, target accuracy $\varepsilon$
  \STATE {\bfseries Output:} $\varepsilon$-approximate solution for problem~(\ref{eq:optTrEq1})
  \STATE Set $X^{(0)}:= vv^T$ for an arbitrary unit length vector $v \in \R^n$.
  \FOR{$k=0\dots\infty$}
  \STATE Let $\alpha := \frac2{k+2}$
  \STATE Compute $v := v^{(k)} = \mbox{ApproxEV}\left(\nabla f(X^{(k)}), \alpha  C_f\right)$
  \STATE Update $X^{(k+1)}:=X^{(k)}+\alpha(vv^T - X^{(k)})$
  \ENDFOR
\end{algorithmic}
\end{algorithm}

Here $\textsc{ApproxEV}(A, \varepsilon')$ is a subroutine that delivers an approximate smallest eigenvector (the eigenvector corresponding to the smallest eigenvalue) to a matrix $A$ with the desired accuracy $\varepsilon' > 0$. More precisely, it must return a unit length vector $v$ such that $v^T A v \le \lmin(A) + \varepsilon'$. Note that as our convex function $f$ takes a symmetric matrix $X$ as an argument, its gradients $\nabla f(X)$ are given as symmetric matrices as well.

If we want to understand this proposed Algorithm~\ref{alg:greedyHazan} as an instance of the general convex optimization Algorithm~\ref{alg:greedyConvex}, we just need to explain why the largest eigenvector should indeed be a solution to the internal linearized problem $\approxLin()$, as required in Algorithm~\ref{alg:greedyConvex}. Formally, we have to show that $v:= \textsc{ApproxEV}(A, \varepsilon')$ does approximate the linearized problem, that is
\[
vv^T \bullet A \le \displaystyle\min_{Y\in \Spectahedron} \, Y \bullet A + \varepsilon'
\]
for the choice of $v:= \textsc{ApproxEV}(A, \varepsilon')$, and any matrix $A\in \Sym^{n\times n}$.

This fact is formalized in Lemma~\ref{lem:minAtVtxOfSpect} below, and
will be the crucial property enabling the fast implementation of Algorithm~\ref{alg:greedyHazan}. 

Alternatively, if exact eigenvector computations are available, we can also implement the exact variant of  Algorithm~\ref{alg:greedyConvex} using $\exactLin()$, thereby halving the total number of iterations.

Observe that an approximate eigenvector here is significantly easier to compute than a projection onto the feasible set $\Spectahedron$. If we were to find the $\frobnorm{.}$-closest PSD matrix to a given symmetric matrix $A$, we would have to compute a complete eigenvector decomposition of $A$, and only keeping those corresponding to positive eigenvalues, which is computationally expensive. By contrast, a single approximate smallest eigenvector computation as in $\textsc{ApproxEV}(A, \varepsilon')$ can be done in near linear time in the number of non-zero entries of $A$. We will discuss the implementation of $\textsc{ApproxEV}(A, \varepsilon')$ in more detail further below.

\paragraph{Sparsity becomes Low Rank.}
As the rank-1 matrices are indeed the ``vertices'' of the domain $\Spectahedron$ as shown in Lemma~\ref{lem:minAtVtxOfSpect} below, our Algorithm~\ref{alg:greedyHazan} can be therefore seen as a matrix generalization of the sparse greedy approximation algorithm  of~\cite{Clarkson:2010hv} for vectors in the unit simplex, see Section~\ref{sec:vecSimplex}, which has seen many successful applications.
Here \emph{sparsity} just gets replaced by \emph{low rank}. By the analysis of the general algorithm in Theorem~\ref{thm:primalGreedy}, we already know that we obtain $\varepsilon$-approximate solutions for any convex optimization problem~(\ref{eq:optTrEq1}) over the spectahedron $\Spectahedron$. Because each iterate $X^{(k)}$ is represented as a sum (convex combination) of $k$ many rank-$1$ matrices $vv^T$, it follows that $X^{(k)}$ is of rank at most $k$. Therefore, the resulting $\varepsilon$-ap\-prox\-i\-ma\-tions are of low rank, i.e. rank~$O\big(\frac1\varepsilon\big)$. %

For large-scale applications where $\frac1\varepsilon \ll n$, the representation of $X^{(k)}$ as a sum of rank-1 matrices is much more efficient than storing an entire matrix $X^{(k)} \in \Sym^{n\times n}$. Later in Section~\ref{sec:applicationMatCompl} (or see also~\cite{Jaggi:2010tz}) we will demonstrate that Algorithm~\ref{alg:greedyHazan} can readily be applied to practical problems for $n \ge 10^6$ on an ordinary computer,  well exceeding the possibilities of interior point methods.

\cite{Hazan:2008kz} already observed that the same Algorithm~\ref{alg:greedyHazan} with a well-crafted function $f$ can also be used to approximately solve arbitrary SDPs with bounded trace, which we will briefly explain in Section~\ref{subsec:arbitrSDP}.

\paragraph{Linearization, the Duality Gap, and Duality of the Norms.}
Here we will prove that the general duality gap~(\ref{eq:defGap}) can be calculated very efficiently for the domain being the spectahedron $\Spectahedron$. From the following Lemma~\ref{lem:minAtVtxOfSpect}, we obtain that
\begin{equation}\label{eq:gapTrace}
\begin{split}
g(X) =& X\bullet \nabla f(X) + \lmax(-\nabla f(X)) \\
       =& X\bullet \nabla f(X) - \lmin(\nabla f(X)) \ .
\end{split}
\end{equation}
As predicted by our Observation~\ref{obs:gapDualNorm} on formulating the duality gap, we have again obtained the \emph{dual norm} to the norm that determines the domain $D$. It can be seen that over the space of symmetric matrices, the dual norm of the matrix trace-norm (also known as the nuclear norm) is given by the spectral norm, i.e. the largest eigenvalue. To see this, we refer the reader to the later Section~\ref{sec:nucNorm} on the properties of the nuclear norm and its dual characterization.

The following Lemma~\ref{lem:minAtVtxOfSpect} shows that any linear function attains its minimum and maximum at a ``vertex'' of the Spectahedron $\Spectahedron$, as we have already proved for the case of general convex hulls in Lemma~\ref{lem:convHull}.

\begin{lemma}\label{lem:minAtVtxOfSpect}
$\!$The spectahedron is the convex hull of the rank-1 matrices,
\[
\Spectahedron = \conv(\SetOf{vv^T}{v \in \R^n, \norm{v}_2 = 1}) \ .
\]
Furthermore, for any symmetric matrix $A \in \Sym^{n\times n}$, it holds that
\[
  \max_{X \in \Spectahedron} A \bullet X   =  \lmax(A) \ .
\]
\end{lemma}
\begin{proof}
  Clearly, it holds that $vv^T \in \Spectahedron$ for any unit length vector $v\in\R^n$, as $\tr(vv^T) = \norm{v}_2^2$.
  To prove the other inclusion, we consider an arbitrary matrix $X \in \Spectahedron$, and let $X = U^T U$ be its Cholesky factorization.
  We let $\alpha_i$ be the squared norms of the rows of $U$, and let $u_i$ be the row vectors of $U$, scaled to unit length. From the observation $1= \tr(X) = \tr(U^T U) = \tr(U U^T) = \sum_{i} \alpha_i$ it follows that any $X \in \Spectahedron$ can be written as a convex combination of at most $n$ many rank-$1$ matrices $X = \sum_{i=1}^{n} \alpha_i u_i u_i^T$ with unit vectors $u_i \in \R^n$, proving the first part of the claim. Furthermore, this implies that we can write
\[
    \max_{X \in \Spectahedron} \ A \bullet X 
    = \max_{u_i,\alpha_i} \ A \bullet \sum_{i=1}^{n} \alpha_i  u_i u_i^T
    = \max_{u_i,\alpha_i} \ \sum_{i=1}^{n} \alpha_i (A \bullet u_i u_i^T),
\]
  where the maximization $\max_{u_i,\alpha_i}$ is taken over unit vectors $u_i \in
  \R^n$, $\norm{u_i} = 1$, for $1 \le i \le n$, and real
  coefficients $\alpha_i \ge 0$, with $\sum_{i=1}^{n} \alpha_i = 1$. 
  Therefore
\[
\begin{array}{rll}\displaystyle
    \max_{X \in \Spectahedron} A \bullet X
    &= \displaystyle\max_{u_i,\alpha_i} \ \sum_{i=1}^{n} \alpha_i (A \bullet u_i u_i^T) \\
    &= \displaystyle \max_{v \in \R^n, \norm{v} = 1} A \bullet v v^T \\
    &= \displaystyle \max_{v \in \R^n, \norm{v} = 1} v^T A v \\
    &= \displaystyle  \lmax \left(A\right),
\end{array}
\]
  where the last equality is the variational characterization of the largest eigenvalue.
\end{proof}

\paragraph{Curvature.}
We know that the constant in the actual running time for a given convex function $f: \Sym^{d \times d} \rightarrow \R$ is given by the \emph{curvature constant} $C_{f}$ as given in~(\ref{eq:Cf}), which for the domain $\Spectahedron$ becomes
\begin{equation}\label{eq:CfTrace}
C_f := \sup_{\substack{X,V \in \Spectahedron, \, \alpha\in[0,1], \\
                    Y=X+\alpha(V-X)}}
            \textstyle\frac{1}{\alpha^{2}} \big(f(Y)-f(X)+(Y-X) \bullet \nabla f(X) \big) \ .
\end{equation}

\paragraph{Convergence.}
We can now see the convergence analysis for Algorithm~\ref{alg:greedyHazan} following directly as a corollary of our simple analysis of the general framework in Section~\ref{sec:optGenConvex}. The following theorem proves that $O\big(\frac1\varepsilon\big)$ many iterations are sufficient to obtain primal error $\le \varepsilon$.
This result was already known in~\cite[Theorem 1]{Hazan:2008kz}, or~\cite[Chapter 5]{Gartner:2011tl} where some corrections to the original paper were made.

\begin{theorem}\label{thm:algTrace}
For each $k\ge1$, the iterate $X^{(k)}$ of Algorithm~\ref{alg:greedyHazan} satisfies
\[
f(X^{(k)}) - f(X^*) \le \frac{8C_f}{k+2} \ .
\]
where $X^*\in \Spectahedron$ is an optimal solution to problem~(\ref{eq:optTrEq1}).

Furthermore, for any $\varepsilon>0$, after at most $2\left\lceil\frac{8C_f}{\varepsilon}\right\rceil +1 = O\left(\frac1\varepsilon\right)$ many steps, it has an iterate $X^{(k)}$ of rank $O\left(\frac1\varepsilon\right)$, satisfying $g(X^{(k)}) \le\varepsilon$.
\end{theorem}
\begin{proof}
This is a corollary of Theorem~\ref{thm:primalGreedy} and Theorem~\ref{thm:primalDualGreedy}.
\end{proof}

\paragraph{Approximating the Largest Eigenvector.}
Approximating the smallest eigenvector of a symmetric matrix $\nabla f(X)$ (which is the largest eigenvector of $-\nabla f(X)$) is a well-studied problem in the literature.
We will see in the following that the internal procedure $\mbox{ApproxEV}(M, \varepsilon')$, can be performed in \emph{near-linear} time, when measured in the number of non-zero entries of the gradient matrix $\nabla f(X)$. This will follow from the analysis of~\cite{Kuczynski:1992va} for the power method or Lanczos' algorithm, both with a random start vector. A similar statement has been used in~\cite[Lemma 2]{Arora:2005uh}.

\begin{theorem}\label{thm:powerLanczos}
Let $M\in \Sym^{n\times n}$ be a positive semidefinite matrix. Then with high probability, both
  \begin{itemize}
    \item[i)] $O\left(\frac{\log(n)}{\gamma}\right)$ iterations of the power method, or  
    \item[ii)] $O\left(\frac{\log(n)}{\sqrt{\gamma}}\right)$ iterations of Lanczos' algorithm
  \end{itemize}
will produce a unit vector $x$ such that $\frac{x^TMx}{\lambda_1(M)} \ge 1-\gamma$.
\end{theorem}
\begin{proof}
The statement for the power method follows from~\cite[Theorem 3.1(a)]{Kuczynski:1992va}, and for Lanczos' algorithm by~\cite[Theorem 3.2(a)]{Kuczynski:1992va}.
\end{proof}

The only remaining obstacle to use this result for our internal procedure $\mbox{ApproxEV}(M, \varepsilon')$ is that our gradient matrix $M=-\nabla f(X)$ is usually not PSD. However, this can easily be fixed by adding a large enough constant $t$ to the diagonal, i.e. $\hat M := M + t\id$, or in other words shifting the spectrum of $M$ so that the eigenvalues satisfy $\lambda_i(\hat M) = \lambda_i(M) + t \ge 0$ $\forall i$. The choice of $t= - \lmin(M)$ is good enough for this to hold.

Now by setting $\gamma:= \frac{\varepsilon'}{L} \le\frac{\varepsilon'}{\lmax(\hat M)}$ for some upper bound $L \ge \lmax(\hat M)=\lmax(M)-\lmin(M)$, this implies that our internal procedure $\mbox{ApproxEV}(M, \varepsilon')$ can be implemented by performing $O\left(\frac{\log(n)\sqrt{L}}{\sqrt{\varepsilon'}}\right)$ many Lanczos steps (that is matrix-vector multiplications). 
Note that a simple choice for $L$ is given by the spectral norm of $M$, since $2\spectnorm{M} = 2\max_i \lambda_i(M) \ge \lmax(M)-\lmin(M)$.
We state the implication for our algorithm in the following corollary.

\begin{theorem}
For $M\in\Sym^{n\times n}$, and $\varepsilon'>0$, the procedure $\mbox{ApproxEV}(M, \varepsilon')$ requires a total of $O\left(N_f\frac{\log(n)\sqrt{L}}{\sqrt{\varepsilon'}}\right)%
$ many arithmetic operations, with high probability, by using Lanczos' algorithm.
\end{theorem}
Here $N_f$ is the number of non-zero entries in $M$, which in the setting of Algorithm~\ref{alg:greedyHazan} is the gradient matrix $-\nabla f(X)$. %
We have also assumed that the spectral norm of $M$ is bounded by $L$. %

Since we already know the number of necessary ``outer'' iterations of Algorithm~\ref{alg:greedyHazan}, by Theorem~\ref{thm:algTrace}, we conclude with the following analysis of the total running time. Here we again use that the required internal accuracy is given by $\varepsilon' = \alpha C_f \le \varepsilon C_f$.

\begin{corollary}\label{cor:algHazanRunningTime}
When using Lanczos' algorithm for the approximate eigenvector procedure $\mbox{ApproxEV}(.,.)$, then Algorithm~\ref{alg:greedyHazan} 
provides an $\varepsilon$-approximate solution in $O\left(\frac{1}{\varepsilon}\right)$ iterations, requiring a total of $\tilde
  O\left(\frac{N_f}{\varepsilon^{1.5}}\right)$ 
  arithmetic operations (with high probability).
\end{corollary}

Here the notation $\tilde O(.)$ suppresses the logarithmic factor in $n$.
This corollary improves the original analysis of~\cite{Hazan:2008kz} by a factor of~$\frac1{\sqrt{\varepsilon}}$, since~\cite[Algorithm 1]{Hazan:2008kz} as well as the proof of~\cite[Theorem 1]{Hazan:2008kz} used an internal accuracy bound of $\varepsilon' = O\big(\frac1{k^2}\big)$ instead of the sufficient choice of $\varepsilon' = O\big(\frac1k\big)$ as in our general analysis here.

\paragraph{Representation of the Estimate $X$ in the Algorithm.}
The above result on the total running time assumes the following: After having obtained an approximate eigenvector $v$, the rank-1 update $X^{(k+1)}:=(1-\alpha)X^{(k)}+\alpha vv^T$ can be performed efficiently, or more precisely in time $N_f$. In the worst case, when a fully dense matrix $X$ is needed, this update cost is $N_f = n^2$. However, there are many interesting applications where the function~$f$ depends only on a small fraction of the entries of~$X$, so that $N_f \ll n^2$. Here, a prominent example is matrix completion for recommender systems. In this case, only those~$N_f$ many entries of~$X$ will be stored and affected by the rank-1 update, see also our Section~\ref{sec:applicationMatCompl}.

An alternative representation of $X$ consists of the low-rank factorization, given by the $v$-vectors of each of the $O\big(\frac1\varepsilon\big)$ many update steps, using a smaller memory of size $O\big(\frac n\varepsilon\big)$. However, computing the gradient $\nabla f(X)$ from this representation of $X$ might require more time then.

\subsection{Solving Arbitrary SDPs}\label{subsec:arbitrSDP}

In~\cite{Hazan:2008kz} it was established that Algorithm~\ref{alg:greedyHazan} can also be used to approximately solve arbitrary semidefinite programs (SDPs) in feasibility form, i.e.,
\begin{equation}\label{eq:feasSDP}
\begin{array}{rl}
\text{find } X \ \st & A_i \bullet X \le b_i  ~~~~i=1..m \\
    & X \succeq 0  \ .
\end{array}
\end{equation}
Also every classical SDP with a linear objective function 
\begin{equation}\label{eq:SDP}
\begin{array}{rl}
\displaystyle\maxi_X& C \bullet X \\
\st & A_i \bullet X \le b_i  ~~~~i=1..m' \\
    & X \succeq 0  \ .
\end{array}
\end{equation}
can be turned into a feasibility SDP~(\ref{eq:feasSDP}) by ``guessing'' the optimal value $C\bullet X$ by binary search~\cite{Arora:2005uh,Hazan:2008kz}.

Here we will therefore assume that we are given a feasibility SDP of the form ~(\ref{eq:feasSDP}) by its constraints $A_i \bullet X \le b_i$, which we want to solve for $X$. We can represent the constraints of~(\ref{eq:feasSDP}) in a smooth optimization objective instead, using the \emph{soft-max} function
\begin{equation}\label{eq:softMax}
f(X) := \frac1\sigma \log \left( \sum_{i=1}^{m} e^{\sigma(A_i \bullet X - b_i)} \right) \ .
\end{equation}

Suppose that the original SDP was feasible, then after $O\left(\frac{1}{\varepsilon}\right)$ many iterations of Algorithm~\ref{alg:greedyHazan}, for a suitable choice of $\sigma$, we have obtained $X$ such that $f(X) \le \varepsilon$, which implies that all constraints are violated by at most $\varepsilon$. This means that $A_i \bullet X \le b_i + \varepsilon$, or in other words we say that $X$ is $\varepsilon$-feasible~\cite{Hazan:2008kz,Gartner:2011tl}. It turns out the best choice for the parameter $\sigma$ is $\frac{\log m}{\varepsilon}$, and the curvature constant $C_f(\sigma)$ for this function is bounded by $\sigma \cdot \max_i \lmax(A_i)^2$. The total number of necessary approximate eigenvector computations is therefore in $O\left(\frac{\log m}{\varepsilon^2}\right)$.
In fact, Algorithm~\ref{alg:greedyHazan} when applied to the function~(\ref{eq:softMax}) is very similar to the multiplicative weights method~\cite{Arora:2005uh}. Note that the soft-max function~(\ref{eq:softMax}) is convex in $X$, see also~\cite{Rennie:2005ww}. 
For a slightly more detailed exhibition of this approach of using Algorithm~\ref{alg:greedyHazan} to approximately solving SDPs, we refer the reader to the book of~\cite{Gartner:2011tl}.

Note that this technique of introducing the soft-max function is closely related to smoothing techniques in the optimization literature~\cite{Nemirovski:2004hr,Baes:2009vh}, where the soft-max function is introduced to get a smooth approximation to the largest eigenvalue function. The transformation to a smooth saddle-point problem suggested by~\cite{Baes:2009vh} is more complicated than the simple notion of $\varepsilon$-feasibility suggested here, and will lead to a comparable computational complexity in total.
\subsection{Two Improved Variants of Algorithm~\ref{alg:greedyHazan}}\label{subsec:imprAlgos}

\paragraph{Choosing the Optimal $\alpha$ by Line-Search.}
As we mentioned already for the general algorithm for convex optimization in Section~\ref{sec:optGenConvex}, the optimal $\alpha$ in Algorithm~\ref{alg:greedyHazan}, i.e. the $\alpha \in [0,1]$ of best improvement in the objective function $f$ can be found by line-search.

In particular for matrix completion problems, which we will discuss in more details in Section~\ref{sec:applicationMatCompl}, the widely used squared error is easy to analyze in this respect: If the optimization function is given by $f(X) = \frac{1}{2} \sum_{ij \in P} (X_{ij} - Y_{ij})^{2}$, where $P$ is the set of observed positions of the matrix $Y$, then the optimality condition~(\ref{eq:bestOnLineSegment}) from Section~\ref{subsec:lineSearch} is equivalent to
\begin{equation}
\alpha = \frac{ \sum_{ij \in P} (X_{ij} - y_{ij})(X_{ij} -  v_{i} v_{j}) }{\sum_{ij \in P} (X_{ij} -  v_{i} v_{j})^{2}   } \ .
\end{equation}
Here $X=X^{(k)}$, and $v$ is the approximate eigenvector $v^{(k)}$ used in step $k$ of Algorithm~\ref{alg:greedyHazan}. The above expression is computable very efficiently compared to the eigenvector approximation task.

\paragraph{Immediate Feedback in the Power Method.}
As a second improvement, we propose a heuristic to speed up the eigenvector computation, i.e. the internal procedure $\textsc{ApproxEV}\left(\nabla f(X),\varepsilon'\right)$. 
Instead of multiplying the current candidate vector $v_k$ with the gradient matrix $\nabla f(X)$ in each power iteration, we multiply with $\frac{1}{2}\left( \nabla f(X) + \nabla f(\overline X) \right)$, 
or in other words the average between the current gradient and the gradient at the new candidate location $\overline X = \big(1-\frac1k\big)X^{(k)}+\frac{1}{k} v^{(k)} {v^{(k)}}^T$. Therefore, we immediately take into account the effect of the new feature vector $v^{(k)}$. This heuristic (which unfortunately does not fall into our current theoretical guarantee) is inspired by stochastic gradient descent as in Simon Funk's method, which we will describe in Section~\ref{sec:simonFunk}. In practical experiments, this proposed slight modification will result in a significant speed-up of Algorithm~\ref{alg:greedyHazan}, as we will observe e.g. for matrix completion problems in Section~\ref{sec:applicationMatCompl}.

\subsection{$\Omega(\frac1\varepsilon)$ Lower Bound on the Rank}\label{sec:matLowerBound}

Analogous to the vector case discussed in Section~\ref{subsec:vecLowerBound}, we can also show that the rank of $O\big(\frac1\varepsilon\big)$, as obtained by the greedy Algorithm~\ref{alg:greedyHazan} is indeed optimal, by providing a lower bound of $\Omega\big(\frac1\varepsilon\big)$. In other words we can now exactly characterize the trade-off between rank and approximation quality, for convex optimization over the spectahedron.

For the lower bound construction, we consider the convex function $f(X) := \frobnorm{X}^{2} = X\bullet X$ over the symmetric matrices $\Sym^{n \times n}$. 
This function has gradient $\nabla f(X) = 2X$. We will later see that its curvature constant is $C_f = 2$.

The following lemmata show that the above sparse SDP Algorithm~\ref{alg:greedyHazan} is optimal for the approximation quality (primal as well as dual error respectively), giving lowest possible rank, up to a small multiplicative constant.

\begin{lemma}\label{lem:matPrimalLB}
For $f(X) := \frobnorm{X}^{2}$, and $1 \leq k \leq n$, it holds that
$$\min_{\substack{X \in \Spectahedron \\ \rk(X)\le k}} f(X) = \frac1k.$$
\end{lemma}
We will see that this claim can be reduced to the analogous Lemma~\ref{lem:vecPrimalLB} for the vector case, by the standard technique of diagonalizing a symmetric matrix. (This idea was suggested by Elad Hazan).
Alternatively, an explicit (but slightly longer) proof without requiring the spectral theorem can be obtained by using the Cholesky-decomposition together with induction on $k$.
\begin{proof}
We observe that the objective function $\frobnorm{.}^{2}$, the trace, as well as the property of being positive semidefinite, are all invariant under orthogonal transformations (or in other words under the choice of basis).

By the standard spectral theorem, for any symmetric matrix $X$ of $\rk(X)\le k$, there exists an orthogonal transformation mapping $X$ to a \emph{diagonal} matrix $X'$ with at most $k$ non-zero entries on the diagonal (being eigenvalues of $X$ by the way). For diagonal matrices, the $\frobnorm{.}$ matrix norm coincides with the $\norm{.}_2$ vector norm of the diagonal of the matrix.
Finally by applying the vector case Lemma~\ref{lem:vecPrimalLB} for the diagonal of $X'$, we obtain that $f(X)=f(X') \ge \frac1k$.

To see that the minimum can indeed be attained, one again chooses the ``uniform'' example $X:=\frac1k \id_k \in \Spectahedron$, being the matrix consisting of $k$ non-zero entries (of $\frac1k$ each) on the diagonal. This gives $f(X)=\frac1k$.
\end{proof}

Recall from Section~\ref{subsec:algHazan} that for convex problems of the form~(\ref{eq:optTrEq1}) over the Spectahedron, the \emph{duality gap} is the non-negative value $g(X) := f(X) - \omega(X) =  X\bullet \nabla f(X)  - \lambda_{\min}(\nabla f(X))$. Also, by weak duality as given in Lemma~\ref{lem:weakDual}, this value is always an upper bound for the primal error, that is $f(X)-f(X^*) \le g(X) ~\forall X$.

\begin{lemma}
For $f(X) := \frobnorm{X}^{2}$, and any $k \in \N$, $k < n$, it holds that
$$g(X) \ge \frac1k \ \ \ \ \ \ \ \ \  \forall X \in \Spectahedron \text{ s.t. } \rk(X)\le k.$$
\end{lemma}
\begin{proof}
$g(X) = \lmax(-\nabla f(X)) + X\bullet \nabla f(X)  = - \lambda_{\min}(X) +X\bullet 2X$. We now use that $\lambda_{\min}(X) = 0$ for all symmetric PSD matrices $X$ that are not of full rank $n$, and that by Lemma~\ref{lem:matPrimalLB}, we have $ X\bullet X  = \tr(X^TX) = f(X) \ge \frac1k$.
\end{proof}

\paragraph{The Curvature.}
We will compute the curvature $C_f$ of our function $f(X) := X \bullet X$, showing that $C_f=2$ in this case. Using the definition~(\ref{eq:Cf}), and the fact that here 
\[
\begin{array}{rl}
  & f(Y)-f(X)-(Y-X)\bullet \nabla f(X) \\
=& Y\bullet Y-X\bullet X-(Y-X)\bullet 2X \\
=& \frobnorm{X-Y}^2 \ ,
\end{array}
\]
one obtains that $C_f = \sup_{X,Y\in \Spectahedron}  \frobnorm{X-Y}^2 = \diam_{Fro}(\Spectahedron)^2 = 2$. Finally the following Lemma~\ref{lem:frobDiamSpect} shows that the diameter is indeed $2$. 

\begin{lemma}[Diameter of the Spectahedron]\label{lem:frobDiamSpect}
\[
\diam_{Fro}(\Spectahedron)^2 = 2 \ .
\]
\end{lemma}
\begin{proof}
Using the fact that the spectahedron $\Spectahedron$ is the convex hull of the rank-1 matrices of unit trace, see Lemma~\ref{lem:minAtVtxOfSpect}, we know that the diameter must be attained at two vertices of $\Spectahedron$, i.e. $u,v\in\R^n$ with $\norm{u}_2=\norm{v}_2=1$, and\vspace{-5pt}
\[
\begin{array}{rl}
 &\frobnorm{vv^T-uu^T}^{2} \\[3pt]
=& vv^T \bullet vv^T + uu^T \bullet uu^T -2 vv^T \bullet uu^T \\
=& v^Tvv^Tv + u^Tuu^Tu - 2 u^Tvv^Tu\\
=& \norm{v}^{4} + \norm{u}^{4} - 2(u^Tv)^2 \ .
\end{array}
\]
Clearly, this quantity is maximized if $u$ and $v$ are orthogonal.
\end{proof}

\paragraph{Note:} We could also study $f(X) := \gamma \frobnorm{X}^{2}$ instead, for some $\gamma > 0$. This function has \emph{curvature} constant $C_f = 2\gamma$, and for this scaling our above lower bounds will also just scale linearly, giving $\frac{C_f}{k}$ instead of $\frac1k$.

\section{Semidefinite Optimization with $\ell_\infty$-Bounded Diagonal}\label{sec:algMaxDiag}

Here we specialize our general Algorithm~\ref{alg:greedyConvex} to  semidefinite optimization problems where all diagonal entries are individually constrained. This will result in a new optimization method that can also be applied to max-norm optimization problems, which we will discuss in more detail in Section~\ref{chap:optNucMax}. As in the previous Section~\ref{sec:algTrace}, here we also consider matrix optimization problems over the space $\X = \Sym^{n\times n}$ of symmetric matrices, equipped with the standard Frobenius inner product $\langle X,Y \rangle = X\bullet Y$.

Formally, we consider the following special case of the general optimization problems~(\ref{eq:optGenConvex}), i.e. 
\begin{equation}\label{eq:optMatInfLe1}
\begin{array}{rl}
   \displaystyle\mini_{X \in\Sym^{n\times n}} & f(X)  \\
    s.t. &  X_{ii} \le 1 ~~~\forall i ,\\
         &  X \succeq 0 \ .
\end{array}
\end{equation}

We will write $\MaxCutPolytope := \SetOf{X \in \Sym^{n \times n}}{X \succeq 0, ~ X_{ii} \le 1 ~\forall i}$ for the feasible set in this case, that is the PSD matrices whose diagonal entries are all upper bounded by one. This class of optimization problems has become widely known for the linear objective case when $f(X) = A \bullet X$, if $A$ being the Laplacian matrix of a graph. In this case, one obtains the standard SDP relaxation of the Max-Cut problem~\cite{Goemans:1995ug}, which we will briefly discuss below.
Also, this optimization domain is strongly related to the matrix \emph{max-norm}, which we study in more detail in Section~\ref{sec:maxNorm}. %

Our general optimization Algorithm~\ref{alg:greedyConvex} directly applies to this specialized class of optimization problems as well, in which case it becomes the method depicted in the following Algorithm~\ref{alg:maxDiag}.

\begin{algorithm}[h!]
  \caption{Sparse Greedy for Max-Norm Bounded Semidefinite Optimization}
  \label{alg:maxDiag}
\begin{algorithmic}
  \STATE {\bfseries Input:} Convex function $f$ with curvature $C_{f}$, target accuracy $\varepsilon$
  \STATE {\bfseries Output:} $\varepsilon$-approximate solution for problem~(\ref{eq:optMatInfLe1})
  \STATE Set $X^{(0)}:= vv^T$ for an arbitrary unit length vector $v \in \R^n$.
  \FOR{$k=0\dots\infty$}
  \STATE Let $\alpha := \frac2{k+2}$
  \STATE Compute $S := \approxLin\left(\nabla f(X^{(k)}),\MaxCutPolytope,\alpha C_f\right)$
  \STATE Update $X^{(k+1)}:=X^{(k)}+\alpha(S - X^{(k)})$
  \ENDFOR
\end{algorithmic}
\end{algorithm}

\paragraph{The Linearized Problem.}
Here, the internal subproblem $\approxLin()$ of approximately minimizing a linear function over the domain $\MaxCutPolytope$ of PSD matrices is a non-trivial task. Every call of $\approxLin(A,\MaxCutPolytope,\varepsilon')$ in fact means that we have to solve a semidefinite program $\min_{Y\in\MaxCutPolytope} Y \bullet A$ for a given matrix $A$, or in other words
\begin{equation}\label{eq:ADualMaxGapSDP}
\begin{array}{rl}
  \displaystyle\mini_Y & Y \bullet A \\ %
    s.t. &Y_{ii} \le 1 ~~~\forall i ,\\
        & Y \succeq 0 
\end{array}
\end{equation}
up to an additive approximation error of $\varepsilon'=\alpha C_f$. 

\paragraph{Relation to Max-Cut.}
In~\cite{Arora:2005uh,Kale:2007wk}, the same linear problem is denoted by $\textsc{(MaxQP)}$.
In the special case that $A$ is chosen as the Laplacian matrix of a graph, then the above SDP is widely known as the standard SDP relaxation of the Max-Cut problem~\cite{Goemans:1995ug} (not to be confused with the combinatorial Max-Cut problem itself, which is known to be NP-hard). In fact the original relaxation uses equality constraints $Y_{ii} = 1$ on the diagonal instead, but for any matrix $A$ of positive diagonal entries (such as e.g. a graph Laplacian), this condition follows automatically in the maximization variant of $(\ref{eq:ADualMaxGapSDP})$, see~\cite{Klein:1996hm}, or also~\cite{Gartner:2011tl,Kale:2007wk} for more background.

\paragraph{Duality and Duality of Norms.}
In Section~\ref{sec:maxNorm} we will see that the above quantity~(\ref{eq:ADualMaxGapSDP}) that determines both the step in our greedy Algorithm~\ref{alg:maxDiag}, but also the duality gap, is in fact the norm of $A$ that is dual to the matrix max-norm.

For optimization problems of the form~(\ref{eq:optMatInfLe1}), it can again be shown that the poor-man's duality given by the linearization (see also Section~\ref{sec:poorDual}) indeed coincides with classical \emph{Wolfe-duality} from the optimization literature.

Fortunately, it was shown by~\cite{Arora:2005uh} that also this linearized convex optimization problem~(\ref{eq:ADualMaxGapSDP}) --- and therefore also our internal procedure $\approxLin(.)$ --- can be solved relatively efficiently, 
if the matrix~$A$ (i.e. $\nabla f(X)$ in our case) is sparse.%
\footnote{%
Also, Kale in~\cite[Theorem 14]{Kale:2007wk} has shown that this problem can be solved very efficiently if the matrix $A=-\nabla f(X)$ is sparse. Namely if $A$ is the Laplacian matrix of a weighted graph, then a multiplicative $\varepsilon$-approximation to $(\ref{eq:ADualMaxGapSDP})$ can be computed in time $\tilde O(\frac{\Delta^2}{d^2}N_{\!A})$ time, where $N_{\!A}$ is the number of non-zero entries of the matrix $A$. Here $\Delta$ is the maximum entry on the diagonal of $A$, and $d$ is the average value on the diagonal.
}%

\begin{theorem}
The algorithm of~\cite{Arora:2005uh} delivers an additive $\varepsilon'$\hyph approximation
to the linearized problem~(\ref{eq:ADualMaxGapSDP}) in time
\[
\tilde O\left(\frac{n^{1.5}L^{2.5}}{\varepsilon'^{2.5}} N_{\!A} \right)
\]
where the constant $L>0$ is an upper bound on the maximum value of $Y\bullet A$ over $Y\in\MaxCutPolytope$, and $N_{\!A}$ is the number of non-zeros in $A$.
\end{theorem}
\begin{proof}
The results of \cite[Theorem 3]{Arora:2005uh} and~\cite[Theorem 33]{Kale:2007wk}
give a running time of order $\tilde O\left(\frac{n^{1.5}}{\varepsilon^{2.5}}\cdot\min\left\{N,\frac{n^{1.5}}{\varepsilon \alpha^*}\right\}\right)$ to obtain a multiplicative $(1-\varepsilon)$-approximation, where $\alpha^*$ is the value of an optimal solution. Formally we obtain $S\in\MaxCutPolytope$ with $S\bullet A \ge (1-\varepsilon)\alpha^*$. In other words by using an accuracy of $\varepsilon := \frac{\varepsilon'}{\alpha^*}$, we obtain an additive $\varepsilon'$-approximation to~(\ref{eq:ADualMaxGapSDP}).
\end{proof}
Here the notation $\tilde O(.)$ again suppresses poly-logarithmic factors in $n$, and $N$ is the number of non-zero entries of the matrix $A$. Note that analogous to the approximate eigenvector computation for Hazan's Algorithm~\ref{alg:greedyHazan}, we need the assumption that the linear function given by $Y\bullet \nabla f(X)$ is bounded over the domain $Y\in\MaxCutPolytope$. However this is a reasonable assumption, as our function has bounded curvature $C_f$ (corresponding to $\nabla f(X)$ being Lipschitz-continuous over the domain $\MaxCutPolytope$), and the diameter of $\MaxCutPolytope$ is bounded. %

The reason we need an absolute approximation quality lies in the analysis of Algorithm~\ref{alg:greedyConvex}, even if it would feel much more natural to work with relative approximation quality in many cases.

\paragraph{Convergence.}
The convergence result for the general Algorithm~\ref{alg:greedyConvex} directly gives us the analysis for the specialized algorithm here. Note that the curvature over the domain $\MaxCutPolytope$ here is given by
\begin{equation}\label{eq:CfMaxDiag}
C_f := \sup_{\substack{X,V \in \MaxCutPolytope, \, \alpha\in[0,1], \\
                    Y=X+\alpha(V-X)}}
            \textstyle\frac{1}{\alpha^{2}} \big(f(Y)-f(X)+(Y-X) \bullet \nabla f(X) \big) \ .
\end{equation}

\begin{theorem}\label{thm:algMaxDiag}
For each $k\ge1$, the iterate $X^{(k)}$ of Algorithm~\ref{alg:maxDiag} satisfies
\[
f(X^{(k)}) - f(X^*) \le \frac{8C_f}{k+2} \ .
\]
where $X^*\in \Spectahedron$ is an optimal solution to problem~(\ref{eq:optMatInfLe1}).

Furthermore, after at most $2\left\lceil\frac{8C_f}{\varepsilon}\right\rceil +1 = O(\frac1\varepsilon)$ many steps, it has an iterate $X^{(k)}$ with $g(X^{(k)}) \le\varepsilon$.
\end{theorem}
\begin{proof}
This is a corollary of Theorem~\ref{thm:primalGreedy} and Theorem~\ref{thm:primalDualGreedy}.
\end{proof}

\paragraph{Applications.}
The new algorithm can be used to solve arbitrary max-norm constrained convex optimization problems, such as max-norm regularized matrix completion problems, which we will study in Section~\ref{chap:optNucMax}.

\section{Sparse Semidefinite Optimization}\label{sec:algMatSparse}

Another interesting optimization domain among the semidefinite matrices is given by the matrices with only one non-zero off-diagonal entry. Here we specialize our general Algorithm~\ref{alg:greedyConvex} to convex optimization over the convex hull given by such matrices. Our algorithm will therefore obtain $\varepsilon$-approximate solutions given by only $O\left(\frac1\varepsilon\right)$ such sparse matrices, or in other words solutions of sparsity $O\left(\frac1\varepsilon\right)$.

\paragraph{Why bother?}
The same sparse matrices are also used in the graph sparsification approach by~\cite{Batson:2009hg}\footnote{%
The theoretical result of~\cite{Batson:2009hg} guarantees that all eigenvalues of the resulting sparse matrix (corresponding to the Laplacian of a sparse graph) do not differ too much from their counterparts in the original graph.
}%
.
Furthermore, sparse solutions to convex matrix optimization problems have gained interest in dimensionality reduction, as in sparse PCA, see~\cite{Zhang:2010tb} for an overview.

\paragraph{Setup.}
Formally, here we again use the standard Frobenius inner product $\langle X,Y \rangle = X\bullet Y$ over the symmetric matrices $\Sym^{n\times n}$, and consider the sparse PSD matrices given by $P^{(ij)} := (\unit_i + \unit_j) (\unit_i + \unit_j)^T = \left(\!\begin{smallmatrix}
\cdot&\cdot&\cdot&\cdot&\cdot\\
\cdot&1&\cdot&1&\cdot\\
\cdot&\cdot&\cdot&\cdot&\cdot\\
\cdot&1&\cdot&1&\cdot\\
\cdot&\cdot&\cdot&\cdot&\cdot\\
\end{smallmatrix}\!\right)$, for any fixed pair of indices $i,j\in[n]$, $i\ne j$. In other words $P^{(ij)}_{uv} = 1$ for $u\in\{i,j\}, v\in\{i,j\}$, and zero everywhere else.
We also consider the analogous ``negative'' counterparts of such matrices, namely $N^{(kl)} := (\unit_i - \unit_j) (\unit_i - \unit_j)^T = \left(\!\begin{smallmatrix}
\cdot&\cdot&\cdot&\cdot&\cdot\\
\cdot&1&\cdot&-1&\cdot\\
\cdot&\cdot&\cdot&\cdot&\cdot\\
\cdot&-1&\cdot&1&\cdot\\
\cdot&\cdot&\cdot&\cdot&\cdot\\
\end{smallmatrix}\!\right)$, i.e. $N^{(ij)}_{uv} = -1$ for the two off-diagonal entries $(u,v)\in\{(i,j),(j,i)\}$, and $N^{(ij)}_{uv} = 1$ for the two diagonal entries $(u,v) \in\{(i,i),(j,j)\}$, and zero everywhere else.\\

Analogously to the two previous applications of our method to semidefinite optimization, we now optimize a convex function, i.e. 
\begin{equation}\label{eq:optFourPointPSD}
\mini_{X\in\FourPointPSD} f(X) 
\end{equation}
over the domain\vspace{-5pt}
\[
D = \FourPointPSD := \conv\left(\bigcup_{ij} P^{(ij)} \cup \bigcup_{ij} N^{(ij)} \right) \ .
\]

\paragraph{Optimizing over Sparse Matrices, and Solving the Linearization.}
Applying our general Algorithm~\ref{alg:greedyConvex} to this class of problems~(\ref{eq:optFourPointPSD}) becomes very simple, as the linear primitive problem $\exactLin\left(D_{\!X},\FourPointPSD\right)$ for any fixed matrix $D_X \in \Sym^{n\times n}$ is easily solved over $\FourPointPSD$. From our simple Lemma~\ref{lem:convHull} on linear functions over convex hulls, we know that this linear minimum is attained by the single sparse matrix $P^{(ij)}$ or $N^{(ij)}$ that maximizes the inner product with $-D_{\!X}$. The optimal pair of indices $(k,l)$ can be found by a linear pass through the gradient $D_{\!X}=\nabla f(X)$. This means that the linearized problem is much easier to solve than in the above two Sections~\ref{sec:algTrace} and~\ref{sec:algMaxDiag}. Altogether, Algorithm~\ref{alg:greedyConvex} will build approximate solutions $X^{(k)}$, each of which is a convex combination of $k$ of the atomic matrices $P^{(ij)}$ or $N^{(ij)}$, as formalized in the following theorem:

\begin{theorem}\label{thm:algFourPointPSD}
Let $n\ge 2$ and let $X^{(0)}:=P^{(12)}$ be the starting point. Then for each $k\ge1$, the iterate $X^{(k)}$ of Algorithm~\ref{alg:greedyConvex} has at most $4(k+1)$ non-zero entries, and satisfies
\[
f(X^{(k)}) - f(X^*) \le \frac{8C_f}{k+2} \ .
\]
where $X^*\in \FourPointPSD$ is an optimal solution to problem~(\ref{eq:optFourPointPSD}).

Furthermore, for any $\varepsilon>0$, after at most $2\left\lceil\frac{8C_f}{\varepsilon}\right\rceil +1 = O\left(\frac1\varepsilon\right)$ many steps, it has an iterate $X^{(k)}$ of only $O\left(\frac1\varepsilon\right)$ many non-zero entries, satisfying $g(X^{(k)}) \le\varepsilon$.
\end{theorem}
\begin{proof}
This is a corollary of Theorem~\ref{thm:primalGreedy} and Theorem~\ref{thm:primalDualGreedy}. The sparsity claim follows from our observation that the step directions given by $\exactLin\left(\nabla f(X),\FourPointPSD\right)$ are always given by one of the sparse matrices $P^{(ij)}$ or $N^{(ij)}$.
\end{proof}

\paragraph{Optimizing over Non-Negative Matrix Factorizations.}
We also consider the slight variant of~(\ref{eq:optFourPointPSD}), namely optimizing only over one of the two types of matrices as the domain $D$, i.e. only combinations of positive $P^{(ij)}$ or only of negative $N^{(ij)}$. This means that the domain is given by $D = \FourPointPSDplus := \conv\left(\bigcup_{ij} P^{(ij)} \right)$ or $D = \FourPointPSDminus := \conv\left(\bigcup_{ij} N^{(ij)} \right)$.
The above analysis for Algorithm~\ref{alg:greedyConvex} holds in exactly the same way. Now for $\FourPointPSDplus$, each step direction $s = s^{(k)}$ used by the algorithm is given by $s = P^{(ij)} = (\unit_i + \unit_j) (\unit_i + \unit_j)^T$ for some $i,j$, and so we have that each of the approximations $X^{(k)}$ is a sum of $k$ many positive rank-1 factors of this form. In other words in each step $k$, $X^{(k)} = LR^T$ is a product of two (entry-wise) non-negative matrices of at most $k$ columns each, i.e. $L \in \R^{n \times k}$ and $R^{n \times k}$. Consequently, our algorithm provides solutions that are \emph{non-negative matrix factorizations}, which is a successful technique in matrix completion problems from recommender systems, see e.g.~\cite{Wu:2007tm}.

\paragraph{Relation to Bounded Trace and Diagonally Dominant Matrices.}
Observe the matrices in $\FourPointPSD$ form a subset of the bounded trace PSD matrices $\Spectahedron$ that we studied in the previous Section~\ref{sec:algTrace}, since every matrix $P^{(ij)}$ or $N^{(ij)}$ is PSD and has trace equal two. Furthermore, we observe that all matrices $X \in \FourPointPSD$ are \emph{diagonally dominant}, meaning that 
\[
\abs{X_{ii}} \ge \sum_{j\ne i} \abs{X_{ij}} ~~~~~\forall i\in[n]
\]
In the case that we restrict to using only one of the two types of matrices $\FourPointPSDplus$ or $\FourPointPSDminus$ as the domain, then we have that \emph{equality} $\abs{X_{ii}} = \sum_{j\ne i} \abs{X_{ij}}$ always holds, since this equality is preserved under taking convex combinations, and holds for the atomic matrices $P^{(ij)}$ and $N^{(ij)}$.

\paragraph{The Curvature.}
The above reasoning also implies that the curvature $C_f$ for problems of the form~(\ref{eq:optFourPointPSD}) is upper bounded by the curvature in the spectahedron-case as given in~(\ref{eq:CfTrace}), since $\FourPointPSDplus \subseteq \FourPointPSD \subseteq 2\cdot \Spectahedron$.

\paragraph{Applications and Future Research.}
Computationally, the approach here looks very attractive, as the cost of a ``sparse'' step here is much cheaper than an approximate eigenvector computation which is needed in the bounded trace case as explained in Section~\ref{sec:algTrace}.

Also, it will be interesting to see how a regularization by constraining to a scaled domain $\FourPointPSD$ or $\FourPointPSDplus$ will perform in practical machine learning applications as for example dimensionality reduction, compared to nuclear norm regularization that we will discuss in the following Chapter~\ref{chap:optNucMax}.

It also remains to investigate further on whether we can approximate general bounded trace semidefinite problems of the form~(\ref{eq:optTrEq1}) by using only sparse matrices.

\section{Submodular Optimization}\label{sec:algSubmodular}

For a finite ground set $S$, a real valued function defined on all subsets of $S$, %
is called \emph{submodular}, if
\[
g(X\cap Y) + g(X\cup Y) \le g(X) + g(Y) ~~~~~\forall X,Y\subseteq S
\]

For any given submodular function $g$ with $g(\emptyset)=0$, the paper~\cite[Section 3]{Lovasz:1983wv} introduces a corresponding convex set in $\R^{\abs{S}}$, called the \emph{submodular polyhedron} (or also Lovasz polyhedron),
\[
\mathcal{P}_g := \SetOf{x \in \R^{\abs{S}}}{\sum_{i\in T}x_i \le g(T) ~~\forall T\subseteq S} \ .
\]

We would now like to study convex optimization problems over such domains, which become compact convex sets if we intersect with the positive orthant, i.e. $D:=\mathcal{P}_g \cap \R^{\abs{S}}_{\ge0}$.

Nicely for our optimization framework,~\cite[Section 3]{Lovasz:1983wv}  already showed that there is a simple greedy algorithm which optimizes any linear function over the domain $\mathcal{P}_g$, i.e. it solves $\displaystyle\max_{x\in\mathcal{P}_g} c^Tx$, or in other words it exactly solves our internal problem $\exactLin\left(c,\mathcal{P}_g\right)$.

Lovasz~\cite{Lovasz:1983wv} already demonstrated how to use this kind of linear optimization over $\mathcal{P}_g$ to solve submodular minimization problems. It remains to investigate if there are interesting applications for the wider class of more general convex (non-linear) functions $f$ over such domains, as addressed by our Algorithm~\ref{alg:greedyConvex}.

\section{Optimization with the Nuclear and Max-Norm}\label{chap:optNucMax}

Matrix optimization problems with a nuclear norm or max-norm regularization, such as e.g. low norm matrix factorizations, have seen many applications recently, ranging from low-rank recovery, dimensionality reduction, to recommender systems.
We propose two new first-order approximation methods building upon two of the simple semidefinite optimizers we have studied above, that is the approximate SDP solver of~\cite{Hazan:2008kz} from Section~\ref{sec:algTrace} on one hand, and our bounded diagonal optimizer from Section~\ref{sec:algMaxDiag} on the other hand.
The algorithms come with strong convergence guarantees.

In contrast to existing methods, our nuclear norm optimizer does not need any Cholesky or singular value decompositions internally, and provides guaranteed approximations that are simultaneously of low rank. The method is free of tuning parameters, and easy to parallelize.

\subsection{Introduction}

Here we consider convex optimization problems over matrices, which come with a regularization on either the nuclear norm or the max-norm of the optimization variable.

Convex optimization with the nuclear norm has become a very successful technique in various machine learning, computer vision and compressed sensing areas such as low-rank recovery \cite{Fazel:2001vw,Candes:2009kj,Candes:2010jb}, dimensionality reduction (such as robust principal component analysis~\cite{Candes:2011bd}), and also recommender systems and matrix completion. Here matrix factorizations~\cite{Srebro:2004tj,Koren:2009jg} --- regularized by the nuclear or max-norm --- have gained a lot of attention with the recently ended \emph{Netflix Prize} competition. Many more applications of similar optimization problems can be found among dimensionality reduction, matrix classification, multi-task learning, spectral clustering and others.
The success of these methods is fueled by the property of the nuclear norm being a natural convex relaxation of the rank, allowing the use of scalable convex optimization techniques.\\

Based on the semidefinite optimization methods that we have presented in the above Sections~\ref{sec:algTrace} and~\ref{sec:algMaxDiag}, we propose two new, yet simple, first-order algorithms for nuclear norm as well as max-norm regularized convex optimization.

For the nuclear norm case, our proposed method builds upon the first-order scheme for semidefinite optimization by~\cite{Hazan:2008kz}, which we have investigated in Section~\ref{subsec:algHazan}. This approach allows us to significantly reduce the computational complexity per iteration, and therefore scale to much larger datasets: 
While existing methods need an entire and exact singular value decomposition (SVD) in each step, our method only uses a single approximate eigenvector computation per iteration, which can be done by e.g. the power method.
A conference version of our work for nuclear norm regularized problems has appeared in~\cite{Jaggi:2010tz}.

In the same spirit, we will also give a new algorithm with a convergence guarantee for optimizing with a max-norm regularization. %
For matrix completion problems,  experiments show that the max-norm can result in an improved generalization performance compared to the nuclear norm in some cases~\cite{Srebro:2004tj,Lee:2010wd}.

\paragraph{Nuclear Norm Regularized Convex Optimization.}
We consider the following convex optimization problems over matrices:
\begin{equation}\label{eq:regNucLagrange}
\min_{Z \in \R^{m\times n}} f(Z) + \mu \nucnorm{Z}
\end{equation}
and the corresponding constrained variant
\begin{equation}\label{eq:constrainedNuc}
\min_{Z \in \R^{m\times n},\, \nucnorm{Z} \le \frac t2} f(Z)
\end{equation}
where $f(Z)$ is any differentiable convex function (usually called the loss function), $\nucnorm{.}$ is the \emph{nuclear norm} of a matrix, also known as the \emph{trace norm} (sum of the singular values, or $\ell_{1}$-norm of the spectrum). %
Here $\mu > 0$ and $t>0$ respectively are given parameters, usually called the \emph{regularization parameter}.

The nuclear norm is know as the natural generalization of the (sparsity inducing) $\ell_1$-norm for vectors, to the case of semidefinite matrices.
When choosing $f(X) := \norm{\mathcal{A}(X)-b}_{2}^{2}$ for some linear map $\mathcal{A} : \R^{n \times m} \rightarrow \R^{p}$, the above formulation~(\ref{eq:regNucLagrange}) is the matrix generalization of the problem
$
\min_{x \in \R^{n}} \norm{Ax-b}_{2}^{2} + \mu \norm{x}_{1}
$, 
for a fixed matrix $A$, which is the important $\ell_{1}$-\emph{regularized least squares problem},  also known as the basis pursuit de-noising problem in the compressed sensing literature, see also Section~\ref{sec:relComprSensing}. The analoguous vector variant of~(\ref{eq:constrainedNuc}) is the \emph{Lasso} problem~\cite{Tibshirani:1996wb} which is $\min_{x \in \R^{n}} \SetOf{ \norm{Ax-b}_{2}^{2} }{ \norm{x}_{1} \le t }$.

\paragraph{Max-Norm Regularized Convex Optimization.}
Intuitively, one can think of the matrix max-norm as the generalization of the vector $\ell_\infty$-norm to PSD matrices. Here we consider optimization problems with a max-norm regularization, which are given by
\begin{equation}\label{eq:regMaxLagrange}
\min_{Z \in \R^{m\times n}} f(Z) + \mu \maxnorm{Z}
\end{equation}
and the corresponding constrained variant being 
\begin{equation}\label{eq:constrainedMax}
\min_{Z \in \R^{m\times n},\, \maxnorm{Z} \le t} f(Z) \ .
\end{equation}

\paragraph{Our Contribution.}
Using the optimization methods from the previous Sections~\ref{sec:algTrace} and~\ref{eq:optMatInfLe1}, we present a much simpler algorithm to solve problems of the form~(\ref{eq:constrainedNuc}), which does not need any internal SVD computations. The same approach will also solve the max-norm regularized problems~(\ref{eq:constrainedMax}). We achieve this by transforming the problems to the convex optimization setting over positive semidefinite matrices which we have studied in the above Sections~\ref{subsec:algHazan} and~\ref{sec:algMaxDiag}.

Our new approach has several advantages for nuclear norm optimization when compared to the existing algorithms such as ``proximal gradient'' methods (APG) and ``singular value thresholding'' (SVT), see e.g.~\cite{Ganesh:2009dn,Cai:2010vw,Toh:2010vx,Ji:2009vd}, and also in comparison to the alternating-gradient-descent-type methods (as e.g.~\cite{Rennie:2005vn,Lin:2007ea}).

\begin{enumerate}

\item[i)] By employing the approximate SDP solver by~\cite{Hazan:2008kz}, see Algorithm~\ref{alg:greedyHazan}, we obtain a guaranteed $\varepsilon$-approximate solution $Z$ after $O\big(\frac1\varepsilon\big)$ iterations. Crucially, the resulting solution $Z$ is simultaneously of low rank, namely rank $O\big(\frac1\varepsilon\big)$.  Also the algorithm maintains a compact representation of $Z$ in terms of a low-rank matrix factorization $Z=LR^T$ (with the desired bounded nuclear norm), and can therefore even be applied if the full matrix $Z$ would be far too large to even be stored.

\item[ii)] Compared to the alternating-gradient-descent-type methods from machine learning, we overcome the problem of working with non-convex formulations of the form $f(LR^T)$, which is NP-hard, and instead solve the original convex problem in $f(Z)$.

\item[iii)] The total running time of our algorithm for nuclear norm problems grows linear in the problem size, allows to take full advantage of sparse problems such as e.g. for matrix completion. More precisely, the algorithm runs in time $O\left(\frac{N_f}{\varepsilon^{1.5}}\right)$, where $N_f$ is the number of matrix entries on which the objective function $f$ depends.
Per iteration, our method consists of only a single approximate (largest) eigenvector computation, allowing it to scale to any problem size where the power method (or Lanczos' algorithm) can still be applied. This also makes the method easy to implement and to parallelize.
Existing APG/SVT methods by contrast need an entire SVD in each step, which is significantly more expensive.

\item[iv)] On the theory side, our simple convergence guarantee of $O\big(\frac1\varepsilon\big)$ steps holds even if the used eigenvectors are only approximate. In comparison, those existing methods that come with a convergence guarantee do require an exact SVD in each iteration, which might not always be a realistic assumption in practice.
\end{enumerate}

We demonstrate that our new algorithm on standard datasets improves over the state of the art methods, and scales to large problems such as matrix factorizations on the Netflix dataset. 

Hazan's Algorithm~\ref{alg:greedyHazan} can be interpreted as the generalization of the \emph{coreset} approach to problems on symmetric matrices, which we have explained in the previous Section~\ref{subsec:algHazan}.  Compared to the $O(1/\sqrt{\varepsilon})$ convergence methods in the spirit of~\cite{Nesterov:1983wy,Nesterov:2007wm}, our number of steps is larger, which is however more than compensated by the improved step complexity, being lower by a factor of roughly $(n+m)$.

Our new method for the nuclear norm case can also be interpreted as a modified, theoretically justified variant of Simon Funk's popular SVD heuristic~\cite{Webb:2006vm} for regularized matrix factorization. To our knowledge this is the first guaranteed convergence result for this class of alternating-gradient-descent-type algorithms.

\paragraph{Related Work.}

For nuclear norm optimization, there are two lines of existing methods. On the one hand, in the optimization community,~\cite{Toh:2010vx,Liu:2009wj},~\cite{Ganesh:2009dn} and~\cite{Ji:2009vd} independently proposed algorithms that obtain an $\varepsilon$-accurate solution to~(\ref{eq:regNucLagrange}) in $O(1/\sqrt{\varepsilon})$ steps, by improving the algorithm of~\cite{Cai:2010vw}.  These methods are known under the names ``accelerated proximal gradient'' (APG) and ``singular value thresholding'' (SVT).
More recently also~\cite{Mazumder:2010va} and~\cite{Ma:2009kv} proposed algorithms along the same idea. Each step of all those algorithms requires the computation of the singular value decomposition (SVD) of a matrix of the same size as the solution matrix, which is expensive even with the currently available fast methods such as \textsf{PROPACK}.~\cite{Toh:2010vx} and~\cite{Ji:2009vd} and also~\cite{Ganesh:2009dn} show that the primal error of their algorithm is smaller than $\varepsilon$ after $O(1/\sqrt{\varepsilon})$ steps, using an analysis inspired by~\cite{Nesterov:1983wy} and~\cite{Beck:2009gh}. For an overview of related algorithms, we also refer the reader to~\cite{Candes:2011bd}.
As mentioned above, the method presented here has a significantly lower computational cost per iteration (one approximate eigenvector compared to a full exact SVD), and is also faster in practice on large matrix completion problems.

On the other hand, in the machine learning community, research originated from matrix completion and factorization~\cite{Srebro:2004tj}, later motivated by the Netflix prize challenge, getting significant momentum from the famous blog post by~\cite{Webb:2006vm}.  Only very recently an understanding has formed that many of these methods can indeed by seen as optimizing with regularization term closely related to the nuclear norm, see Section~\ref{sec:simonFunk} and~\cite{Jaggi:2010tz,Salakhutdinov:2010tpa}.
The majority of the currently existing machine learning methods such as for example~\cite{Rennie:2005vn,Lin:2007ea} and later also 
\cite{Paterek:2007va,Zhou:2008hh,Koren:2009jg,Takacs:2009ug,Ilin:2010tl,Gemulla:2011wk}
are of the type of ``alternating'' gradient descent applied to $f(LR^T)$, where at each step one of the factors $L$ and $R$ is kept fixed, and the other factor is updated by a gradient or stochastic gradient step. Therefore, despite working well in many practical applications, all these mentioned methods can get stuck in local minima --- and so are theoretically not well justified, see also the discussion in~\cite{DeCoste:2006wt} and our Section~\ref{sec:transform}.

The same issue also comes up for max-norm optimization, where for example~\cite{Lee:2010wd} optimize over the non-convex factorization~(\ref{eq:maxNormMF}) for bounded max-norm. To our knowledge, no algorithm with a convergence guarantee was known so far.

Furthermore, optimizing with a rank constraint was recently shown to be NP-hard~\cite{Gillis:2010wua}.
In practical applications, nearly all approaches for large scale problems are working over a factorization $Z=LR^T$ of bounded rank, therefore ruling out their ability to obtain a solution in polynomial time in the worst-case, unless $\text{P}=\text{NP}$.

Our new method for both nuclear and max-norm avoids all the above described problems by solving an equivalent convex optimization problem, and provably runs in near linear time in the nuclear norm case.

\subsection{The Nuclear Norm for Matrices}\label{sec:nucNorm}

The \emph{nuclear norm} $\nucnorm{Z}$ of a rectangular matrix $Z \in \R^{m \times n}$, also known as the \emph{trace norm} or \emph{Ky Fan norm}, is given by the sum of the singular values of $Z$, which is equal to the $\ell_1$-norm of the singular values of $Z$ (because singular values are always non-negative).
Therefore, the nuclear norm is often called the Schatten $\ell_1$-norm. In this sense, it is a natural generalization of the $\ell_{1}$-norm for vectors which we have studied earlier. 

The nuclear norm has a nice equivalent characterization in terms of matrix factorizations of~$Z$, i.e.
\begin{equation}\label{eq:nucNormMF}
\nucnorm{Z} :=
\min_{LR^T = Z} \frac{1}{2}\big( \frobnorm{L}^{2} +
                                           \frobnorm{R}^{2} \big) \ ,
\end{equation}
where the number of columns of the factors $L \in \R^{m \times k}$ and $R^{n \times k}$ is not constrained~\cite{Fazel:2001vw,Srebro:2004tj}. In other words, the nuclear norm constrains the average Euclidean row or column norms of any factorization of the original matrix $Z$.

Furthermore, the nuclear norm is \emph{dual} to the standard spectral matrix norm (i.e. the matrix operator norm), meaning that
\[
\nucnorm{Z} = \max_{B, \spectnorm{B} \le 1} B \bullet Z \ ,
\]
see also~\cite{Recht:2010tf}. Recall that $\spectnorm{B}$ is defined as the first singular value $\sigma_1(B)$ of the matrix~$B$.

Similarly to the property of the vector $\norm{.}_1$-norm being the best convex approximation to the sparsity of a vector, as we discussed in Section~\ref{sec:vecL1} the nuclear norm is the best convex approximation of the matrix rank. More precisely, $\nucnorm{.}$ is the convex envelope of the rank~\cite{Fazel:2001vw}, meaning that it is the largest convex function that is upper bounded by the rank on the convex domain of matrices $\SetOf{Z}{\spectnorm{Z} \le 1}$.
This motivates why the nuclear norm is widely used as a proxy function (or convex relaxation) for \emph{rank minimization}, which otherwise is a hard combinatorial problem.

Its relation to semidefinite optimization --- which explains why the nuclear norm is often called the trace norm --- is that 
\begin{equation}\label{eq:nucNormSDP}
\begin{array}{rl}
\nucnorm{Z} = \displaystyle \mini_{V,W} & t\\
s.t. &  \begin{pmatrix}
  V & Z  \\
  Z^T & W 
\end{pmatrix} \succeq 0 \ \text{ and } \\ &  \tr(V) + \tr(W) \le 2t  \ .
\end{array}
\end{equation}

Here the two optimization variables range over the symmetric matrices $V \in\Sym^{m\times m}$ and $W \in\Sym^{n\times n}$. 
This semidefinite characterization will in fact be the central tool for our algorithmic approach for nuclear norm regularized problems in the following.
The equivalence of the above characterization to the earlier ``factorization'' formulation~(\ref{eq:nucNormMF}) is a consequence of the following simple Lemma~\ref{lem:nucPSD}. The Lemma gives a correspondence between the (rectangular) matrices $Z \in \R^{m \times n}$ of bounded nuclear norm on one hand, and the (symmetric) PSD matrices $X \in \Sym^{(m+n) \times (m+n)}$ of bounded trace on the other hand.

\begin{lemma}[{\cite[Lemma 1]{Fazel:2001vw}}]\label{lem:nucPSD}
For any non-zero matrix $Z \in \R^{m \times n}$ and $t \in \R$, it holds that
\[
\nucnorm{Z}  \le \frac t2 %
\vspace{-0.1em}
\]
if and only if\vspace{-0.4em}
\[
\begin{array}{l}
\exists \text{ symmetric matrices }  V \in \Sym^{m \times m}, W \in \Sym^{n \times n} \\
s.t.\ 
\begin{pmatrix}
  V & Z  \\
  Z^T & W 
\end{pmatrix} \succeq 0  ~\text{  and  }~ 
\tr(V)+\tr(W) \le t \ .
\end{array}
\]
\end{lemma}
\begin{proof}
\fbox{$\Rightarrow$}
Using the characterization~(\ref{eq:nucNormMF}) of the nuclear norm $\nucnorm{Z} = \min_{LR^T = Z} \frac{1}{2}(\frobnorm{L}^{2} + \frobnorm{R}^{2})$ we get that $\exists ~ L,R$, $LR^T = Z$ s.t. $\frobnorm{L}^{2} + \frobnorm{R}^{2} = \tr(LL^T) + \tr(RR^T) \le t$, or in other words we have found a matrix $\big(\begin{smallmatrix} LL^T & Z  \\  Z^T & RR^T \end{smallmatrix}\big) = (\substack{L\\R})(\substack{L\\R})^T \succeq 0$ of trace $\le t$.
\\
\fbox{$\Leftarrow$}
As the matrix $\big(\begin{smallmatrix}V & Z  \\  Z^T & W \end{smallmatrix}\big)$ is symmetric and PSD, it can be (Cholesky) factorized to $(L; R)(L; R)^T$ s.t. $L R^T = Z$ and $t \ge \tr(LL^T) + \tr(RR^T) = \frobnorm{L}^2 + \frobnorm{R}^2$, therefore $\nucnorm{Z}  \le \frac{t}{2}$.
\end{proof}

Interestingly, for characterizing bounded nuclear norm matrices, it does not make any difference whether we enforce an equality or inequality constraint on the trace. This fact will turn out to be useful in order to apply our Algorithm~\ref{alg:greedyHazan} later on.

\begin{corollary}\label{cor:nucPSDeq}
For any non-zero matrix $Z \in \R^{m \times n}$ and $t \in \R$, it holds that
\[
\nucnorm{Z}  \le \frac t2 %
\vspace{-0.1em}
\]
if and only if\vspace{-0.4em}
\[
\begin{array}{l}
\exists \text{ symmetric matrices }  V \in \Sym^{m \times m}, W \in \Sym^{n \times n} \\
s.t.\ 
\begin{pmatrix}
  V & Z  \\
  Z^T & W 
\end{pmatrix} \succeq 0  ~\text{  and  }~ 
\tr(V)+\tr(W) = t \ .
\end{array}
\]
\end{corollary}
\begin{proof}
\fbox{$\Rightarrow$}
From Lemma~\ref{lem:nucPSD} we obtain a matrix $\big(\begin{smallmatrix}V & Z  \\  Z^T & W \end{smallmatrix}\big) =:X \succeq 0$ of trace say $s \le t$. 
If $s < t$, we add $(t - s)$ to the top-left entry of $V$, i.e. we add to $X$ the PSD rank-1 matrix $(t - s) \unit_{1} \unit_{1}^T$ (which again gives a PSD matrix). 
\fbox{$\Leftarrow$} follows directly from Lemma~\ref{lem:nucPSD}.
\end{proof}

\subsubsection{Weighted Nuclear Norm}\label{sec:weightedNucNorm}

A promising weighted nuclear norm regularization for matrix completion was 
recently proposed by~\cite{Salakhutdinov:2010tpa}. For fixed weight vectors $p \in \R^m, q \in
\R^n$, the weighted nuclear norm $\norm{Z}_{nuc(p,q)}$ of $Z \in \R^{m
  \times n}$ is defined as
\[
\norm{Z}_{nuc(p,q)} := \nucnorm{PZQ} ,
\]
where $P = \diag(\sqrt{p}) \in \R^{m \times m}$ denotes the diagonal
matrix whose $i$-th diagonal entry is $\sqrt{p_i}$, and analogously
for $Q = \diag(\sqrt{q}) \in \R^{n \times n}$. Here $p \in \R^m$ is the vector whose
entries are the probabilities $p(i)>0$ that the $i$-th row is observed
in the sampling $\Omega$. Analogously, $q \in \R^n$ contains the
probability $q(j)>0$ for each column $j$. The opposite weighting (using $\frac1{p(i)}$ and $\frac1{q(j)}$ instead of $p(i)$,$q(j)$) has also been suggested by~\cite{Weimer:2008tl}. %

Any optimization problem with a weighted nuclear norm regularization
\begin{equation}\label{eq:constrainedWeightedNuc}
\min_{Z \in \R^{m \times n},\, \norm{Z}_{nuc(p,q)} \,\le\, t/2} f(Z)
\end{equation}
and arbitrary loss function $f$ can therefore be formulated equivalently over
the domain $\nucnorm{PZQ} \le t/2$, such that it reads as (if we
substitute $\bar Z := PZQ$), 
\[
\min_{\bar Z \in \R^{m \times n},\, \nucnorm{\bar Z} \,\le\, t/2} f(P^{-1}\bar Z Q^{-1}).
\]
Hence, we have reduced the task to our standard convex
problem~(\ref{eq:constrainedNuc}) for $\hat f$ that here is defined as
\[
\hat f(X) := f(P^{-1}\bar Z Q^{-1}),
\]
where $X =: \big(\begin{smallmatrix} V &\!\! \bar Z \\  \bar Z^T &\!\!  W \end{smallmatrix}\big)$.
This equivalence implies that any algorithm solving~(\ref{eq:constrainedNuc}) also serves as an algorithm for weighted nuclear norm regularization. 
In particular, Hazan's Algorithm~\ref{alg:greedyHazan} does imply a guaranteed approximation quality of $\varepsilon$ for problem~(\ref{eq:constrainedWeightedNuc}) after $O\big(\frac1\varepsilon\big)$ many rank-1 updates, as we discussed in Section~\ref{sec:algTrace}.
So far, to the best of our knowledge, no approximation guarantees were known for the weighted nuclear norm.

Solution path algorithms (maintaining approximation guarantees when the regularization parameter $t$ changes) as proposed by~\cite{Giesen:2010fx,Giesen:2012uj,Giesen:2012vh}, and the author's PhD thesis~\cite{Jaggi:2011ux}, can also be extended to the case of the weighted nuclear norm.

\subsection{The Max-Norm for Matrices}\label{sec:maxNorm}

We think of the matrix max-norm as a generalization of the vector $\ell_\infty$-norm to the case of positive semidefinite matrices, which we have studied before.

In some matrix completion applications, the max-norm has been observed to provide solutions of better generalization performance than the nuclear norm~\cite{Srebro:2004tj}. Both matrix norms can be seen as a convex surrogate of the rank~\cite{Srebro:2005jj}. %

The \emph{max-norm} $\maxnorm{Z}$ of a rectangular matrix $Z \in \R^{m \times n}$ has a nice characterization in terms of matrix factorizations of $Z$, i.e.
\begin{equation}\label{eq:maxNormMF}
\maxnorm{Z} :=
\min_{LR^T = Z} \max\{ \norm{L}_{2,\infty}^2 , 
                                    \norm{R}_{2,\infty}^2 \} \ ,
\end{equation}
where the number of columns of the factors $L \in \R^{m \times k}$ and $R^{n \times k}$ is
not constrained~\cite{Lee:2010wd}. Here $\norm{L}_{2,\infty}$ is the maximum $\ell_2$-norm of any row $L_{i:}$ of $L$, that is $\norm{L}_{2,\infty} := \max_i \norm{L_{i:}}_2 = \max_i \sqrt{\sum_k L_{ik}^2}$. Compared to the nuclear norm, we therefore observe that the max-norm constrains the maximal Euclidean row-norms of any factorization of the original matrix $Z$, see also~\cite{Srebro:2005jj}.
\footnote{%
Note that the max-norm does \emph{not} coincide with the matrix norm induced by the vector $\norm{.}_\infty$-norm, that is $\norm{Z}_\infty := \sup_{x\ne 0}\frac{\norm{Zx}_\infty}{\norm{x}_\infty}$. The latter matrix norm by contrast is known to be the maximum of the row sums of $Z$ (i.e. the $\ell_1$-norms of the rows). %
}%

An alternative formulation of the max-norm was given by~\cite{Linial:2007gv} and~\cite{Srebro:2005jj}, stating that 
\[
\maxnorm{Z} = \min_{LR^T=Z} (\max_i ||L_{i:}||_2) (\max_i ||R_{i:}||_2) \ .
\]

The dual norm to the max-norm, as given in~\cite{Srebro:2005jj}, is
\[
\begin{array}{rl}
\maxnorm{Z}^* =& \displaystyle\max_{\maxnorm{Y}\le 1} Z\bullet Y \\
                        =& \displaystyle\max_{\substack{k,\\
                                       l_i\in\R^k, \norm{l_i}_2\le1\\
                                       r_j\in\R^k, \norm{r_j}_2\le1}}
                                    \sum_{i,j} Z_{ij} \, l_i^T r_j
\ ,
\end{array}
\]
where the last equality follows from the characterization~(\ref{eq:maxNormMF}).

The relation of the max-norm to semidefinite optimization --- which also explains the naming of the max-norm --- is that 
\begin{equation}\label{eq:maxNormSDP}
\begin{array}{rl}
\maxnorm{Z} = \displaystyle \mini_{V,W} & t\\
s.t. &  \begin{pmatrix}
  V & Z  \\
  Z^T & W 
\end{pmatrix} \succeq 0 \ \text{ and } 
\begin{array}{rl}
V_{ii} \le t &\! \forall i\in[m] ,\\
W_{ii} \le t &\! \forall i\in[n]
\end{array}
\end{array}
\end{equation}

Here the two optimization variables range over the symmetric matrices $V \in\Sym^{m\times m}$ and $W \in\Sym^{n\times n}$, see for example~\cite{Lee:2010wd}. %
As already in the nuclear norm case, this semidefinite characterization will again be the central tool for our algorithmic approach for max-norm regularized problems in the following.
The equivalence of the above characterization to the earlier ``factorization'' formulation~(\ref{eq:maxNormMF}) is a consequence of the following simple Lemma~\ref{lem:maxPSD}. The Lemma gives a correspondence between the (rectangular) matrices $Z \in \R^{m \times n}$ of bounded max-norm on one hand, and the (symmetric) PSD matrices $X \in \Sym^{(m+n) \times (m+n)}$ of uniformly bounded diagonal on the other hand.

\begin{lemma}\label{lem:maxPSD}
For any non-zero matrix $Z \in \R^{n \times m}$ and $t \in \R$:
\[
\maxnorm{Z} \le t  \vspace{-4pt}
\]
if and only if
\vspace{-0.5\baselineskip}
\[
\begin{array}{l}
\exists \text{ symmetric matrices }  V \in \Sym^{m \times m}, W \in \Sym^{n \times n} \\
s.t.\ 
\begin{pmatrix}
  V & Z  \\
  Z^T & W 
\end{pmatrix} \succeq 0  ~\text{  and  }~ 
\begin{array}{rl}
V_{ii} \le t &\! \forall i\in[m] ,\\
W_{ii} \le t &\! \forall i\in[n]
\end{array}
\end{array}
\]
\end{lemma}
\begin{proof}
\fbox{$\Rightarrow$}
Using the above characterization~(\ref{eq:maxNormMF}) of the max-norm, or namely that $\maxnorm{Z} = \min_{LR^T = Z} \max\{ \norm{L}_{2,\infty}^2 , 
                                    \norm{R}_{2,\infty}^2 \}$, we get that $\exists L,R$ with $LR^T = Z$, s.t. $\max\{ \norm{L}_{2,\infty}^2 , 
                                    \norm{R}_{2,\infty}^2 \}
                       = \max\{ \max_i \norm{L_{i:}}_2^2 , 
                                    \max_i \norm{R_{i:}}_2^2 \} \le t$, or in other words we have found a matrix $\big(\begin{smallmatrix} LL^T & Z  \\  Z^T & RR^T \end{smallmatrix}\big) = (L; R)(L; R)^T \succeq 0$ where every diagonal element is at most $t$, that is $\norm{L_{i:}}_2^2 = (LL^T)_{ii} \le t \ \forall i\in[m]$, and $\norm{R_{i:}}_2^2 = (RR^T)_{ii} \le t \ \forall i\in[n]$.
\\
\fbox{$\Leftarrow$}
As the matrix $\big(\begin{smallmatrix}V & Z  \\  Z^T & W \end{smallmatrix}\big)$ is symmetric and PSD, it can be (Cholesky) factorized to $(L; R)(L; R)^T$ s.t. $L R^T = Z$ and $\norm{L_{i:}}_2^2 = (LL^T)_{ii} \le t \ \forall i\in[m]$ and $\norm{R_{i:}}_2^2 = (RR^T)_{ii} \le t \ \forall i\in[n]$, which implies $\maxnorm{Z} \le t$.
\end{proof}

\subsection{Optimizing with Bounded Nuclear Norm and Max-Norm}\label{sec:transform}

Most of the currently known algorithms for matrix factorizations as well as nuclear norm or max-norm regularized optimization problems, such as~(\ref{eq:regNucLagrange}),~(\ref{eq:constrainedNuc}),~(\ref{eq:regMaxLagrange}) or~(\ref{eq:constrainedMax}), do suffer from the following problem: 

In order to optimize the convex objective function $f(Z)$ while controlling the norm $\nucnorm{Z}$ or $\maxnorm{Z}$, the methods instead try to optimize $f(LR^T)$, with respect to both factors $L \in \R^{m\times k}$ and $R \in \R^{n\times k}$, with the corresponding regularization constraint imposed on $L$ and $R$. This approach is of course very tempting, as the constraints on the factors --- which originate from the matrix factorization characterizations ~(\ref{eq:nucNormMF}) and~(\ref{eq:maxNormMF}) --- are simple and in some sense easier to enforce.

\paragraph{Unhealthy Side-Effects of Factorizing.}
However, there is a significant price to pay: Even if the objective function $f(Z)$ is \emph{convex} in $Z$, the very same function expressed as a function $f(LR^T)$ of both the factor variables $L$ and $R$ becomes a severely \emph{non-convex} problem, naturally consisting of a large number of saddle-points (consider for example just the smallest case $L,R \in \R^{1 \times 1}$ together with the identity function $f(Z) = Z \in \R$). 

The majority of the currently existing methods such as for example~\cite{Rennie:2005vn,Lin:2007ea} and later also 
\cite{Paterek:2007va,
Zhou:2008hh,%
Koren:2009jg,
Takacs:2009ug,
Ilin:2010tl,
Gemulla:2011wk}
is of this ``alternating'' gradient descent type, where at each step one of the factors $L$ and $R$ is kept fixed, and the other factor is updated by e.g. a gradient or stochastic gradient step. Therefore, despite working well in many practical applications, all these mentioned methods can get stuck in local minima --- and so are theoretically not well justified, see also the discussion in~\cite{DeCoste:2006wt}. %
 
The same issue also comes up for max-norm optimization, where for example~\cite{Lee:2010wd} optimize over the non-convex factorization~(\ref{eq:maxNormMF}) for bounded max-norm.

Concerning the fixed rank of the factorization,~\cite{Gillis:2010wua} have shown that finding the optimum under a rank constraint (even if the rank is one) is NP-hard (here the used function $f$ was the standard squared error on an incomplete matrix). On the positive side,~\cite{Burer:2003fg} have shown that if the rank $k$ of the factors $L$ and $R$ exceeds the rank of the optimum solution $X^*$, then --- in some cases --- it can be guaranteed that the local minima (or saddle points) are also global minima. However, in nearly all practical applications it is computationally infeasible for the above mentioned methods to optimize with the rank $k$ being in the same order of magnitude as the original matrix size $m$ and $n$ (as e.g. in the Netflix problem, such factors $L,R$ could possibly not even be stored on a single machine\footnote{%
Algorithm~\ref{alg:greedyHazan} in contrast does never need to store a full estimate matrix $X$, but instead just keeps the rank-1 factors $v$ obtained in each step, maintaining a factorized representation of $X$.
}%
).

\paragraph{Relief: Optimizing Over an Equivalent Convex Problem.}
Here we simply overcome this problem by using the transformation to semidefinite matrices, which we have outlined in the above Corollary~\ref{cor:nucPSDeq} and Lemma~\ref{lem:maxPSD}. These bijections of bounded nuclear and max-norm matrices to the PSD matrices over the corresponding natural convex domains do allow us to directly optimize a convex problem, avoiding the factorization problems explained above. We describe this simple trick formally in the next two Subsections~\ref{subsec:transfNuc} and~\ref{subsec:transfMax}.

\paragraph{But what if you really need a Matrix Factorization?}
In some applications (such as for example embeddings or certain collaborative filtering problems) of the above mentioned regularized optimization problems over $f(Z)$, one would still want to obtain the solution (or approximation) $Z$ in a factorized representation, that is $Z=LR^T$.
We note that this is also straight-forward to achieve when using our transformation: An explicit factorization of any feasible solution to the transformed problem~(\ref{eq:optTrEq1}) or~(\ref{eq:optMatInfLe1}) --- if needed --- can always be directly obtained since $X \succeq 0$. 

Alternatively, algorithms for solving the transformed problem~(\ref{eq:optTrEq1}) can directly maintain the approximate solution $X$ in a factorized representation (as a sum of rank-1 matrices), as achieved for example by Algorithms~\ref{alg:greedyHazan} and~\ref{alg:maxDiag}.

\subsubsection{Optimization with a Nuclear Norm Regularization}\label{subsec:transfNuc}
Having Lemma~\ref{lem:nucPSD} at hand, we immediately get to the crucial observation of this section, allowing us to apply Algorithm~\ref{alg:greedyHazan}: %

Any optimization problem over bounded nuclear norm matrices~(\ref{eq:constrainedNuc}) is in fact equivalent to a standard bounded trace semidefinite problem~(\ref{eq:optTrEq1}). The same transformation also holds for problems with a bound on the \emph{weighted} nuclear norm, as given in~(\ref{eq:constrainedWeightedNuc}).

\begin{corollary}\label{cor:nucTrans}
Any nuclear norm regularized problem of the form~(\ref{eq:constrainedNuc}) is equivalent to a bounded trace convex problem of the form~(\ref{eq:optTrEq1}), namely
\begin{equation}
\begin{array}{rl}
   \displaystyle\mini_{X \in \Sym^{(m+n) \times (m+n)}} & \hat f(X)  \\
    s.t. &  \tr(X) = t \ ,\\
         &  X \succeq 0
\end{array}
\end{equation}
where $\hat f$ is defined by $\hat f(X) := f(Z)$ for $Z \in \R^{m\times n}$ being the upper right part of the symmetric matrix $X$. Formally we again think of $X \in \Sym^{(n+m) \times (n+m)}$ as consisting of the four parts $X =:  \small\begin{pmatrix}V & Z  \\  Z^T & W \end{pmatrix}$ with $V \in \Sym^{m \times m}, W \in \Sym^{n \times n}$ and $Z \in \R^{m\times n}$.
\end{corollary}
Here ``equivalent'' means that for any feasible point of one problem, we have a feasible point of the other problem, attaining the same objective value. %
The only difference to the original formulation~(\ref{eq:optTrEq1}) is that the function argument $X$ needs to be rescaled by $\frac1t$ in order to have unit trace, which however is a very simple operation in practical applications.
Therefore, we can directly apply Hazan's Algorithm~\ref{alg:greedyHazan} for any max-norm regularized problem as follows:

\begin{algorithm}[h!]
  \caption{Nuclear Norm Regularized Solver}
  \label{alg:nucSolv}
\begin{algorithmic}
  \STATE {\bfseries Input:} A convex nuclear norm regularized problem~(\ref{eq:constrainedNuc}), 
  \STATE \hspace{2cm} target accuracy $\varepsilon$
  \STATE {\bfseries Output:} $\varepsilon$-approximate solution for problem~(\ref{eq:constrainedNuc})
  \STATE  {\bfseries 1.} Consider the transformed symmetric problem for $\hat f$,
  \STATE \hspace{5cm} as given by Corollary~\ref{cor:nucTrans}
  \STATE  {\bfseries 2.} Adjust the function $\hat f$ so that it first rescales its argument by $t$
  \STATE {\bfseries 3.} Run Hazan's Algorithm~\ref{alg:greedyHazan} for $\hat f(X)$ over the domain $X \in \Spectahedron$.
\end{algorithmic}
\end{algorithm}

Using our analysis of Algorithm~\ref{alg:greedyHazan} from Section~\ref{subsec:algHazan}, we see that Algorithm~\ref{alg:nucSolv} runs in time \emph{near linear} in the number $N_f$ of non-zero entries of the gradient $\nabla f$. This makes it very attractive in particular for recommender systems applications and matrix completion, where $\nabla f$ is a sparse matrix (same sparsity pattern as the observed entries), which we will discuss in more detail in Section~\ref{sec:applicationMatCompl}.

\begin{corollary}\label{cor:algNuc}
After at most $O\left(\frac1\varepsilon\right)$ many iterations (i.e. approximate eigenvalue computations), Algorithm~\ref{alg:nucSolv} obtains a solution that is $\varepsilon$ close to the optimum of~(\ref{eq:constrainedNuc}). The algorithm requires a total of $\tilde
  O\left(\frac{N_f}{\varepsilon^{1.5}}\right)$ arithmetic operations (with high probability).
\end{corollary}
\begin{proof}
We use the transformation from Corollary~\ref{cor:nucTrans} and then rescale all matrix entries by $\frac{1}{t}$. Then result then follows from Corollary \vref{cor:algHazanRunningTime} on the running time of Hazan's algorithm.
\end{proof}

The fact that each iteration of our algorithm is computationally very cheap --- consisting only of the computation of an approximate eigenvector --- strongly contrasts the existing ``proximal gradient'' and ``singular value thresholding'' methods \cite{Ganesh:2009dn,Ji:2009vd,Ma:2009kv,Liu:2009wj,Cai:2010vw,Toh:2010vx}, which \emph{in each step} need to compute an entire SVD. 
Such a single incomplete SVD computation (first $k$ singular vectors) amounts to the same computational cost as an entire run of our algorithm (for $k$ steps).
Furthermore, those existing methods which come with a theoretical guarantee, in their analysis assume that all SVDs used during the algorithm are exact, which is not feasible in practice. By contrast, our analysis is rigorous even if the used eigenvectors are only $\varepsilon'$-approximate.

Another nice property of Hazan's method is that the returned solution is guaranteed to be simultaneously of low rank ($k$ after $k$ steps), and that by incrementally adding the rank-1 matrices $v_{k} v_{k}^T$, the algorithm automatically maintains a matrix factorization of the approximate solution. 

Also, Hazan's algorithm, as being an instance of our presented general framework, is designed to automatically stay within the feasible region $\Spectahedron$, where most of the existing methods do need a projection step to get back to the feasible region (as e.g.~\cite{Lin:2007ea,Liu:2009wj}), making both the theoretical analysis and implementation more complicated. 

\subsubsection{Optimization with a Max-Norm Regularization}\label{subsec:transfMax}
The same approach works analogously for the max-norm, by using Lemma~\ref{lem:maxPSD} in order to apply Algorithm~\ref{alg:maxDiag}:

Any optimization problem over bounded max-norm matrices~(\ref{eq:constrainedMax}) is in fact equivalent to a semidefinite problem~(\ref{eq:optMatInfLe1}) over the ``box'' of matrices where each element on the diagonal is bounded above by $t$. We think of this domain as generalizing the positive cube of vectors, to the PSD matrices. %

\begin{corollary}\label{cor:maxTrans}
Any max-norm regularized problem of the form~(\ref{eq:constrainedMax}) is equivalent to a bounded diagonal convex problem of the form~(\ref{eq:optMatInfLe1}), i.e.,
\begin{equation}\label{eq:maxDiag}
\begin{array}{rl}
   \displaystyle\mini_{X \in \Sym^{(m+n) \times (m+n)}} & \hat f(X)  \\
    s.t. &  X_{ii} \le 1 ~~~\forall i ,\\
         &  X \succeq 0
\end{array}
\end{equation}
where $\hat f$ is defined by $\hat f(X) := f(Z)$ for $Z \in \R^{m\times n}$ being the upper right part of the symmetric matrix $X$. Formally we again think of any $X \in \Sym^{(n+m) \times (n+m)}$ as consisting of the four parts $X =:  \small\begin{pmatrix}V & Z  \\  Z^T & W \end{pmatrix}$ with $V \in \Sym^{m \times m}, W \in \Sym^{n \times n}$ and $Z \in \R^{m\times n}$.
\end{corollary}
Again the only difference to the original formulation~(\ref{eq:optMatInfLe1}) is that the function argument $X$ needs to be rescaled by $\frac1t$ in order to have the diagonal bounded by one, which however is a very simple operation in practical applications. 
This means we can directly apply Algorithm~\ref{alg:maxDiag} for any max-norm regularized problem as follows:

\begin{algorithm}[h!]
  \caption{Max-Norm Regularized Solver}
  \label{alg:maxSolv}
\begin{algorithmic}
  \STATE {\bfseries Input:} A convex max-norm regularized problem~(\ref{eq:constrainedMax}), 
  \STATE \hspace{2cm} target accuracy $\varepsilon$
  \STATE {\bfseries Output:} $\varepsilon$-approximate solution for problem~(\ref{eq:constrainedMax})
  \STATE  {\bfseries 1.} Consider the transformed symmetric problem for $\hat f$,
  \STATE \hspace{5cm} as given by Corollary~\ref{cor:maxTrans}
  \STATE  {\bfseries 2.} Adjust the function $\hat f$ so that it first rescales its argument by $t$
  \STATE {\bfseries 3.} Run Algorithm~\ref{alg:maxDiag} for $\hat f(X)$ over the domain $X \in \MaxCutPolytope$.
\end{algorithmic}
\end{algorithm}

Using the analysis of our new Algorithm~\ref{alg:maxDiag} from Section~\ref{subsec:algHazan}, we obtain the following guarantee:

\begin{corollary}\label{cor:algMax}
After $\left\lceil\frac{8C_{f}}{\varepsilon}\right\rceil$ many iterations, Algorithm~\ref{alg:maxSolv} obtains a solution that is $\varepsilon$ close to the optimum of~(\ref{eq:constrainedMax}).
\end{corollary}
\begin{proof}
We use the transformation from Corollary~\ref{cor:maxTrans} and then rescale all matrix entries by $\frac{1}{t}$. Then the running time of the algorithm follows from Theorem~\ref{thm:algMaxDiag}.
\end{proof}

\paragraph{Maximum Margin Matrix Factorizations.}
In the case of matrix completion, the ``loss'' function~$f$ is defined as measuring the error from $X$ to some fixed observed matrix, but just at a small fixed set of ``observed'' positions of the matrices. As we already mentioned, semidefinite optimization over $X$ as above can always be interpreted as finding a \emph{matrix factorization}, as a symmetric PSD matrix $X$ always has a (unique) Cholesky factorization.

Now for the setting of matrix completion, it is known that the above described optimization task under bounded max-norm, can be geometrically interpreted as learning a maximum margin separating hyperplane for each user/movie. %
In other words the factorization problem decomposes into a collection of SVMs, one for each user or movie, if we think of the corresponding other factor to be fixed for a moment~\cite{Srebro:2004tj}. We will discuss matrix completion in more detail in Section~\ref{sec:applicationMatCompl}.

\paragraph{Other Applications of Max-Norm Optimization.}
Apart from matrix completion, optimization problems employing the max-norm have other prominent applications in spectral methods, spectral graph properties, low-rank recovery, and combinatorial problems such as Max-Cut.

\subsection{Applications}\label{sec:applicationMatCompl}

Our Algorithm~\ref{alg:nucSolv} directly applies to arbitrary nuclear norm regularized problems of the form~(\ref{eq:constrainedNuc}). Since the nuclear norm is in a sense the most natural generalization of the sparsity-inducing $\ell_1$-norm to the case of low rank matrices  (see also the discussion in the previous chapters) there are many applications of this class of optimization problems.

\subsubsection{Robust Principal Component Analysis}
One prominent example of a nuclear norm regularized problem in the area of dimensionality reduction is given by the technique of \emph{robust PCA} as introduced by~\cite{Candes:2011bd}, also called principal component pursuit, which is the optimization task
\begin{equation}\label{eq:robustPCA}
  \min_{Z \in \R^{m \times n}} \nucnorm{Z} + \mu \norm{M-Z}_1  \ .
\end{equation}
Here $M\in \R^{m \times n}$ is the given data matrix, and $\norm{.}_1$ denotes the entry-wise $\ell_1$-norm.
By considering the equivalent constrained variant $\nucnorm{Z}\le \frac t2$ instead, we obtain a problem the form~(\ref{eq:constrainedNuc}), suitable for our Algorithm~\ref{alg:nucSolv}. 
However, since the original objective function $f(Z) = \norm{M-Z}_1$ is not differentiable, a smoothed version of the $\ell_1$-norm has to be used instead. This situation is analogous to the hinge-loss objective in maximum margin matrix factorization~\cite{Srebro:2004tj}.

Existing algorithms for robust PCA do usually require a complete (and exact) SVD in each iteration, as e.g.~\cite{Toh:2010vx,Aybat:2011vd}, and are often harder to analyze compared to our approach.
The first algorithm with a convergence guarantee of $O\big(\frac1\varepsilon\big)$ was given by~\cite{Aybat:2011vd}, requiring a SVD computation per step. Our Algorithm~\ref{alg:nucSolv} obtains the same guarantee in the same order of steps, but only requires a single approximate eigenvector computation per step, which is significantly cheaper.

Last but not least, the fact that our algorithm delivers approximate solutions to~(\ref{eq:robustPCA}) of rank $O\left(\frac1\varepsilon\right)$ will be interesting for practical dimensionality reduction applications, as it re-introduces the important concept of low-rank factorizations as in classical PCA. In other words our algorithm produces an embedding into at most $O\left(\frac1\varepsilon\right)$ many new dimensions, which is much easier to deal with in practice compared to the full rank $n$ solutions resulting from the existing solvers for robust PCA, see e.g.~\cite{Candes:2011bd} and the references therein.

We did not yet perform practical experiments for robust PCA, but chose to demonstrate the practical performance of Algorithm~\ref{alg:greedyHazan} for matrix completion problems first.

\subsubsection{Matrix Completion and Low Norm Matrix Factorizations}

For matrix completion problems as for example in collaborative filtering and recommender systems~\cite{Koren:2009jg}, our algorithm is particularly suitable as it retains the sparsity of the observations, and constructs the solution in a factorized way. In the setting of a partially observed matrix such as in the Netflix case, the loss function $f(X)$ only depends on the observed positions, which are very sparse, so $\nabla f(X)$ --- which is all we need for our algorithm --- is also sparse.

We want to approximate a partially given matrix $Y$ (let $\Omega$ be the set of known training entries of the matrix) by a product $Z=LR^T$ such that some convex loss function $f(Z)$ is minimized. By $\Omega_{test}$ we denote the unknown test entries of the matrix we want to predict. 

\paragraph{Complexity.}
Just recently it has been shown that the standard low-rank matrix completion problem --- that is finding the best approximation to an incomplete matrix by the standard $\ell_2$-norm --- is an NP-hard problem, if the rank of the approximation is constrained. The hardness is claimed to hold even for the rank 1 case~\cite{Gillis:2010wua}.
In the light of this hardness result, the advantage of relaxing the rank by replacing it by the nuclear norm (or max-norm) is even more evident.
Our near linear time Algorithm~\ref{alg:nucSolv} relies on a convex optimization formulation and does indeed deliver an guaranteed $\varepsilon$-accurate solution for the nuclear norm regularization, for arbitrary $\varepsilon>0$. 
Such a guarantee is lacking for the ``alternating'' descent heuristics such as~\cite{Rennie:2005vn,Lin:2007ea,Paterek:2007va,Zhou:2008hh,Koren:2009jg,Takacs:2009ug,Ilin:2010tl,Gemulla:2011wk,Salakhutdinov:2010tpa,Lee:2010wd,Recht:2011wv} (which build upon the non-convex factorized versions~(\ref{eq:nucNormMF}) and~(\ref{eq:maxNormMF}) while constraining the rank of the used factors~$L$ and~$R$).

\paragraph{Different Regularizations.}
Regularization by the weighted nuclear norm is observed by~\cite{Salakhutdinov:2010tpa} to provide better generalization performance than the classical nuclear norm. As it can be simply reduced to the nuclear norm, see Section~\ref{sec:weightedNucNorm}, our Algorithm~\ref{alg:nucSolv} can directly be applied in the weighted case as well.
On the other hand, experimental evidence also shows that the max-norm sometimes provides better generalization performance than the nuclear norm~\cite{Srebro:2004tj,Lee:2010wd}. For any convex loss function, our Algorithm~\ref{alg:maxSolv} solves the corresponding max-norm regularized matrix completion task.

\paragraph{Different Loss Functions.}
Our method applies to any convex loss function on a low norm matrix factorization problem, and we will only mention two loss functions in particular:

\emph{Maximum Margin Matrix Factorization} (MMMF)~\cite{Srebro:2004tj} can directly be solved by our Algorithm~\ref{alg:nucSolv}. Here the original (soft margin) formulation is the trade-off formulation~(\ref{eq:regNucLagrange}) with $f(Z) := \sum_{ij \in \Omega} \abs{Z_{ij} - y_{ij}}$ being the hinge or $\ell_{1}$-loss. Because this function is not differentiable, the authors recommend using the differentiable smoothed hinge loss instead.

When using the standard squared loss function $f(Z) := %
\sum_{ij \in \Omega} (Z_{ij} - y_{ij})^{2}$, the problem is known as \emph{Regularized Matrix Factorization}~\cite{Wu:2007tm}, and both our algorithms directly apply. This loss function is widely used in practice, has a very simple gradient, and is the natural matrix generalization of the $\ell_{2}$-loss (notice the analogous Lasso and regularized least squares formulation). The same function is known as the rooted mean squared error, %
which was the quality measure used in the Netflix competition. We write
$\textit{RMSE}_{train}$ and $\textit{RMSE}_{test}$
for the rooted error on the training ratings $\Omega$ and test ratings $\Omega_{test}$ respectively.

\paragraph{Running time and memory.}
From Corollary~\ref{cor:algNuc} we have that the running time of our nuclear norm optimization Algorithm~\ref{alg:nucSolv} is linear in the size of the input: Each matrix-vector multiplication in Lanczos' or the power method exactly costs $\abs{\Omega}$ (the number of observed positions of the matrix) operations, and we know that in total we need at most $O\left(1/\varepsilon^{1.5}\right)$ many such matrix-vector multiplications.

Also the memory requirements are very small: Either we store the entire factorization of $X^{(k)}$ (meaning the $O\left(\frac1\varepsilon\right)$ many vectors $v^{(k)}$) --- which is still much smaller than the full matrix $X$ --- or then instead we can only update and store the prediction values $X^{(k)}_{ij}$ for $ij \in \Omega \cup \Omega_{test}$ in each step. This, together with the known ratings $y_{ij}$ determines the sparse gradient matrix $\nabla f(X^{(k)})$ during the algorithm. Therefore, the total memory requirement is only $\abs{\Omega \cup \Omega_{test}}$ (the size of the output) plus the size $(n+m)$ of a single feature vector $v$.

\paragraph{The constant $C_{f}$ in the running time of Algorithm~\ref{alg:greedyHazan}.}
One might ask if the constant hidden in the $O(\frac1\varepsilon)$ number of iterations is indeed controllable. Here we show that for the standard squared error on any fixed set of observed entries $\Omega$, this is indeed the case. For more details on the constant $C_f$, we refer the reader to Sections~\ref{subsec:Curvature} and~\ref{subsec:algHazan}.
\begin{lemma}\label{lem:cf}
For the squared error 
$f(Z) = \frac{1}{2} \sum_{ij \in \Omega} (Z_{ij} - y_{ij})^{2}$ over the spectahedron $\Spectahedron$, it holds that 
$C_{\hat f} \le 1.$
\vspace{-0.3\baselineskip}
\end{lemma}
\begin{proof}
In Lemma~\ref{lem:CfHessBound}, we have seen that the constant $C_{\hat f}$ is upper bounded by half the diameter of the domain, times the largest eigenvalue of the Hessian
$\nabla^2 \hat f(\vec X)$. Here we consider $\hat f$ as a function on vectors $\vec X \in \R^{n^2}$ corresponding to the matrices $X\in\Sym^{n\times n}$.
However for the squared error as in our case here, the Hessian will be a diagonal matrix. One can directly compute that the diagonal entries of $\nabla^2 \hat f(\vec X)$ are $1$ at the entries corresponding to $\Omega$, and zero everywhere else.
Furthermore, the squared diameter of the spectahedron is upper bounded by $2$, as we have shown in Lemma~\ref{lem:frobDiamSpect}.
Therefore $C_{\hat f} \le 1$ for the domain $\Spectahedron$.
\end{proof}

If the domain is the scaled spectahedron $t\cdot \Spectahedron$ as used in our Algorithm~\ref{alg:nucSolv}, then the squared diameter of the domain is $2t^2$, compare to Lemma~\ref{lem:frobDiamSpect}.
This means that the curvature is upper bounded by $C_{\hat f} \le t^2$ in this case. Alternatively, the same bound for the curvature of $\tilde f(X):=\hat f(tX)$ can be obtained along the same lines as for the spectahedron domain in the previous lemma, and the same factor of $t^2$ will be the scaling factor of the Hessian, resulting from the chain-rule for taking derivatives.

\subsubsection{The Structure of the Resulting Eigenvalue Problems}

For the actual computation of the approximate largest eigenvector in Algorithm~\ref{alg:greedyHazan}, i.e. the internal procedure $\textsc{ApproxEV}\left(-\nabla \hat f (X^{(k)}), \frac{2C_{\hat f}}{k+2} \right)$, either Lanczos' method or the power method (as in PageRank, see e.g.~\cite{Berkhin:2005ww}) can be used.
In our Theorem~\ref{thm:powerLanczos} of Section~\ref{subsec:algHazan}, we stated that both the power method as well as Lanczos' algorithm do provably obtain the required approximation quality in a bounded number of steps if the matrix is PSD, with high probability, see also~\cite{Kuczynski:1992va,Arora:2005uh}. 

Both methods are known to scale well to very large problems and can be parallelized easily, as each iteration consists of just one matrix-vector multiplication.
However, we have to be careful that we obtain the eigenvector for the \emph{largest} eigenvalue which is not necessarily the principal one (largest in absolute value). In that case the spectrum can be shifted by adding an appropriate constant to the diagonal of the matrix. 

For arbitrary loss function $f$, the gradient $-\nabla \hat f(X)$, which is the matrix whose largest eigenvector we have to compute in the algorithm, is always a symmetric matrix of the block form $\nabla \hat f(X) = \begin{pmatrix} 0 &\!\! G \\ G^T &\!\!  0 \end{pmatrix}$ for $G = \nabla f(Z)$, when $X =:  \small\begin{pmatrix}V & Z  \\  Z^T & W \end{pmatrix}$.
In other words $\nabla \hat f(X)$ is the adjacency matrix of a weighted \emph{bipartite graph}. One vertex class corresponds to the $n$ rows of the original matrix $X_{2}$ (\emph{users} in recommender systems), the other class corresponds to the $m$ columns (\emph{items} or \emph{movies}). It is easy to see that the spectrum of $\nabla \hat f$ is always \emph{symmetric}: Whenever $(\begin{smallmatrix} v \\ w \end{smallmatrix})$ is an eigenvector for some eigenvalue $\lambda$, then $(\begin{smallmatrix} v \\ -w \end{smallmatrix})$ is an eigenvector for $-\lambda$. 

Hence, we have exactly the same setting as in the established Hubs and Authorities (HITS) model~\cite{Kleinberg:1999wa}. The first part of any eigenvector is always an eigenvector of the hub matrix $G^T G$, and the second part is an eigenvector of the authority matrix $G G^T$. 

\paragraph{Repeated squaring.} In the special case that the matrix $G$ is very rectangular ($n \ll m$ or $n\gg m$), one of the two square matrices $G^T G$ or $G G^T$ is very small. Then it is known that %
one can obtain an exponential speed-up in the power method by repeatedly squaring the smaller one of the matrices, analogously to the ``square and multiply''-approach for computing large integer powers of real numbers. In other words we can perform $O(\log \frac{1}{\varepsilon})$ many \emph{matrix-matrix} multiplications instead of $O(\frac{1}{\varepsilon})$ \emph{matrix-vector} multiplications. %

\subsubsection{Relation to Simon Funk's SVD Method}\label{sec:simonFunk}
Interestingly, our proposed framework can also be seen as a theoretically justified variant of Simon Funk's~\cite{Webb:2006vm} and related approximate SVD methods, which were used as a building block by most of the teams participating in the Netflix competition (including the winner team). Those methods have been further investigated by~\cite{Paterek:2007va,Takacs:2009ug} and also~\cite{Kurucz:2007wr}, which already proposed a heuristic using the HITS formulation.
These approaches are algorithmically extremely similar to our method, although they are aimed at a slightly different optimization problem, and do \emph{not} directly guarantee bounded nuclear norm. Recently,~\cite{Salakhutdinov:2010tpa} observed that Funk's algorithm can be seen as stochastic gradient descent to optimize~(\ref{eq:regNucLagrange}) when the regularization term is replaced by a \emph{weighted} variant of the nuclear norm.

Simon Funk's method considers the standard squared loss function $\hat f(X) = \frac{1}{2} \sum_{ij \in S} (X_{ij} - y_{ij})^{2}$, and finds the new rank-1 estimate (or feature) $v$ by iterating $v := v + \lambda (-\nabla \hat f(X) v - K v)$, or equivalently 
\begin{equation}\label{eq:simonFunk}
v := \lambda \left(-\nabla \hat f(X) + \left(\frac{1}{\lambda} - K \right) \id \right)  v \ ,
\end{equation}
a fixed number of times. Here $\lambda$ is a small fixed constant called the learning rate. Additionally a decay rate $K > 0$ is used for regularization, i.e. to penalize the magnitude of the resulting feature~$v$. This matrix-vector multiplication formulation~(\ref{eq:simonFunk}) is equivalent to a step of the power method applied within our framework\footnote{%
Another difference of our method to Simon Funk's lies in the stochastic gradient descent type of the latter, i.e. ``immediate feedback'':
During each matrix multiplication, it already takes the modified current feature $v$ into account when calculating the loss $\hat f(Z)$, whereas our Algorithm~\ref{alg:greedyHazan} alters $Z$ only after the eigenvector computation is finished.
}%
, and for small enough learning rates $\lambda$ the resulting feature vector will converge to the largest eigenvector of $-\nabla \hat f(Z)$.

However in Funk's method, the magnitude of each new feature strongly depends on the starting vector $v_{0}$, the number of iterations, the learning rate $\lambda$ as well as the decay $K$, making the convergence very sensitive to these parameters. This might be one of the reasons that so far no results on the convergence speed could be obtained. Our method is free of these parameters, the $k$-th new feature vector is always a unit vector scaled by $\frac{1}{\sqrt{k}}$. Also, we keep the Frobenius norm $\frobnorm{U}^2+\frobnorm{V}^2$ of the obtained factorization exactly fixed during the algorithm, whereas in Funk's method --- which has a different optimization objective --- this norm strictly increases with every newly added feature.

Our described framework therefore gives theoretically justified variant of the experimentally successful method~\cite{Webb:2006vm} and its related variants such as~\cite{Kurucz:2007wr,Paterek:2007va,Takacs:2009ug}.

\subsection{Experimental Results}

We run our algorithm for the following standard datasets%
\footnote{See \href{http://www.grouplens.org}{www.grouplens.org} and \href{http://archive.ics.uci.edu/ml/datasets/Netflix+Prize}{archive.ics.uci.edu/ml}.%
} for matrix completion problems, using the squared error function.

\begin{small}
\begin{center}
\begin{tabular}{|l|r|r|r|}
\hline					
 \textit{dataset} & \textit{\#ratings} & $n$ & $m$  \\
\hline			
\textsf{MovieLens 100k} &
$10^5$ & $943$ & $1682$\\
\textsf{MovieLens 1M} &
$10^6$ & $6040$ & $3706$\\
\textsf{MovieLens 10M}%
 &
$10^7$ & $69878$ & $10677$\\
\textsf{Netflix} &
$10^8$ & $480189$ & $17770$\\
\hline
\end{tabular}
\end{center}
\end{small}

Any eigenvector method can be used as a black-box in our algorithm. To keep the experiments simple, we used the power method%
\footnote{We used the power method starting with the uniform unit vector. $\frac{1}{2}$ of the approximate eigenvalue corresponding to the previously obtained feature $v_{k-1}$ was added to the matrix diagonal to ensure good convergence.%
}, and performed $0.2 \cdot k$ power iterations in step $k$. If not stated otherwise, the only optimization we used is the improvement by averaging the old and new gradient as explained in Section~\ref{subsec:imprAlgos}.
All results were obtained by our (single-thread) implementation in \textsf{Java 6} on a 2.4 GHz \textsf{Intel C2D} laptop.

\paragraph{Sensitivity.}
The generalization performance of our method is relatively stable under changes of the regularization parameter, see Figure~\ref{fig:sensitivity}:
\begin{figure}[ht]
\vspace{-0.2em}
\begin{center}
\includegraphics[width=0.5\columnwidth]{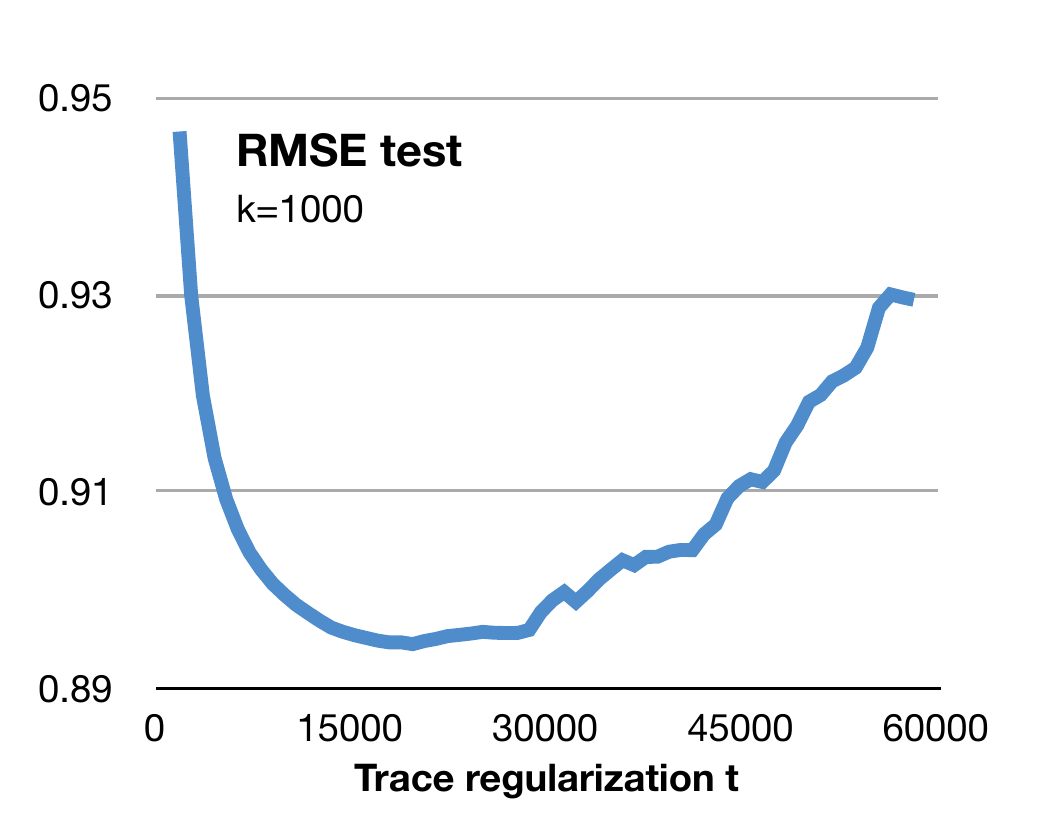}\vspace{-0.2em}
\caption{Sensitivity of the method on the choice of
the regularization parameter $t$ in~(\ref{eq:constrainedNuc}), on \textsf{MovieLens 1M}.}
\label{fig:sensitivity}
\end{center}
\vspace{-1em}
\end{figure}

\paragraph{Movielens.}
Table~\ref{tbl:movielens} reports the running times of our algorithm on the three \textsf{MovieLens} datasets. Our algorithm gives an about 5.6 fold speed increase over the reported timings by~\cite{Toh:2010vx}, which is a very similar method to~\cite{Ji:2009vd}.~\cite{Toh:2010vx} already improves the ``singular value thresholding'' methods~\cite{Cai:2010vw} and~\cite{Ma:2009kv}.
For MMMF,~\cite{Rennie:2005vn} report an optimization time of about
$5$ hours on the \textsf{1M} dataset, but use the different smoothed
hinge loss function so that the results cannot be directly compared. \cite{Ma:2009kv},~\cite{Srebro:2003ui} and~\cite{Ji:2009vd} only obtained results on much smaller datasets.

\begin{table}[ht]
\caption{Running times $t_{\text{our}}$ (in seconds) of our algorithm on the three \textsf{MovieLens} datasets compared to the reported timings $t_{\text{TY}}$ of~\cite{Toh:2010vx}. The ratings $\{1,\dots,5\}$ were used as-is and \emph{not} normalized to any user and/or movie means. In accordance with~\cite{Toh:2010vx}, 50\% of the ratings were used for training, the others were used as the test set. Here \textit{NMAE} is the mean absolute error, times $\frac{1}{5-1}$, over the total set of ratings. %
\textit{k} is the number of iterations of our algorithm, \textit{\#mm} is the total number of sparse matrix-vector multiplications performed, and \textit{tr} is the used trace parameter $t$ in~(\ref{eq:constrainedNuc}). They used \textsf{Matlab/PROPACK} on an \textsf{Intel Xeon} 3.20 GHz processor.
}
\label{tbl:movielens}
\vskip -0.1in
\begin{center}
\begin{small}
\begin{tabular}{|l|r|r|r|r|r|r|}
\hline
 & \textit{NMAE} & $t_{\text{TY}}$ & $t_{\text{our}}$ & \textit{k} & \textit{\#mm} & \textit{tr} \\
\hline	
\textsf{100k} &
0.205 &%
7.39 &%
0.156 & 15 & 33 & 9975\\ %
\textsf{1M} &
0.176 &%
24.5 &%
1.376 & 35 & 147 & 36060\\ %
\textsf{10M} &
0.164 &%
202 &%
36.10 & 65 & 468 & 281942 \\ %
\hline
\end{tabular}
\end{small}
\end{center}
\vskip -0.1in
\end{table}

In the following experiments on the \textsf{MovieLens} and \textsf{Netflix} datasets we have pre-normalized all training ratings to the simple average $\frac{\mu_{i}+\mu_{j}}{2}$ of the user and movie mean values, for the sake of being consistent with comparable literature.

For \textsf{MovieLens 10M}, we used partition $r_b$ provided with the dataset (10 test ratings per user). The regularization parameter $t$ was set to $48333$. We obtained a $\textit{RMSE}_{test}$ of 0.8617 after $k=400$ steps, in a total running time of 52 minutes (16291 matrix multiplications). Our best $\textit{RMSE}_{test}$ value was 0.8573, compared to 0.8543 obtained by~\cite{Lawrence:2009wr} using their non-linear improvement of MMMF.

\paragraph{Algorithm Variants.}
Comparing the proposed algorithm variants from Section~\ref{subsec:imprAlgos}, Figure~\ref{fig:3algo-variants} demonstrates moderate improvements compared to our original Algorithm~\ref{alg:nucSolv}.

\begin{figure}[!h]
\vspace{-0.2em}
\begin{center}
\includegraphics[width=0.6\columnwidth]{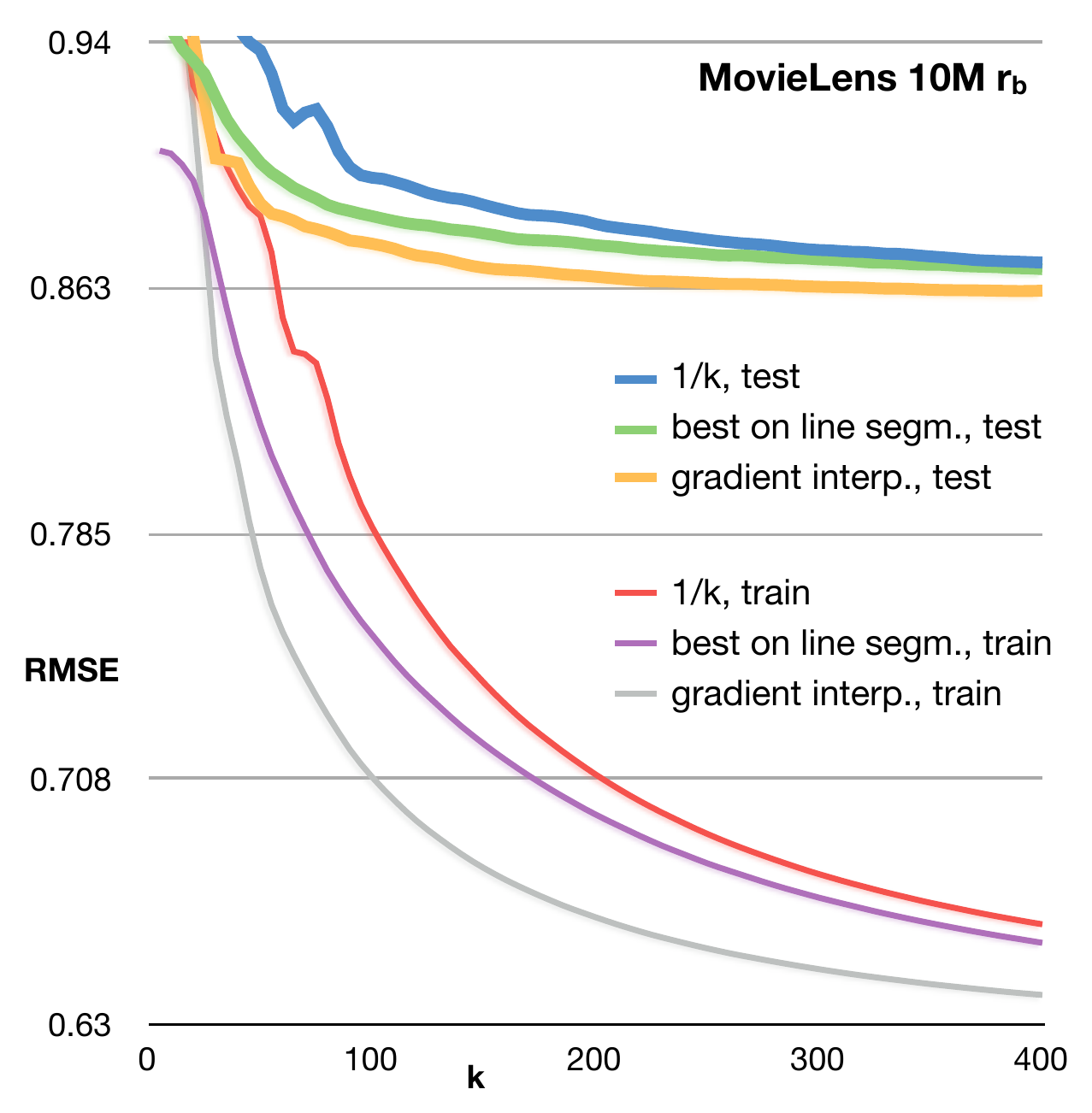}\vspace{-0.2em}
\caption{Improvements for the two algorithm variants described in Section~\ref{subsec:imprAlgos}, when running on \textsf{MovieLens 10M}. The thick lines above indicate the error on the test set, while the thinner lines indicate the training error.}
\label{fig:3algo-variants}
\end{center}
\vspace{-0.2em}
\end{figure}

\paragraph{Netflix.}
Table~\ref{tbl:netflix} compares our method to the two ``hard impute'' and ``soft impute'' singular value thresholding methods of~\cite{Mazumder:2010va} on the \textsf{Netflix} dataset, where they used \textsf{Matlab/PROPACK} on an \textsf{Intel Xeon} 3 GHz processor.
The ``soft impute'' variant uses a constrained rank heuristic in each update step, and an ``un-shrinking'' or fitting heuristic as post-processing. Both are advantages for their method, and were not used for our implementation. Nevertheless, our algorithm seems to perform competitive compared to the reported timings of~\cite{Mazumder:2010va}.

\begin{table}[!h]
\caption{Running times $t_{\text{our}}$ (in hours) of our algorithm on the \textsf{Netflix} dataset compared to the reported timings $t_{\text{M,\tiny{hard}}}$ for ``hard impute'' by~\cite{Mazumder:2009va} and $t_{\text{M,\tiny{soft}}}$ for ``soft impute'' by~\cite{Mazumder:2010va}.
}
\label{tbl:netflix}
\vskip -0.15in
\begin{center}
\begin{small}
\begin{tabular}{|l|r|r|r|r|r|r|}
\hline
$\textit{RMSE}_{test}$ & $t_{\text{M,\tiny{hard}}}\!\!\!$ & $t_{\text{M,\tiny{soft}}}\!\!\!$ & $t_{\text{our}}$ & \textit{k} & \textit{\#mm} & \textit{tr} \\
\hline			%
0.986 &
3.3 & n.a. &
0.144 %
 & 20 & 50 & 99592\\
0.977 &
5.8 & n.a. &
0.306 %
 & 30 & 109 & 99592\\
0.965 &			 %
6.6 & n.a. &
0.504 %
 & 40 & 185 & 99592\\		%
0.962 &			 %
n.a. & 1.36 &
1.08 %
 & 45 & 243 & 174285\\
0.957 &
n.a. & 2.21 &
1.69 %
 & 60 & 416 & 174285\\
0.954 &
n.a. & 2.83 &
2.68 %
 & 80 & 715 & 174285\\
0.9497 &
n.a. & 3.27 &
6.73 %
 & 135 & 1942 & 174285\\
0.9478 &
n.a. & n.a. &
13.6 %
 & 200 & 4165 & 174285\\
\hline
\end{tabular}
\end{small}
\end{center}
\vskip -0.23in
\end{table}

Note that the primary goal of this experimental section is \emph{not} to compete with the prediction quality of the best engineered recommender systems (which are usually ensemble methods, i.e. combinations of many different individual methods). We just demonstrate that our method solves nuclear norm regularized problems of the form~(\ref{eq:constrainedNuc}) on large sample datasets, obtaining strong performance improvements.

\subsection{Conclusion}
We have introduced a new method to solve arbitrary convex problems with a nuclear norm regularization, which is simple to implement and to parallelize.
The method is parameter-free and comes with a convergence guarantee. This guarantee is, to our knowledge, the first guaranteed convergence result for the class of Simon-Funk-type algorithms, as well as the first algorithm with a guarantee for max-norm regularized problems.

It remains to investigate if our algorithm can be applied to other matrix factorization problems such as (potentially only partially observed) low rank approximations to kernel matrices as used e.g. by the PSVM technique~\cite{Chang:2007tq}, regularized versions of latent semantic analysis (LSA), or non-negative matrix factorization~\cite{Wu:2007tm}.

%



\bibliographystyle{alpha} 
\addcontentsline{toc}{section}{Bibliography}
\begin{small}
\bibliography{bibliography}
\end{small}

\end{document}